\documentclass{amsart}

\usepackage{amsmath}
\usepackage{color}

\title[Spectral flow]{On spectral flow for operator algebras}

\author{P. W. Ng}

\address{Department of Mathematics\\
University of Louisiana at Lafayette\\
P. O. Box 43568\\
Lafayette, Louisiana\\
70504-3568\\
USA}

\email{png@louisiana.edu}

\author{Arindam Sutradhar}

\address{Helencha South, Helencha Colony\\
North 24 Parganas\\
West Bengal, India\\
743270}

\email{arindam1050@gmail.com}

\author{Cangyuan Wang}

\address{School of Mathematics\\
East China University of Science and Technology\\
130 Meilong Road\\
Shanghai, 200237\\
P. R. China}

\email{cy.wang@ecust.edu.cn,  cy.wang2020@outlook.com}

\newtheorem{thm}{Theorem}[section]
\newtheorem{lem}[thm]{Lemma}
\newtheorem{prop}[thm]{Proposition}

\newtheorem{df}[thm]{Definition}

\newtheorem{rmk}[thm]{Remark}

\newtheorem{ex}[thm]{Example}

\newcommand{\A}{\mathcal{A}}
\newcommand{\B}{\mathcal{B}}
\newcommand{\C}{\mathcal{C}}
\newcommand{\D}{\mathcal{D}}
\newcommand{\E}{\mathcal{E}}
\newcommand{\W}{\mathcal{W}}
\newcommand{\Z}{\mathcal{Z}}

\newcommand{\F}{\mathcal{F}}

\newcommand{\Ff}{\textit{Fred}}

\newcommand{\J}{\mathcal{J}}  
\newcommand{\K}{\mathcal{K}}

\newcommand{\M}{\mathbb{M}}

\newcommand{\hh}{\mathcal{H}}

\newcommand{\Mul}{\mathcal{M}}

\newcommand{\ProI}{\makebox{\textit{Proj}}_{\infty}}

\newcommand{\Pg}{\mathfrak{P}}

\numberwithin{equation}{section}

\begin{document}

\maketitle  

\section{Introduction}

Spectral flow began as   an integer-valued
homotopy invariant for paths of self-adjoint Fredholm
operators on a separable infinite dimensional Hilbert space $l_2$, which
was first studied by Atiyah and Lusztig (unpublished) 
and first
appeared in print  
in 
\cite{AtiyahPatodiSinger1} and   
\cite{AtiyahPatodiSinger3}.   
Roughly speaking, it counted the \emph{net number} of eigenvalues that changed
sign, i.e., passed through zero in the positive direction,
as one moved along the path of self-adjoint Fredholm operators. 
Among other things, spectral flow is closely related to the eta invariant
term in the Atiyah--Patodi--Singer index theorem (e.g., see 
\cite{AtiyahPatodiSinger1} and \cite{AtiyahPatodiSinger3}).  
In another related direction, the spectral flow of a closed path
of self-adjoint
elliptic operators on a compact manifold can be computed by a 
topological index, giving an interesting special case of the Atiyah--Singer index theorem for families of elliptic operators (\cite{AtiyahPatodiSinger3} 
Theorem 7.4).  We note that this result (\cite{AtiyahPatodiSinger3} Theorem
7.4) is analogous to the original
Atiyah--Singer index theorem (\cite{AtiyahSingerFirst}, \cite{AtiyahSinger1}) 
but with the Fredholm index (or analytic index) of a single elliptic operator
replaced with the spectral flow of a closed path of self-adjoint elliptic
operators, leading to an analogy between Fredholm index and spectral flow.
This analogy   
is one of the themes of the present paper.      

Ever since the early work of Atiyah, Lusztig, Patodi and Singer, 
there has been an enormous literature on spectral flow and its 
applications which
are beyond the scope of this work, and this is true even for the case of
spectral flow for paths of operators over a Hilbert space $l_2$ 
 (e.g., see the ``Glimpse
at the Literature" in the last section of the lecture notes
\cite{Waterstraat}; for another example,  see \cite{Georgescu}). 
Our analytic point of view is that just
as the index of a 
Fredholm operator has had many fruitful and important
generalizations to general operator algebras (beyond $B(l_2)$),  
generalizing the spectral flow of a path of
self-adjoint Fredholm
operators would also be of great interest to operator theory. 
In this paper, we are mainly focused on operator theory,
and developing  interesting (and general) results about spectral flow
for bounded linear operators living in a C*-algebra.  
We focus on the general Hilbert C*-module  case, for which the current 
literature
is considerably thinner.

We now discuss, in a nutshell, some parts of  
the history of the Hilbert C*-module case, with an
emphasis on our point of view -- though, necessarily,
we will briefly move out of our framework and 
mention unbounded operators. Also, before moving forward, we mention that for
the convenience of the reader, a very short summary of the definition of 
spectral flow in \cite{Wu}, for the bounded case,
can be found in the Appendix Subsection
\ref{subsection:DZWLPDefinition}.
 
Dai and Zhang extended the \cite{AtiyahPatodiSinger3} notion of
spectral flow to the setting of a path of families of 
Dirac operators, where the families are
parameterized by a fixed compact manifold $X$ -- i.e., this is a
path of Dirac operators over the standard Hilbert module 
$\hh_{C(X)}$
(\cite{DaiZhang}). 
In their development, Dai and Zhang used the notion of \emph{spectral section}
introduced by Melrose and Piazza \cite{MelrosePiazza}; and Dai and Zhang also
invented the notion of \emph{difference element} or \emph{difference class}
 of two spectral sections
of a given Dirac operator, which is a special 
case of the notion of \emph{essential codimension}, and they used this concept
to define spectral flow for the aforementioned path of Dirac operators 
over $\hh_{C(X)}$.      
Wu generalized the concept of spectral section to the noncommutative case
(e.g., see Definition \ref{df:SpectralCutSection}),
and showed that
 Dai and Zhang's definitions of difference class and spectral flow work
to give a definition of spectral flow for paths of
 self-adjoint regular operators
with compact resolvents over a standard Hilbert module $\hh_{\A}$, 
where
$\A$ is a unital C*-algebra (\cite{Wu}, \cite{LeichtnamPiazzaJFA}, 
\cite{WahlSpectralFlow}).   
(Of course, 
such self-adjoint regular operators with compact resolvents
have to be unbounded, so Wu also gives
a separate exposition for the case of bounded self-adjoint Fredholm operators
in the last section of his preprint.)
Unfortunately, Wu's interesting preprint \cite{Wu} was never published, but,
happily, 
Leichtnam and Piazza (\cite{LeichtnamPiazzaJFA}) gave a nice exposition of a 
significant portion of Wu's work, including the nontrivial
filling in of some missing proofs (e.g., see \cite{LeichtnamPiazzaJFA} page 363,
third paragraph in the proof of Theorem 3).  We note that the paper
\cite{LeichtnamPiazzaJFA} contains many other very interesting results.
Wahl removed the compact resolvent condition in the works of 
\cite{LeichtnamPiazzaJFA} and \cite{Wu}, to give a definition of spectral flow for paths of  self-adjoint
regular Fredholm operators over a standard Hilbert module $\hh_{\A}$, where
$\A$ is a unital C*-algebra (\cite{WahlSpectralFlow}).   
In process, Wahl weakened the definition of spectral section (thus making
it more flexible and easier to work with;  see the
second part of Remark \ref{rmk:SpectralSectionTerminology}), and replaced the notion of 
difference class with the \emph{relative index of a Fredholm pair of projections} which,
again, is a special case of essential codimension. The interesting paper
\cite{WahlSpectralFlow} contains many other interesting items, including 
axiomatizations
for spectral flow as well as for the relative
index of a Fredholm pair of projections.  We also mention that the paper \cite{WahlSpectralFlow} also contains an interesting and important
 alternative definition
of spectral flow whose hypotheses are almost as general as those in our present paper (see \cite{WahlSpectralFlow} Section 4) --  but this second approach, like that of
\cite{PereraPaper} and \cite{PereraThesis}, is abstract and nonconstructive, using  Bott periodicity, and does not capture the original 
concrete intuition of \cite{AtiyahPatodiSinger3}  that spectral flow measures the ``net mass" of the part of the spectrum that passes through zero. 
Our present 
paper uses many ideas of \cite{WahlSpectralFlow}.

Since the present paper is called ``Spectral flow 
for \emph{operator algebras}", we here briefly mention two interesting directions which are not presently part of our main development.  Firstly, in his interesting work, Perera
defines spectral flow for a path of bounded self-adjoint Fredholm operators
in a type $II_{\infty}$ factor $\Mul$ with separable predual
(\cite{PereraPaper}, \cite{PereraThesis})
as well as in a multiplier algebra $\Mul(\A \otimes \K)$ where $\A$ is a unital
C*-algebra. (Recall that $\Mul(\A \otimes \K) \cong \mathbb{B}(\hh_{\A})$,
and $\A \otimes \K$ is called the \emph{canonical ideal}. In fact,
$\hh_{\A} \cong \hh_{\A \otimes \K}$.)  The approach of
Perera is abstract, nonconstructive, using Bott periodicity, and does not contain the original
intuition of \cite{AtiyahPatodiSinger3} where spectral flow measures the 
``net mass" of the spectrum that is crossing zero.  We will nonetheless 
return to the interesting work of Perera in a later part of this paper (e.g.,
see
Section \ref{sect:PreSFIsomorphism}).   
Secondly, John Phillips and his
collaborators have defined spectral flow in the context of semifinite von
Neumann algebras 
(e.g., see \cite{PhillipsVictoria}, \cite{BenPhillEtAl} and
\cite{Georgescu}).  This version is ``analytic" and does indeed capture the
original intuition of \cite{AtiyahPatodiSinger3}, and while this is a very
interesting topic, we will not be returning to it substantially in the present paper, though
we were certainly
 inspired by reading the interesting papers of Phillips and his coauthors.

We now return to the discussion of the Hilbert module case.
We note that the present theory of spectral flow in the Hilbert module
case (\cite{DaiZhang}, \cite{Wu}, \cite{LeichtnamPiazzaJFA}, 
\cite{WahlSpectralFlow})  
requires that the canonical ideal have an 
approximate unit consisting of a sequence of projections. In fact,
the definition of difference class  itself already necessitates 
that the canonical ideal have a nonzero projection which in itself
is already quite nontrivial.  (Similar for the relative index for a Fredholm
pair of projections;  see \cite{WahlSpectralFlow} 3.2;  see also 
the beginning of Subsection \ref{subsection:DZWLPWDefinition} for more issues.)   From the perspective
of modern C*-algebra theory, this is a very strong assumption.  
For example, in the
Elliott program for classifying simple nuclear C*-algebras, many
interesting simple stably    
projectionless C*-algebras which exhaust the K theory invariant have
been constructed (\cite{GongLinRange}). (Recall that a
 C*-algebra $\D$ is \emph{stably projectionless} if the
only projection in the stabilization $\D \otimes \K$ is zero.) 

In still another direction, previous 
definitions of spectral flow actually require 
the existence of spectral sections, which is a very 
strong assumption that is shown, under various conditions, to
be  equivalent to the
vanishing of a certain K theory index (e.g., see \cite{MelrosePiazza} 
Proposition 1, \cite{DaiZhang} Proposition 1.3, \cite{Wu} Theorem 2.2, and 
\cite{LeichtnamPiazzaJFA} Theorem 3). 
We will essentially
use a weakening of the notion of spectral section (the so-called
``generalized spectral sections") introduced in \cite{WahlSpectralFlow} 
Definition 3.4 (see Remark \ref{rmk:SpectralSectionTerminology} (2)).    
It turns out that the existence of generalized 
spectral sections is equivalent to a projection lifting problem (Theorem
\ref{thm:ProjectionLiftingCondition} (1) $\Leftrightarrow$ (3)), an  
important problem in operator theory, which also   
gives a condition that is both very strong and also not so easy to directly
check.  In this paper, instead of talking about ``the existence of spectral
sections", we prefer to talk about ``lifting projections", because the
latter terminology 
is more concrete and familar to analysts.
Projection lifting in multiplier algebras is connected to many fundamental
and interesting
problems in operator theory, including the Brown--Pedersen--Zhang conjectures,
the Weyl--von Neumann theorem, real rank zero, K theory conditions, and  
more (e.g.,
see Subsection \ref{subsect:LiftingProjectionsCondition}, including
Remark \ref{rmk:LiftingProjections}). As we will see (e.g., Theorem 
\ref{thm:ProjectionLiftingCondition} the last paragraph), 
perhaps one reason for the smoothness
and success of the theory of spectral flow in $\mathbb{B}(l_2)$ and more
general semifinite von Neumann algebras (e.g.,
see \cite{BenPhillEtAl}) is that von Neumann algebras have real rank zero.
Projection lifting implies the vanishing index condition, but the converse is
not true  
(see Theorem \ref{thm:ProjectionLiftingCondition} and Example
\ref{ex:VanishingK1NoProjectionLift}).  
In the earlier works, existence of spectral sections was shown to be equivalent
to the vanishing index condition, only because there were extra assumptions on 
the relevant operators (e.g., see \cite{LeichtnamPiazzaJFA} Theorem 2;
also, real rank zero multiplier algebras
will also guarantee the equivalence -- see
Theorem \ref{thm:ProjectionLiftingCondition} the last paragraph).

In this paper, our first goal is to provide a clean presentation of the 
theory of spectral flow for (norm-) continuous paths of self-adjoint Fredholm
operators, in a multiplier algebra, with invertible endpoints.  Our approach does not require that
the canonical ideal have any projections other than zero.  
Our most general definition of spectral flow does not \emph{explicitly}
require the projection lifting hypothesis,
but we require that the
endpoints of the path be invertible operators.  Nonetheless,
the hypotheses are relatively simple and are easy to check and easy to
work with.   A simple homotopy projection-lifting result will then imply 
the needed projection-lifting (see Lemma \ref{lem:HomotopyProjectionLift}).
This result follows from the Homotopy Lifting Theorem of \cite{BlackadarHomotopyLifting}, but 
for the convenience of the reader, we provide the short, easy proof for our special
case (Lemma \ref{lem:HomotopyProjectionLift}).  This also again illustrates how many difficulties are 
ultimately operator theoretic in nature.
All this is the content of section 2. (See Definition \ref{df:SpectralFlow}
for the main definition.)  
Our definition of spectral flow uses a quite general and
 independent definition of essential
codimension which does not require that the canonical ideal have a nonzero
projection (e.g., see the beginning of Subsection
\ref{subsection:DZWLPWDefinition}, Definition \ref{df:EssentialCodimension},
and Proposition \ref{prop:ECProperties}). 
Essential codimension and its properties are extensively discussed in 
\cite{LoreauxNgSutradharEC}.  However, 
 since this paper is not yet published, we use an equivalent KK definition of essential codimension, and the basic 
properties of essential codimension follow immediately from KK theory.  However,
the reader will not be expected to have extensive knowledge of KK theory, and
we provide the definition of the generalized homomorphism picture of KK at
the beginning of Subsection \ref{subsection:DZWLPWDefinition}. 
(See also the end of this section for the prerequisites for reading this
paper.)

We mention that for later substantial results, we introduce stronger hypotheses than what is mentioned above
(e.g., in the Spectral Flow Isomorphism Theorem of 
Subsection \ref{subsect:SFIsomorphism} 
(see Theorem \ref{thm:GeneralSFIsomorphism}) and  in the axiomatization of spectral flow
in Subsection \ref{subsect:Axiomatization} (see Definition 
\ref{df:Pg})).
Nonethess, these stronger hypotheses are still clean and simple, 
and are relatively easy to check and easy to work with.

In Section \ref{sec:FunctorialAxioms}, we prove functorial properties for spectral flow.              
Later on, in Subsection \ref{subsect:Axiomatization}, we will prove that some of these properties give
an axiomatization of spectral flow, under appropriate hypotheses (as mentioned
in previous paragraph).

In Section \ref{sect:PreSFIsomorphism} and Subsection
\ref{subsect:SFIsomorphism}, we prove the Spectral Flow Isomorphism Theorem,
which says that, under appropriate hypotheses,
 spectral flow induces a group isomorphism
$$\pi_1(\Ff_{SA, \infty}) \cong K_0(\B).$$  This is one of the main results
of this paper, and it generalizes a result of \cite{AtiyahPatodiSinger3} (see
\cite{AtiyahPatodiSinger3} (7.3);  see also, \cite{Phillips1996} the last
theorem, which is stated on page 464, and see also
\cite{PhillipsVictoria} Theorem 2.9 for the von Neumann factor case).   
In the process of proving the Spectral Flow Isomorphism Theorem, 
we also show that for every invertible $X \in \Ff_{SA, \infty}(\Mul(\B))$, 
$\Omega_X \Ff_{SA, \infty} (\Mul(\B))$ is a classifying space 
for the functor $K \mapsto 
K_0(C(K) \otimes \B)$,  
generalizing a result from \cite{AtiyahSingerSkew} and which is in itself
of interest.
We use many ideas from the type $II_{\infty}$ factor case of
\cite{PereraThesis} and \cite{PereraPaper}, with substantial modifications  (see this paper
Section
\ref{sect:PreSFIsomorphism}; see also \cite{AtiyahPatodiSinger3} Section 3
after page 81 and Section 7 page 94, \cite{AtiyahSingerSkew},
and \cite{Phillips1996} the proof of the last theorem for some earlier sources).
 
In Subsection \ref{subsect:Axiomatization},
we use the Spectral Flow Isomorphism Theorem to provide an axiomatization of
spectral flow, under appropriate hypotheses.
Our axiomatization follows the strategy of the proof
of \cite{LeschUniqueness} Theorem 5.4,
which is a ``folklore result" for the case of $\mathbb{B}(l_2)$.

In the last section, which is the Appendix, we provide some helpful 
results for the convenience of the reader.  We provide some miscellaneous
computations in operator theory, which are not as well known to beginning
operator theory students as they should be.  We provide some homotopy
results which are used in this paper, and which go beyond a first course,
but for which we cannot find a good reference. For the convenience of 
the reader, we also provide a quick summary of the approach to spectral flow
in \cite{Wu}.   We also fix some notation for Fredholm operators and Fredholm index in a multiplier algebra.

As mentioned previously, in this paper, we are mainly focused on operator
theory and bounded linear operators living in a C*-algebra. 
This paper is directed towards operator theory students with a basic background
that can be found in, for example, \cite{DavidsonBook}, \cite{WeggeOlsen}
or the parts of \cite{LinBook} that do not involve KK theory.  KK theory
is only used in several places in this paper.  There is a substantial amount of
homotopy theory, but it is mostly at the level of a first year graduate course in
most North American universities.
We provide full references to the standard literature for both KK and homotopy theory,
and for the one case where the homotopy theory goes beyond a first course, and for
which we could not find a good reference, we write out all the details  (e.g.,
see Subsection \ref{subsect:KappaIsHomotopyEquivalence}).   
Thus, a beginning student of analysis, with the above prerequisites, who is willing 
to read some definitions and accept some statements from the standard literature, should
be able to smoothly read this paper.\\
    
The authors thank Gabor Szabo for pointing out the result \cite{BlackadarHomotopyLifting}, which led to a considerable simplification
of the general definition of spectral flow.\\

\noindent {\bf  Acknowledgements  }\\

Part of this research was done while P. W. Ng and C. Wang were visiting the Research Center for Operator
Algebras at East China Normal University, Shanghai, China.  P. W. Ng was supported by funds from the Research Center for
Operator Algebras as well as a travel grant from the University of Louisiana at Lafayette.  C. Wang was partially supported
by the National Natural Science Foundation of China (grant number 12201216).\\

\section{The definition of spectral flow}

The most general definition of spectral flow, presented in this paper, is
in Definition \ref{df:SpectralFlow}.  It is defined for an arbitrary
norm-continuous path of self-adjoint Fredholm operators, with
invertible endpoints, in a multiplier
algebra $\Mul(\B)$ of an arbitrary separable stable C*-algebra $\B$ (see Subsection \ref{subsect:GeneralizedFredholm}).
In particular, the C*-algebra $\B$ need not have nonzero projections, and
the hypotheses are easy to check.\\  

\subsection{The DZWLPW definition} \label{subsection:DZWLPWDefinition} 
This subsection presents what is essentially a slight modification
of the definition of spectral flow in
\cite{WahlSpectralFlow}, but specialized to the bounded case, and
with the more general notion of essential 
codimension replacing the ``relative index of a Fredholm pair of projections"
in \cite{WahlSpectralFlow}, and without the need for the canonical ideal to 
have a nonzero projection  (see Definition 
\ref{df:WahlSFDefinition}).  For the benefit of the reader, we provide 
complete proofs.  In the title of this subsection, ``DZWLPW" abbreviates
``Dai--Zhang--Wu--Leichtnam--Piazza--Wahl".    

We begin by introducing  the definition of essential codimension, which generalizes the codimension of a subprojection of a bigger projection.  This concept was introduced in the groundbreaking paper of Brown--Douglas--Fillmore while proving functorial properties of their homology $Ext(X)$ (\cite{BDF}).
This notion has since found many applications, including in the theory of spectral flow (e.g., \cite{BenPhillEtAl}), giving K-theory characterizations of projection lifting for certain corona algebras (e.g., \cite{LeeProjLift}, 
\cite{BrownLee}),
as well as explaining the mysterious integers appearing in Kadison's characterization of the diagonal of a projection  (e,g., see 
\cite{KadisonPyth}, \cite{KaftalLoreaux}), and other places.

One of the weaknesses of the previous theories (while they were quite interesting and important)
 was that they used special cases of essential  codimension, which required
that the  canonical ideal have an approximate unit consisting of projections.  This includes the \emph{difference class}
of \cite{DaiZhang}, \cite{Wu} and \cite{LeichtnamPiazzaJFA} as well as the \emph{relative index of a Fredholm pair of
projections} of \cite{WahlSpectralFlow}.  Moreover, the proof of the axiomatization of the relative index appeals to spectral flow, which in turn
is defined using the relative index  (see \cite{WahlSpectralFlow} Proposition 3.11).  All this is remedied with the notion of essential codimension
of a Fredholm pair of projections, where the canonical ideal is an arbitrary separable stable C*-algebra (\cite{LoreauxNgSutradharEC}).
Among other things, the paper \cite{LoreauxNgSutradharEC} presents an axiomatization of essential codimension whose proof
is elementary.

We reemphasize
 that one of our improvements on the theory of spectral flow uses the fact that the notion  of essential codimension is a very general index for 
a Fredholm pair of projection (as opposed to the special cases used by 
\cite{DaiZhang}, \cite{Wu}, \cite{LeichtnamPiazzaJFA}, \cite{WahlSpectralFlow};
e.g., see Subsection \ref{subsection:DZWLPDefinition}) and 
does not require that the canonical ideal has an approximate unit consisting of projections. In fact,
essential codimension does not even require that the canonical ideal has
one nonzero projection.   

While there are concrete definitions of essential codimension (e.g.,
see \cite{LoreauxNgSutradharEC}), since the paper \cite{LoreauxNgSutradharEC}
has not yet been published, following \cite{LeeProjLift} (see 
also \cite{LoreauxNgSutradhar}), 
we introduce the (equivalent) KK theory version, and properties of 
essential codimension will immediately follow from properties of the 
KK functor.  Nonetheless, this paper is meant to be elementary without 
many technical prerequisites.  Hence, 
we introduce the KK theory version in a way so that those who are not 
familiar with KK theory will still be able to understand and work 
through this paper, as long as they read through the definitions and 
are willing to accept a few statements without proof (for which we will 
provide full references from the standard literature).  
In fact, the notion of essential codimension is a basic example of and 
an excellent introduction to $KK^0$ (e.g., \cite{LoreauxNgSutradhar},
\cite{LoreauxNgSutradharEC}).   To relate the KK version to more concrete, 
intuitive versions, please see \cite{LoreauxNgSutradharEC}.

Recall that 
for a nonunital C*-algebra $\B$, $\Mul(\B)$ denotes the multiplier 
algebra of $\B$ and $\C(\B) =_{df} \Mul(\B)/\B$ denotes the corona algebra
of $\B$.  We also let $\pi_{\B} : \Mul(\B) \rightarrow \C(\B)$ denote the usual
quotient map.  Often, we drop the subscript ``$\B$" 
and write ``$\pi$" instead of ``$\pi_{\B}$".

To introduce the definition of essential codimension, we need  to use the generalized homomorphism picture of KK theory, for which essential codimension will be a concrete example.   For more information, please see
\cite{JensenThomsenBook} Chapter 4 (see also \cite{BlackadarBook}).\\

Let $\A,\B$ be C*-algebras with $\A$ separable and $\B$ $\sigma$-unital and stable. A 
\emph{$KK_h(\A,\B)$-cycle} 
is a pair $(\phi,\psi)$ of *-homomorphisms $\phi,\psi:\A \to \Mul(\B)$ such 
that $\phi(a)-\psi(a) \in \B$ for all $a \in \A$. Recall that 
two $KK_h(\A,\B)$-cycles $(\phi_0,\psi_0)$ and $(\phi_1,\psi_1)$ are \emph{homotopic} 
if there exists a path $\{(\phi_t ,\psi_t)\}_{t \in [0,1]}$ of $KK_h(\A,\B)$-cycles such that for all $a \in \A$,

\begin{enumerate}
\item the maps $[0,1] \rightarrow \Mul(\B)$ given by $t \mapsto \phi_t(a)$ 
and $t \mapsto \psi_t(a)$ are both strictly continuous, and  
\item the map $[0,1] \to \B: t \mapsto \phi_t(a)-\psi_t(a)$ 
is norm continuous.
\end{enumerate}

For any $KK_h(\A,\B)$-cycle $(\phi,\psi)$, we let $[\phi, \psi]$ or $[\phi, \psi]_{KK}$
 denote its homotopy class. Then we may define $KK^0(\A,\B)$ or $KK(\A,\B)$ by
\begin{equation} \label{equdf:KK_h}
KK(\A,\B)=_{df} \{[\phi,\psi]:(\phi,\psi) \text{ is a } KK_h(\A,\B)-\text{cycle}.\}
\end{equation} 

With an appropriate addition operation, $KK(\A,\B)$ is an abelian group.  See 
\cite{JensenThomsenBook} Chapter 4 for more details. \\

\begin{df}  \label{df:EssentialCodimension}
	Let $\B$ be a separable, stable C*-algebra and $P,Q \in \Mul(\B)$ be projections such that $P-Q \in \B$.
	
	The (KK) essential codimension of $Q$ in $P$ is defined as
	\[[P:Q] =_{df} [\phi,\psi]_{KK} \in KK^0(\mathbb{C},\B) \cong K_0(B)\]
	where $\phi,\psi: \mathbb{C} \to \Mul(\B)$ are the *-homomorphisms 
given by $\phi(1)=P$ and $\psi(1)=Q$.\\  
\end{df}

For the convenience of the reader, we here now list some basic properties of essential codimension, all of which immediately follow from the properties of KK
(see \cite{JensenThomsenBook}).

\begin{prop}  \label{prop:ECProperties}
Let $\B$ be a separable stable C*-algebra.
Let $P, Q, R, P', Q' \in \Mul(\B)$ be projections such that $P - Q, Q - R, P' - Q' \in \B$.
Then the following statements are true:
\begin{enumerate}
\item (EC1) Let $\D$ be a separable stable C*-algebra and $\phi : \Mul(\B) \rightarrow \Mul(\D)$ a strictly continuous 
*-homomorphism for which $\phi(\B) \subseteq \D$.

Then $[\phi(P): \phi(Q)] = \phi_*([P:Q])$ in $K_0(\D)$.
\item (EC2) Suppose that there exists a path $\{ (P_t, Q_t) \}_{t \in [0,1]}$ where
$P_t, Q_t \in \Mul(\B)$ are projections and $P_t - Q_t \in \B$ for all $t \in [0,1]$,
the two maps $[0,1] \rightarrow \Mul(\B)$ given by $t \mapsto P_t$ and $t \mapsto Q_t$ are strictly continuous, and 
the map $[0,1] \rightarrow \B : t \mapsto P_t - Q_t$ is norm-continuous.

Then $[P_0: Q_0] = [P_1 :Q_1]$.
\item (EC3)  If $P, Q \in \B$ then $[P:Q] = [P] - [Q] \in K_0(\B)$.
\item (EC4) $[P:P]=0$.
\item (EC5)  $[P:Q] = -[Q:P]$.
\item (EC6) If $P \perp P'$ and $Q \perp Q'$ then
$[P + P' : Q + Q'] = [P:Q] + [P': Q']$.
\item (EC7) If $S \in \Mul(\B)$ is an isometry then $[P:Q] = [SPS^* : SQS^*]$.
\item (EC8) $[P:Q] + [Q:R] = [P:R]$.
\end{enumerate}
\end{prop}

\begin{proof}  These properties follow immediately from the properties of KK 
(see \cite{JensenThomsenBook}).  
See also, for example, \cite{Lee13} Lemma 2.3 and \cite{LoreauxNgSutradharEC}. 
\end{proof}

We note that some of the functorial properties listed in Proposition \ref{prop:ECProperties} actually give an axiomatization of essential codimension.  For this and other interesting aspects, please see 
\cite{LoreauxNgSutradharEC}.\\

We begin with the definition (Definition \ref{df:WahlSFDefinition})
 of spectral flow which is essentially a slight modification
of the definition from  \cite{WahlSpectralFlow}
(which was inspired by the earlier works of \cite{DaiZhang}, \cite{Wu}
and 
\cite{LeichtnamPiazzaJFA}), for the bounded case, and we call this definition
the \emph{DZWLPW definition} of spectral flow.  
This will motivate the
simplifying results and more general approach
that we will introduce afterwards. 
We note that one way that the DZWLPW definition (Definition
\ref{df:WahlSFDefinition}) differs from the original approach in
\cite{WahlSpectralFlow} is that in Definition \ref{df:WahlSFDefinition}, 
the notion of essential codimension (which is more general then the ``relative
index of a Fredholm pair of projections" in \cite{WahlSpectralFlow})
is used, and  
Definition \ref{df:WahlSFDefinition} does not require that the canonical ideal has a nonzero projection (unlike the definition in \cite{WahlSpectralFlow}).\\

We let $1_{\geq 0} : \mathbb{R} \rightarrow [0,1]$ denote the
\emph{indicator function} on $[0,\infty)$, i.e., for all $t \in \mathbb{R}$, 
\begin{equation}  \label{equdf:IndicatorFunction}
1_{\geq 0}(t) =_{df} 
\begin{cases} 
	1 & t \in [0, \infty) \\
	0 & t \in (-\infty, 0)
\end{cases}  
\end{equation}
Recall that for an invertible element of a unital C*-algebra, its spectrum is bounded
away from $0$.  Hence, for an invertible self-adjoint element $x$ of a unital
C*-algebra, $1_{\geq 0}(x)$ is the support projection of the positive part
of $x$.\\

We provide a preliminary lemma and some preliminary definitions first, before
getting to the DZWLPW Definition (\ref{df:WahlSFDefinition}).

\begin{lem}
	Let $\B$ be a separable stable C*-algebra, and let
	$A_1, A_2 \in \Mul(\B)$ be invertible self-adjoint elements such that
	$A_1 - A_2 \in \B$.
	
	Then $1_{\geq 0} (A_1) - 1_{\geq 0}(A_2) \in \B$.
	\label{lem:FredholmA1A2} 
\end{lem}

\begin{proof}  Since $A_1 - A_2 \in \B$, $\pi(A_1)=\pi(A_2)$.
	Since $A_1, A_2, \pi(A_1), \pi(A_2)$ are all invertible,
all of $1_{\geq 0}(A_1), 1_{\geq 0}(A_2), 1_{\geq 0}(\pi(A_1)), 1_{\geq 0}(\pi(A_2))$ are 
defined.  From the above, and the continuous functional calculus,
$$\pi(1_{\geq 0}(A_1))= 1_{\geq 0}(\pi(A_1)) = 1_{\geq 0}(\pi(A_2)) = \pi(1_{\geq 0}(A_2)).$$
	Therefore, $1_{\geq 0}(A_1)-1_{\geq 0}(A_2) \in \B$.\\ 
\end{proof}

Recall that for a nonunital C*-algebra $B$, we say that an element $X \in \Mul(\B)$
is \emph{Fredholm} if $\pi_{\B}(X)$ is invertible in $\C(\B)$.  Often,
we also call $\pi_{\B}(X)$ a Fredholm operator.  See the Appendix Subsection
\ref{subsect:GeneralizedFredholm} for a short summary of the relevant details
including the definiton of  
\emph{generalized Fredholm index}, which we will use in this paper.

\begin{df}  \label{df:WahlSFPrelim}
	Let $\B$ be a separable stable C*-algebra.
	\begin{enumerate}
		\item For any self-adjoint (necessarily Fredholm) element $A \in \Mul(\B)$,
		a \emph{trivializing operator} for $A$ (if it exists) is a self-adjoint element $b \in \B$
		such that $A + b$ is invertible. 
		\item Suppose that $A \in \Mul(\B)$ is self-adjoint and $b_1, b_2 \in \B_{SA}$
		are trivializing operators for $A$. We denote
		$$Ind(A, b_1, b_2) =_{df} [1_{\geq 0}(A + b_1): 1_{\geq 0}(A + b_2)] \in K_0(\B).$$  
		(See Lemma \ref{lem:FredholmA1A2}.  We also remind the reader that $[:]$ is our notation for essential codimension.)  
		\item Let $I$ be a compact metric space, and
		let $\{ A_t \}_{t \in I}$ be a norm-continuous family of self-adjoint
		(necessarily Fredholm) operators in $\Mul(\B)$. Let $\{ b_t \}_{t \in I}$
		be a norm-continuous family of self-adjoint elements of $\B$.
		We say that $\{ b_t \}_{t \in I}$ is a \emph{trivializing family} for 
		$\{ A_t \}_{t \in I}$ if viewing $\{ A_t \}_{t \in I}$ as an element of
		$\Mul(C(I) \otimes \B)$ and viewing $\{ b_t \}_{t \in I}$ as an element
		of $C(I) \otimes \B$, $\{ b_t \}_{t \in I}$ is a trivializing operator
		for $\{ A_t \}_{t \in I}$.  
		\item Let $I$ be a compact metric space, let $\{ A_t \}_{t \in I}$ be a norm-continuous family of self-adjoint
		(necessarily Fredholm) operators in $\Mul(\B)$. We say that there are
		\emph{local trivializing families} for $\{ A_t \}_{t \in I}$ if for every
		$t_0 \in I$, there exists a compact neighbourhood $U \ni t_0$ such that
		$\{ A_t \}_{t \in U}$ has a trivializing family.  
	\end{enumerate}
\end{df}

\begin{rmk} \label{rmk:SpectralSectionTerminology}
\begin{enumerate}
\item[(a)]  
 For a self-adjoint Fredholm operator in a multiplier algebra, a trivializing operator need not exist.  
In fact, the existence of trivializing operators is a strong, 
restrictive condition which is related to fundamental and interesting 
questions in operator theory, among them a projection lifting
condition. As a consequence, trivializing families and 
local trivializing families also need not exist. 
See Subsection \ref{subsect:LiftingProjectionsCondition} for more discussion
of this issue.  
\item[(b)] Let $\B$ be a separable stable C*-algebra.  
If $A \in \Mul(\B)_{SA}$ and $b \in \B_{SA}$ is a trivializing operator
for $A$ (i.e., $A + b$ is invertible; i.e., we are in the context of 
Definition \ref{df:WahlSFPrelim} (1)), then \cite{WahlSpectralFlow} Definition
3.4 calls the projection $1_{\geq 0}(A + b)$ a \emph{generalized spectral 
section} or \emph{spectral section}.  This generalizes the definition of
spectral section used in \cite{MelrosePiazza}, \cite{DaiZhang}, \cite{Wu} 
and \cite{LeichtnamPiazzaJFA} (see Definition \ref{df:SpectralCutSection}). 
Existence of spectral sections is a crucial
assumption in previous treatments of spectral flow.
 However, we will not use this terminology
much, preferring instead to discuss projection lifting, which is more concrete
operator theoretic terminology that is more familiar to analysts (both the
terminology and the phenomena).  
See Subsecton \ref{subsect:LiftingProjectionsCondition}. 
\end{enumerate}
\end{rmk}

We now give the DZWLPW definition of spectral flow, which has, as hypothesis,
the strong 
assumption of the existence of local trivializing families:

\begin{df}  \label{df:WahlSFDefinition}
Let $I =_{df} [0,1]$. 
Let $\B$ be a separable stable C*-algebra, and suppose that 
$\{ A_t \}_{t \in I}$ is a norm-continuous path of  self-adjoint Fredholm
operators in $\Mul(\B)$ for which local trivializing families exist.  
Let 
$$0 = t_0 < t_1 < ... < t_n = 1$$
be a partition of $I$ such that for all $0 \leq j \leq n-1$,
$\{ b^j_t \}_{t \in [t_j, t_{j+1}]}$ is a trivializing family for
$\{ A_t \}_{t \in [t_j, t_{j+1}]}$.  
For $k = 0,1$, let $b'_k \in \B_{SA}$ be a trivializing operator
for $A_k$.

We define the \emph{spectral flow}
\begin{equation}\label{equ:WahlSfDefinition}  Sf(\{ A_t \}_{t \in I}, b'_0, b'_1) =_{df} 
Ind(A_0, b'_0, b^0_0) + Ind(A_1, b^{n-1}_1, b'_1) + \sum_{j=1}^{n-1}
Ind(A_{t_j}, b^{j-1}_{t_j}, b^j_{t_j}) \end{equation}    
\noindent which is an element of $K_0(\B)$. 
\end{df}

\begin{rmk} \label{rmk:SfIntuitionAndOrientation1}
Firstly, given the definition of essential codimension, we see that 
Definition \ref{df:WahlSFDefinition} roughly measures the 
``net mass"  of the part of the 
spectrum which passes through zero in the negative direction, as one moves along the continuous path of self-adjoint Fredholm operators.  
(See, for example,  the intuitive discussion of the commuting
case for a Fredholm pair of projections in
 \cite{BenPhillEtAl} two paragraphs before Theorem 2.1.)   This captures the original intuition of \cite{AtiyahPatodiSinger3} for spectral flow.  
Note though that in the original \cite{AtiyahPatodiSinger3} 
spectral flow, as well as many subsequent versions of spectral flow
(e.g., see \cite{BenPhillEtAl} and the references therein), the spectral
flow measures the ``net mass" of the part of the
 spectrum which passes through zero in the
\emph{positive direction}.  Thus our version of spectral flow
has \emph{opposite orientation} to that of Atiyah--Patodi--Singer
and many subsequent versions of spectral flow. But this is just a sign change,
so it is not so important.      

Also, it is not hard to see that in Definition \ref{df:WahlSFDefinition}, 
the compact interval $[0,1]$ can be replaced by any compact subinterval 
of the real line,  and this will be true for all the results in this paper.  
However, for simplicity, thoughout this paper, we will mostly (though not 
always) stick with the interval $[0,1]$.

Next, we will prove that Definition \ref{df:WahlSFDefinition} is well-defined,
i.e., independent of choice of local trivializing families and partition,
in Proposition \ref{prop:WahlSFWellDefined}.

Finally, we will later on show that in Definition 
\ref{df:WahlSFDefinition}, \textbf{we do not need to partition the interval
$I$}, i.e., under the hypotheses of Definition \ref{df:WahlSFDefinition}
including 
that local trivializing families exist, \textbf{we will be able to find a
global trivializing family for $\{ A_t \}_{t \in I}$.}  
(See Proposition \ref{prop:WahlSFSingleCell}.)
  However, in practice, it will often be easier to work
with a partition that naturally arises from a given norm-continuous path
of self-adjoint Fredholm operators.  Also, our axiomatization of spectral
flow will eventually need the Axiom of Concatenation (which increases
cells in a partition). Thus, we prefer to keep the flexibility
afforded by a general partition.\\   
\end{rmk}

To prove that Definition \ref{df:WahlSFDefinition} is well-defined, we need some additional technical results, which are short exercises in basic operator theory:

\begin{lem}
Let $\A$ be a unital C*-algebra and $x \in \A$ a self-adjoint invertible. 
Then for every $\epsilon > 0$, we can choose $\delta >0$ so that the following statement is true:

If $y \in \A$ is a self-adjoint element and $\| x - y \| < \delta$ then $y$ is invertible and
$$\| 1_{\geq 0}(x) - 1_{\geq 0}(y) \| < \epsilon.$$ 

\label{lem:Aug520232am}
\end{lem}

\begin{proof}[Sketch of proof]  
Since $x$ is a self-adjoint invertible, $dist(sp(x), 0) > 0$.   Hence, by a standard spectral theory argument,
choose $\delta_1 > 0$
such that if $y \in \A$ is self-adjoint and $\| x - y \| < \delta_1$ then 
$$dist(sp(y), 0) > \frac{1}{2} dist(sp(x), 0);$$  
in particular, such a $y$ is invertible.   Let $r =_{df} \frac{1}{2} dist(sp(x), 0)$.  

By the Stone--Weierstrass theorem, choose a polynomial $p$ so that for all $t \in [- \| x \| - \delta_1, -r] \cup [r, \| x \| + \delta_1]$,
$$|p(t) - 1_{\geq 0}(t) | < \frac{\epsilon}{10}.$$

Choose $\delta > 0$ such that $\delta < \delta_1$ and 
whenever $z, z' \in \A$ are elements with  $\| z \|, \| z' \| < \| x \| + 
1 + \delta_1$
and  
$\| z - z' \| < \delta$ then
$$\| p(z) - p(z') \| < \frac{\epsilon}{10}.$$

Hence,
for all self-adjoint $y \in \A$ with $\| x - y \| < \delta$, 
$y$ is invertible (since $\delta < \delta_1$) and
\begin{eqnarray*}
\| 1_{\geq 0}(x) - 1_{\geq 0}(y) \| & \leq & \| 1_{\geq 0}(x) - p(x) \| + \| p(x) - p(y) \| + \| p(y) - 1_{\geq 0}(y) \| \\
& <  & \frac{\epsilon}{10} + \frac{\epsilon}{10} + \frac{\epsilon}{10} < \epsilon.\\ 
\end{eqnarray*}

\end{proof}

\begin{lem}  Let $\A$ be a unital C*-algebra, let $\{ x_n \}$ be a sequence of self-adjoint invertible elements of 
$\A$ and let $x \in \A$ be invertible.

If $x_n \rightarrow x$, then $x$ is a self-adjoint invertible element of $\A$ and
$$1_{\geq 0}(x_n)  \rightarrow 1_{\geq 0}(x).$$
\label{lem:Aug520231am}
\end{lem}

\begin{proof}
That $x$ is self-adjoint is a standard argument.  
That $1_{\geq 0}(x_n)  \rightarrow 1_{\geq 0}(x)$ follows from 
Lemma \ref{lem:Aug520232am}.\\  
\end{proof}

\begin{lem}
Let $I \subseteq \mathbb{R}$ be a compact interval.
Let $\A$ be a unital C*-algebra, and let
$\{ a_t \}_{t \in I}$  
be a norm-continuous path of self-adjoint invertible elements of $\A$.

Then $\{ 1_{\geq 0}(a_t) \}_{t \in I}$ is a norm-continuous path of
projections in $\A$.
\label{lem:1geq0Continuous}
\end{lem}

\begin{proof}
This follows immediately from Lemma \ref{lem:Aug520231am}\\ 
\end{proof}

We now prove that spectral flow (as in Definition \ref{df:WahlSFDefinition})
 is well-defined.  

\begin{prop} \label{prop:WahlSFWellDefined}
	The spectral flow (i.e., Definition \ref{df:WahlSFDefinition}), if it exists, is well-defined.
	I.e., it is   
	independent of the choice of trivializing families as well as partition. 
\end{prop}

\begin{proof}  Let $\B$ be a separable stable C*-algebra, and suppose that $\{ A_t \}_{t \in [0,1]}$ is a norm-continuous path of self-adjoint Fredholm operators in $\Mul(\B)$ for which local trivializing families exist.  
Fix self-adjoint elements $b'_0, b'_1 \in \B$ so that $b'_j$ is a trivializing operator for $A_j$ for $j=0,1$. 
We want to show that $Sf(\{  A_t \}_{t \in [0,1]}, b'_0, b'_1)$ is well-defined. 

As a first step, let us show  that for a given partition of $[0,1]$ and for any two local 
trivializing families for $\{ A_t \}$ with that partition (if they exist), and with the same 
endpoint trivializing operators $b'_0, b'_1$, the recipe for spectral flow in 
Definition \ref{df:WahlSFDefinition} will give the same value in $K_0(\B)$.   
So suppose that $0 = t_0 < t_1 < ... < t_n = 1$ is a partition for $[0,1]$.
Suppose that for all $0 \leq j \leq n-1$, $\{ b^j_t \}_{t \in [t_j, t_{j+1}]}$ and $\{ c^j_t \}_{t \in [t_j, t_{j+1}]}$ are two trivializing families for $\{ A_t \}_{t \in [t_j, t_{j+1}]}$.   
We want to show that the two trivializing families give the same value for the spectral flow.  I.e., from (\ref{equ:WahlSfDefinition}), we want to prove 
that\\

\begin{equation}\label{equ:Nov120231AM}  \end{equation}
\begin{eqnarray*}
& & Ind(A_0, b'_0, b^0_0) +  Ind(A_1, b^{n-1}_1, b'_1) + \sum_{j=1}^{n-1} Ind(A_{t_j}, b^{j-1}_{t_j}, b^j_{t_j})\\
& = & Ind(A_0, b'_0, c^0_0) +  Ind(A_1, c^{n-1}_1, b'_1) + \sum_{j=1}^{n-1} Ind(A_{t_j}, c^{j-1}_{t_j}, c^j_{t_j}) \makebox{  in   } K_0(\B).
\end{eqnarray*}

So by the definition of $Ind$ in part (2) of Definition \ref{df:WahlSFPrelim} 
(which is defined in terms of essential codimension $[:]$), we want to 
prove that $Sf_1 = Sf_2$ where
$$Sf_1  =_{df}  
[P'_0 : P^0_0] + [P^{n-1}_1: P'_1] + 
\sum_{j=1}^{n-1} [P^{j-1}_{t_j} : P_{t_j}^j]$$ 
and 
$$Sf_2 =_{df} [P'_0 : Q^0_0] + [Q^{n-1}_1: Q'_1] +
\sum_{j=1}^{n-1} [Q^{j-1}_{t_j} : Q_{t_j}^j].$$
In the above (following equation (\ref{equ:Nov120231AM}) and
 Definition \ref{df:WahlSFPrelim} item (2)),
$P'_0, P'_1, P^j_t, Q^j_t$ are projections in $\Mul(\B)$ such that  
$P'_0 =  1_{\geq 0}(A_0 + b'_0)$,  $P'_1 = 1_{\geq 0}(A_1 + b'_1)$,
$P^j_t = 1_{\geq 0}(A_t + b^j_t)$ and $Q^j_t = 1_{\geq 0}(A_t + c^{j}_t)$
for all $0 \leq j \leq n-1$ and for all $t \in [t_j, t_{j+1}]$.  
Note that by our assumptions (and by Lemma \ref{lem:FredholmA1A2}), 
$P'_0 - P^0_0, P'_1 - Q^{n-1}_1 \in \B$,
$P^j_t - Q^j_t \in \B$ for all $0 \leq j \leq n-1$ and for all $t \in [t_j, 
t_{j+1}]$, and (by Lemma \ref{lem:1geq0Continuous})
 $\{ P^j_t \}_{t \in [t_j, t_{j+1}]}$ and 
$\{ Q^j_t \}_{t \in [t_j, t_{j+1}]}$ are norm-continuous paths of projections
in $\Mul(\B)$ for all $0 \leq j \leq n-1$.

We will repeatedly use the properties from Proposition \ref{prop:ECProperties}.
We have that 
\begin{eqnarray*}
&& Sf_1\\
& = & [P'_0:P^0_0] + [P^{n-1}_1 :P'_1] + \sum_{j=1}^{n-1} [P^{j-1}_{t_j}:
P^j_{t_j}]\\
& = &  [P'_0: Q^0_0] + [Q^0_0: P^0_0] + [P^{n-1}_1 : Q^{n-1}_1] + 
[Q^{n-1}_1:Q'_1] + \sum_{j=1}^{n-1}([P^{j-1}_{t_j} : Q^{j-1}_{t_j}]
+ [Q^{j-1}_{t_j}: Q^j_{t_j}] + [Q^j_{t_j}:P^j_{t_j}])\\
& & \makebox{(by repeated applications of Propositon \ref{prop:ECProperties}
Property (EC8))}\\
& = & Sf_2 + [Q^0_0: P^0_0] + [P^{n-1}_1 : Q^{n-1}_1] 
+ \sum_{j=1}^{n-1} ([P^{j-1}_{t_j} : Q^{j-1}_{t_j}]
+ [Q^j_{t_j}:P^j_{t_j}])\\  
& = & Sf_2 + [Q^0_{t_1}: P^0_{t_1}] + [P^{n-1}_1 : Q^{n-1}_1]
+ \sum_{j=1}^{n-1} ([P^{j-1}_{t_j} : Q^{j-1}_{t_j}]
+ [Q^j_{t_{j+1}}:P^j_{t_{j+1}}])\\
 & & \makebox{(by repeated applications of Propositon \ref{prop:ECProperties}
Property (EC2))}\\
& = & Sf_2 + \sum_{j=1}^{n} ([P^{j-1}_{t_j} : Q^{j-1}_{t_j}]
+ [Q^{j-1}_{t_{j}}:P^{j-1}_{t_{j}}])\\ 
& = & Sf_2\\
& & \makebox{(by repeated applications of Propositon \ref{prop:ECProperties}
Property (EC5))}
\end{eqnarray*}  
as required.  Since these partition and local trivializing families are
arbitrary, we have completed the first step  (the case where we have two sets
of local
trivializing families for $\{ A_t \}_{t \in [0,1]}$ with the same partition).

To complete the proof, suppose that we have two different sets of
 local trivializing families 
of $\{ A_t \}_{t \in [0,1]}$, corresponding to two different partitions.  Take a common 
refinement of the two partitions and apply the first step or first part of the proof.\\ 
\end{proof}

We end this section with a concrete example.  This is similar to  
\cite{WahlSpectralFlow} item 5 on page 153, but we add a (short)
explanation, and we use essential codimension (which is more general
than the relative index of \cite{WahlSpectralFlow}), and also,
our computation also
works for stably projectionless  canonical ideals (see also last paragraph
of Remark \ref{rmk:SfIntuitionAndOrientation2}).
                  
\begin{prop}  \label{prop:FirstExample} 
Let $\B$ be a separable stable C*-algebra, and let
$P, Q \in \Mul(\B)$ be projections such that $P - Q \in \B$.
(E.g., $P, Q \in \B$ are allowed.)  

Let $\{ A_t \}_{t \in [0,1]}$ be the continuous  path of 
self-adjoint Fredholm operators in $\Mul(\B)$ given by 
$$A_t =_{df} (1-t) (2 P - 1) + t (2Q - 1) \makebox{ for all  }
t \in [0,1].$$

Then 
$$Sf(\{ A_t \}_{t \in [0,1]},0,0) = [P:Q].$$ 
\end{prop}

\begin{proof}
A trivializing family for $\{ A_t \}_{t \in [0,1]}$
is $\{ b_t =_{df} - 2t(Q - P) \}_{t \in [0,1]}$.
Note that $b_t \in \B$ for all $t \in [0,1]$, since $Q - P \in \B$.  

In fact, for all $t \in [0,1]$,
\begin{eqnarray*}
A_t + b_t & = & (1 -t) (2P - 1) + t(2 Q - 1) + b_t \\
& = & (2P -1)  -t(2P - 1) + t (2Q - 1) + b_t \\
& = & (2P - 1) - t(2P) + t(2Q) + b_t \\
& = & (2P-1)  + 2t(Q - P) + b_t \\
& = & 2 P - 1.
\end{eqnarray*}

Hence, by  Definitions \ref{df:WahlSFDefinition} and \ref{df:WahlSFPrelim} (2), 
\begin{eqnarray*}
Sf(\{ A_t \}_{t \in [0,1]}, 0, 0) 
& = & Ind(2P - 1, 0, b_0) + Ind(2 Q - 1, b_1, 0)\\
& = & Ind(2P - 1, 0, 0) + Ind(2Q - 1, -2(Q - P), 0)\\
& = & [P:P] + [P:Q]\\
& = & [P:Q]   \makebox{     (by 
Proposition \ref{prop:ECProperties} (EC4))}\\   
\end{eqnarray*}
\end{proof}

\begin{rmk}  \label{rmk:SfIntuitionAndOrientation2}
Note that in the above example, if $P, Q \in \B$, then the spectral flow
is $[P:Q] = [P] - [Q] \in K_0(\B)$.  Hence, as one moves along the continuous path of 
self-adjoint Fredholm operators, the ``net mass" of the part of
the spectrum that passes through zero in the negative direction
is $[P] - [Q]$.  This conforms with our natural intuitions 
for this example.

Note again that, as mentioned previously, the \emph{orientation is opposite} 
from that of the original
Atiyah--Patodi--Singer spectral flow, as well as  subsequent versions
of spectral flow (e.g., \cite{AtiyahPatodiSinger3}, \cite{BenPhillEtAl}), 
where the 
spectral flow measures the ``net mass" of the part of the spectrum that passes 
through zero in the \emph{positive direction}. 

Finally, Proposition \ref{prop:FirstExample} can be used to give interesting
 examples  
where the canonical ideal is stably projectionless 
(recall that a C*-algebra $\D$ is stably projectionless if the stabilization
$\D \otimes \K$ has no projection other than zero).
By \cite{GongLinRange}, we can find a separable, simple, nuclear, 
stably finite, 
stable and stably projectionless C*-algebra $\B$ such that 
$K_0(\B) = \mathbb{Z}$. Since $K_0(\B) = KK(\mathbb{C}, \B)$, using the
generalized homomorphism picture of KK theory (see (\ref{equdf:KK_h})),
we can find projections $P, Q \in \Mul(\B)$ with $P - Q \in \B$ such that
$0 \neq [P:Q] = 1 \in K_0(\B)$.         
Defining $\{ A_t \}_{t \in [0,1]}$ as in Proposition \ref{prop:FirstExample},
we have an interesting example, with stably projectionless canonical 
ideal $\B$, for which the spectral flow $Sf(\{ A_t \}_{t \in [0,1]}, 0, 0)$
is nonzero.\\  
\end{rmk}

\subsection{The lifting projections hypothesis}
\label{subsect:LiftingProjectionsCondition}

Before discussing our main, general definition of spectral flow
(Definition \ref{df:SpectralFlow}), we discuss a relatively strong
hypothesis which has been present in all previous treatments of 
spectral flow for Hilbert C*-modules, as well as in the early work
of Melrose--Piazza  (\cite{MelrosePiazza}, \cite{DaiZhang}, \cite{Wu}, \cite{LeichtnamPiazzaJFA}, 
\cite{WahlSpectralFlow}).   This is essentially an existence of spectral
sections hypothesis (see Definition \ref{df:SpectralCutSection}) and this
hypothesis was shown, in various specific contexts, 
to be equivalent to the vanishing of a certain 
index (e.g., see \cite{MelrosePiazza} Proposition 1, \cite{DaiZhang} 
Proposition 1.3 (A), \cite{Wu} Theorem 2.2,  \cite{LeichtnamPiazzaJFA}
Theorem 3).  We will see that in general, the existence of spectral sections
implies vanishing of the $K_1$ index, but the converse is not true (even
after weakening the notion of spectral section) without
additional assumptions.

We will essentially be working with the more general notion of generalized
spectral section
from \cite{WahlSpectralFlow} Definition 3.4 (see 
Remark \ref{rmk:SpectralSectionTerminology}), 
but we prefer not to use the terminology
``existence of spectral section".  
Instead, we prefer to talk about ``projection lifting" which is an equivalent
condition where the terminology is more familiar to analysts, and which
points to well-known and longstanding
 questions in operator theory.  (We will also employ 
the terminology ``existence of trivializing operator", which
we have used already in defining spectral flow, and which also gives
an equivalent condition. See Theorem \ref{thm:ProjectionLiftingCondition}
for the equivalences.) 

Before moving forward, we will also define the index in the ``vanishing index
characterization" (for the existence of
spectral sections) in \cite{MelrosePiazza}, \cite{DaiZhang},
\cite{Wu} and \cite{LeichtnamPiazzaJFA}.  Firstly, recall that if $\B$  
is a separable stable C*-algebra, and if $A \in \Mul(\B)$   
is a self-adjoint Fredholm operator, then the Fredholm index 
$\partial([\pi(A)]) = 0$, where $\partial$ is the index map from the 
appropriate six-term
exact sequence (see Subsection \ref{subsect:GeneralizedFredholm}).        
(This is because since $\pi(A) \in \C(\B)$ is a self-adjoint invertible, 
it is homotopic
to $1$ in $GL(\C(\B))$,  and $[1] = 0$ in $K_1(\C(\B))$.) 
Thus, the index that is defined in \cite{MelrosePiazza}, \cite{DaiZhang},
\cite{Wu}, and \cite{LeichtnamPiazzaJFA} is something other than the usual
Fredholm index. Following,  \cite{Wu} and \cite{LeichtnamPiazzaJFA}, we use
$\partial_1$ to define this index.  Recall that $\partial_1 : 
K_0(\C(\B)) \rightarrow K_1(\B)$ is the \emph{exponential map} (see
Subsection \ref{subsect:GeneralizedFredholm}) and that it can be considered
an index map since it is defined using $\partial$ and some other obvious
maps.\\  

\begin{df} \label{df:K1Index}
Let $\B$ be a separable stable C*-algebra, and let $A \in \Mul(\B)$ be a 
self-adjoint Fredholm operator.

Then the \emph{$K_1$-index} of $A$ is defined to be
$$Ind_1(A) =_{df} \partial_1([1_{\geq 0}(\pi(A))])  \in K_1(\B).$$\\  

\noindent(Recall that $1_{\geq 0}(\pi(A))$ is the support projection of $\pi(A)_+$. See
the paragraph before Lemma \ref{lem:FredholmA1A2}.)
\end{df}

For a C*-algebra $\C$ and an ideal $\J \subseteq \C$, we let
$\pi_{\J} : \C \rightarrow \C/\J$ denote the quotient map.
Next we provide a preliminary exercise whose proof we give for the 
convenience of the reader.

\begin{prop} \label{prop:posinvlift}
Let $\C$  be a unital C*-algebra and $\J \subseteq \C$ a C*-ideal. Then every positive invertible element of
$\C/\J$ lifts to positive invertible element of   $\C$.
\end{prop}

\begin{proof}
Let $a \in \C/\J$ be a positive invertible element.  We may assume that $a$ is contractive.
Let $r > 0$ be such that the spectrum of $a$ is at least a distance
$2r$ away from $0$.  Let $f : [0,  \infty) \rightarrow [0,1]$ be the unique continuous function such that
\[
f(t) =
\begin{cases}
0 & t \in [r, \infty) \\
1 & t \in [0, \frac{r}{2}] \\
\makebox{linear on  } & [\frac{r}{2}, r].
\end{cases}
\]
Let $a_0 \in \C$ be a contractive positive element that lifts $a$, i.e., $\pi_{\J}(a_0) = a$.   
Then $\pi_{\J}(f(a_0)) = f(\pi_{\J}(a_0)) = f(a) =0$, and so $f(a_0) \in \J$.
So $a_0 + f(a_0) \in C^*(a_0, 1_{\C}) \subseteq \C$ is a positive lift of $a$.
Moreover, since $t + f(t)$ is strictly positive on the spectrum of $a_0$, 
$a_0 + f(a_0)$ is an invertible element of $C^*(a_0, 1_{\C})$.  Hence, $a_0 + f(a_0)$
is an invertible element of $\C$.  So $a_0 + f(a_0)$ is a positive invertible lift of $a$.\\ 
\end{proof}

The next preliminary result is the key connection between the existence
of a trivializing operator and the projection lifting
condition.

\begin{prop} \label{prop:SAInvertibleLift}
Let $\C$ be a unital C*-algebra and $\J \subseteq \C$ a C*-ideal. 
Let $a \in \C/\J$ be a self-adjoint invertible element.
Then $a$ can be lifted to a self-adjoint invertible element of $\C$ if and only 
if $1_{\geq 0}(a)$ can be lifted to a projection in $\C$.
\end{prop}

\begin{proof}

Recall that $\pi_{\J} :\C\to \C/\J$ is the quotient map. Assume that 
there is a self-adjoint and invertible $b\in \C$ such that 
$\pi_{\J}(b)=a$. Since the function $1_{\geq 0}$ is continuous on the spectrum of $b$, by the continuous functional calculus, we have $\pi_{\J}(1_{\geq 0}(b))=1_{\geq 0}(\pi(b))=1_{\geq 0}(a).$ 

For the converse direction, we assume that $p\in\C$ is a projection such that 
$\pi_{\J}(p)=1_{\geq 0}(a)$. Recall that $a_+$ and 
$a_-$ denote the positive and negative parts of $a$, respectively.  Recall that since the spectrum of  a self-adjoint 
invertible is bounded 
away from zero, $\pi_{\J}(p) = 1_{\geq 0}(a)$ is the support 
projection of $a_+$.  Similarly, $\pi_{\J}(1_{\C} - p) = 
1_{<0}(a)$ is the support projection of $a_-$ (here, $1_{<0}$ is the 
indicator function of $(-\infty, 0)$).   Since $a$ is invertible, $a_+$ and $a_-$ are both invertible in their respective
support projections.  So by Proposition \ref{prop:posinvlift}, let $a'_+$ be a positive invertible element of $p \C p$ and let $a'_-$ be a positive invertible element of $(1_{\C} - p) \C (1_{\C} - p)$ such that $\pi_{\J}(a'_+) = a_+$ and 
$\pi_{\J}(a'_-) = a_-$.  Then $a' =_{df} a'_+ - a'_-$ is a self-adjoint invertible element of $\C$ for which $\pi_{\J}(a') = a$.\\  
\end{proof}

We continue with a preliminary exercise in the continuous functional
calculus:

\begin{lem} \label{lem:May520235PM}
Let $0 < \epsilon < \frac{1}{2}$ be given.
Let $\C$ be a unital C*-algebra and $\J \subseteq \C$ a C*-ideal.
   
If $u \in \C/\J$ is a unitary such that $\| u - 1_{\C/\J} \| < \epsilon$
then there exists a unitary $U \in \C$ such that
$\pi(U) = u$ and $\| U - 1_{\C} \| < \epsilon$.
\end{lem}

\begin{proof}  Since $\| u - 1\| < \epsilon < \frac{1}{2}$, let
$a \in \C/\J$ be
a self-adjoint element with $\| a \| < 1$ such that
$u = e^{ia}$, and if $\max\{ |s - 0| : s \in sp(a) \} = \delta$ then $|e^{i\delta} -1| < \epsilon$.
Let $A \in \C_{SA}$ be such that $\pi(A) = a$ and $\| A \| = \| a \|$.
If we let $U =_{df} e^{iA} \in U(\C)$, then $\pi(U) = u$ and
$\| U - 1 \| < \epsilon.$\\
\end{proof}

The next preliminary result follows immediately from \cite{BlackadarHomotopyLifting}  (see \cite{BlackadarHomotopyLifting} Theorem 5.1), but for the
convenience of the reader, we provide the short concrete proof for our special case.

\begin{lem}  \label{lem:HomotopyProjectionLift}
Let $\C$ be a unital C*-algebra and $\J \subseteq \C$ a C*-ideal.
 
Let $p \in \C/\J$ be a projection, and suppose that $p$ is homotopy-equivalent (in $Proj(\C/\J)$) to a projection $q \in \C/\J$ which can
be lifted to a projection in $\C$.

Then $p$ can be lifted to a projection in $\C$.
\end{lem}

\begin{proof}
Fix $0 < \epsilon < \frac{1}{2}$.  

Let $\{ p_t \}_{t \in [0,1]}$ be a norm-continuous path of projections in $\C/\J$
such that $p_0 = q$  and $p_1 = p$.   Since $\{ p_t \}_{t \in [0,1]}$ is norm-continuous, by \cite{WeggeOlsen} Proposition 5.2.6,
we can find a partition of $[0,1]$
$$0 = t_0 < t_1 < ... < t_n = 1$$
and unitaries $u_0, u_1, ..., u_{n-1} \in \C/\J$, with  $\| u_j - 1 \| < \epsilon$ for all $0 \leq j \leq n-1$, such that
$u_j p_{t_j} u_j^* = p_{t_{j+1}}$ for all $0 \leq j \leq n-1$.   In particular,
if we define $v =_{df} u_{n-1} u_{n-2} \cdot \cdot \cdot u_0$, then  $v q v^* = p$.

But by Lemma \ref{lem:May520235PM}, $u_j$ can be lifted to a unitary in $\C$ for all $0 \leq j \leq n-1$.  Hence, $v$ can be lifted
to a unitary in $\C$.  Hence, $v q v^*$ can be lifted to a projection in $\C$.  Hence, $p$ can be lifted to a projection in $\C$.\\
\end{proof}

We now prove the first main result of this Subsection, 
which is a characterization of the existence of trivializing operators,
generalizing \cite{MelrosePiazza} Proposition 1, \cite{DaiZhang} Proposition
1.3 (A), \cite{Wu} Theorem 2.2 and \cite{LeichtnamPiazzaJFA} Theorem 3.
We note that, in view of the last statement in Theorem \ref{thm:ProjectionLiftingCondition}, perhaps one reason for the smoothness and success of the theory
of spectral flow in $\mathbb{B}(l_2)$ and more general semifinite von Neumann
algebras (e.g., see \cite{BenPhillEtAl}) is that
von Neumann algebras have real rank zero.   

\begin{thm} \label{thm:ProjectionLiftingCondition}
Let $\B$ be a separable stable C*-algebra, and let 
$A \in \Mul(\B)$ be a self-adjoint Fredholm operator.
Recall that $\pi(A) \in \C(\B)$ is a self-adjoint invertible operator. 

Then the following statements are equivalent:
\begin{enumerate}
\item $A$ has a trivializing operator (i.e., $A$ has a generalized
spectral section; see Remark \ref{rmk:SpectralSectionTerminology} (2)). 
\item $\pi(A)$ can be lifted to a self-adjoint invertible operator
 in $\Mul(\B)$.  
\item $1_{\geq 0}(\pi(A))$ can be lifted to a projection in $\Mul(\B)$.
\item  $1_{\geq 0}(\pi(A))$ is homotopy equivalent to a projection  
$q \in \C(\B)$ such that $q$ can be lifted to a projection in $\Mul(\B)$.\\ 
\end{enumerate} 

Now consider the following statement: 
\begin{enumerate}
   
\item[(5)] The $K_1$-index $Ind_1(A) = 0$. 
\end{enumerate}

Then we have that  (3) $\Rightarrow$ (5).\\

If $1_{\C(\B)} \preceq 1_{\geq 0}(\pi(A))$ \emph{and}
$1_{\C(\B)} \preceq 1_{\C(\B)} - 1_{\geq 0}(\pi(A))$,
then (3) $\Leftrightarrow$ (5), and hence,
(1) $\Leftrightarrow$ (2) $\Leftrightarrow$ (3) $\Leftrightarrow$
(4) $\Leftrightarrow$ (5).\\ 

If $\Mul(\B)$ has real rank zero, then all the statements (1), (2), (3), (4)
and (5) are true.   

\end{thm}

\begin{proof}
That (1) $\Leftrightarrow$ (2) is trivial. 
That (2) $\Leftrightarrow$ (3) follows immediately from 
Proposition  \ref{prop:SAInvertibleLift}.  
Hence, we have that (1) $\Leftrightarrow$ (2) $\Leftrightarrow$ (3).

That (3) $\Rightarrow$ (4) is trivial.  The converse, that (4) $\Rightarrow$ (3), follows immediately from
Lemma \ref{lem:HomotopyProjectionLift}.
Let us now prove that  (3) $\Rightarrow$ (5):  Suppose that
$1_{\geq 0}(\pi(A))$ can be lifted to a projection, say $R \in \Mul(\B)$
(so $\pi(R) = 1_{\geq}(\pi(A))$).  By \cite{WeggeOlsen} Theorem 10.2,
$K_0(\Mul(\B)) = 0$, so $[R] = 0$ in $K_0(\Mul(\B))$; and hence,
$[1_{\geq 0}(\pi(A))] = [\pi(R)] = 0$ in $K_0(\C(\B))$.
Thus, 
the $K_1$-index $Ind_1(A) = \partial_1([1_{\geq 0}(\pi(A))]) = 0$ in
$K_1(\B)$.  This completes the proof of (3) $\Rightarrow$ (5).

Now suppose that 
$1_{\C(\B)} \preceq 1_{\geq 0}(\pi(A))$ and  
$1_{\C(\B)} \preceq 1_{\C(\B)} - 1_{\geq 0}(\pi(A))$.
We now prove that (5) $\Rightarrow$ (3).  So assume (5).
So $Ind_1(A) = \partial_1([1_{\geq 0}(\pi(A))]) = 0$ in $K_1(\B)$.    
Since $\partial_1$ is a group isomorphism (see Subsection
\ref{subsect:GeneralizedFredholm}, the paragraph before
Definition \ref{df:FredholmIndex}),      
we have that $[1_{\geq 0}(\pi(A))] = 0$ in $K_0(\C(\B))$.
But since $[1_{\Mul(\B)}] = 0$ in $K_0(\Mul(\B))$ (by 
\cite{WeggeOlsen} Theorem 10.2), we must have that 
$[1_{\C(\B)}] = \pi_*([1_{\Mul(\B)}]) = 0$ in $K_0(\C(\B))$.
Hence, $[1_{\C(\B)} - 1_{\geq 0}(\pi(A))] = 0$. Hence,
$[1_{\geq 0}(A)] = [1_{\C(\B)} - 1_{\geq 0}(\pi(A))] = 0$ in 
$K_0(\C(\B))$.  Hence, 
since $1_{\C(\B)} \preceq 1_{\geq 0}(\pi(A))$
and $1_{\C(\B)} \preceq 1_{\C(\B)} - 1_{\geq 0}(\pi(A))$, and
since $1_{\C(\B)}$ is properly infinite, we must have that  
$$1_{\geq 0}(\pi(A)) \sim 1_{\C(B)} \sim 
1_{\C(\B)} - 1_{\geq 0}(\pi(A)).$$   
 Now since $\B$ is stable,  we can find a projection
 $P \in \Mul(\B)$  such that
$P \sim 1_{\Mul(\B)} \sim 1_{\Mul(\B)} - P$.
Hence, $\pi(P) \sim 1_{\C(\B)}  \sim 1_{\C(\B)} - \pi(P)$.
Hence, $\pi(P) \sim 1_{\geq 0}(\pi(A))$ and 
$1_{\C(\B)} - \pi(P) \sim 1_{\C(\B)} - 1_{\geq 0}(\pi(A))$.
Hence, let $u \in \C(\B)$ be a unitary such that 
$u \pi(P) u^* = 1_{\geq 0}(\pi(A))$. 
Since $\pi(P) \sim 1_{\C(\B)}$ and since $1_{\C(\B)}$ is properly infinite,
we can find a unitary $v \in \pi(P) \C(\B) \pi(P)$ such that 
if we define $u' =_{df} u ( v + 1_{\C(\B)} - \pi(P))$, then $u' \in \C(\B)$
is a unitary which is homotopic to $1_{\C(\B)}$.  Hence, by 
\cite{WeggeOlsen} Corollary 4.3.3, we can find a unitary $U \in \Mul(\B)$
for which $\pi(U) = u'$.
But by the definition of $u'$, $u' \pi(P) (u')^* = 1_{\geq 0}(\pi(A))$.
Hence, $1_{\geq 0}(\pi(A))$ can be lifted to the projection
$U P U^* \in \Mul(\B)$.   I.e., we have proven (3).

Now assume that $\Mul(\B)$ has real rank zero.
By \cite{ZhangPacific}, every projection in $\C(\B)$ lifts to a projection
in $\Mul(\B)$.  Hence, we have condition (3), and thus, from the above,
we have all the
conditions (1), (2), (3), (4) and (5).
\end{proof}

\begin{ex}  \label{ex:VanishingK1NoProjectionLift}
Here is an example that shows that vanishing of the $K_1$ index (alone) 
does not
imply the existence of a trivializing operator (equivalently, projection lifting).

Let $\W$ be the continuous scale Razak algebra (\cite{Razak}).
By \cite{EGLNW}, $\W$ is the unique continuous scale simple nuclear 
$\Z$-stable stably projectionless C*-algebra with a unique (up to positive
scalar multiple) trace and
satisfying the Universal 
Coefficient Theorem such that $K_*(\W) = 0$.  
(Here, $\Z$ is the Jiang--Su algebra (\cite{JiangSu}).)
Let $\tau$ be the unique (up to positive scalar multiple) trace of $\W$,
normalized to being a tracial state on $\W$.   
Let $\B =_{df} \W \otimes \K$, and we extend $\tau$ to a strict 
lower semicontinuous trace $\Mul(\B)_+ \rightarrow [0, \infty]$, which we 
also denote by ``$\tau$".
By \cite{LinContScale} (see also \cite{KNZMinimal}), let $\J \subset \Mul(\B)$
be the (unique) minimal ideal of $\Mul(\B)$ which properly contains $\B$.
By \cite{KNZMinimal}, $\J/\B$ is simple purely infinite, and by 
\cite{KNZPI}, $\pi(\J)$ is the unique proper nontrivial ideal of
$\C(\B)$.  By \cite{LinNgZ} Proposition 4.2, Corollary 4.6 and Theorem 4.7, 
we have that $(K_0(\J), K_0(\J)_+) = (\mathbb{R}, \mathbb{R}_+)$; in fact,
every element of $K_0(\J)_+ - \{ 0 \}$ can be realized as the class of a
projection in $\J - \B$, and for every projection $P \in \J - \B$,   
$[P] = \tau(P) \in \mathbb{R}_+ - \{ 0 \} = K_0(\J)_+ - \{ 0  \}$.    
Note also that by applying the six term exact sequence in K theory to the 
extension $0 \rightarrow \B \rightarrow \J  \rightarrow \J/\B \rightarrow 0$,
and using that $K_0(\J) = \mathbb{R}$ and
$K_j(\B) = K_j(\W) = 0$ ($j = 0,1$), we have that the map
$(\pi_{\B})_* : K_0(\J) \rightarrow K_0(\J/\B)$ (induced by the 
quotient map $\pi_{\B} : \J \rightarrow \J/\B$)  is a group isomorphism -- 
indeed it is the identity map on $\mathbb{R}$; and so 
$K_0(\J/\B) = K_0(\J) =
\mathbb{R}$.  
Also, by \cite{WeggeOlsen} Theorem 10.2 and Corollary 10.3, $K_0(\Mul(\B)) = 0$
and $K_0(\C(\B)) = K_1(\B) = K_1(\W) = 0$.  

Now since $\J/\B$ is simple purely infinite, we can find a nonzero projection
$p \in \J/\B$ such that $[p] = 0$ in $K_0(\J/\B)$.  $p$ cannot be lifted to
a projection in $\J$, i.e., there is no projection $P \in \J$  such that
$\pi_{\B}(P) = p$.  This is because if such a $P$ exists then, necessarily, 
$P \in \J - \B$, and so there exists a nonzero $r \in \mathbb{R}$ for
which $[P] = \tau(P) = r \in K_0(\J) = \mathbb{R}$.  So since the  map    
$(\pi_{\B})_* :  K_0(\J) \rightarrow K_0(\J/\B)$ is the 
identity map on $\mathbb{R}$, $[p] = [\pi_{\B}(P)] = r \neq 0$.  This 
contradicts that $[p] = 0$ in $K_0(\J/\B)$.    

Since $p$ cannot be lifted to a projection in $\J$, $p$ cannot be lifted to 
a projection in $\Mul(\B)$.  Let $A \in \Mul(\B)$ be 
a self-adjoint lift of $p - (1 - p) \in GL(\C(\B))_{SA}$.  Hence, 
$A$ is a self-adjoint Fredholm operator in $\Mul(\B)$ for which 
$Ind_1(A) = \partial_1([1_{\geq 0 }(\pi(A))]) = \partial_1([p]) 
= \partial_1(0) = 0$.  But $1_{\geq 0 }(\pi(A)) = p$ cannot be lifted to 
a projection in $\Mul(\B)$.\\ 
\end{ex}

\begin{rmk} \label{rmk:LiftingProjections} 
Note that a self-adjoint invertible element of a unital C*-algebra is path-connected to the unit via a norm-continuous
path of invertible elements.  Hence,
the self-adjoint invertible $a \in \C/\J$ in Proposition 
\ref{prop:SAInvertibleLift} can always be lifted to an invertible element of  
$\C$, even without the
requirement that $1_{\geq 0} (a)$ can be lifted to a projection in 
$\C$ (see \cite{WeggeOlsen} Corollary 4.3.3).  
However, the lift need not be self-adjoint. This subtlety 
is part of the point of 
Proposition \ref{prop:SAInvertibleLift} and Theorem 
\ref{thm:ProjectionLiftingCondition}.

Also, a projection in $\C/\J$ need not be liftable to a projection in $\C$.
In fact, this is a very strong condition that has been very much the subject
of study in operator theory and 
operator algebras. E.g., a well-known result is that if $\C$ has real 
rank zero, then every projection in $\C/\J$ lifts to a projection in
$\C$ (see \cite{ZhangPacific} 3.2 and \cite{BrownPedersen} Theorem
3.14).

Let us look at the projection lifting
condition, which is an important statement in analysis,
from still another point of view.  Say that
$\J = \B$ where $\B$ is a separable stable C*-algebra, and say that
$\C = \Mul(\B)$.  The previous two paragraphs also say the following: 
A self-adjoint Fredholm operator in $\Mul(\B)$ has Fredholm index zero,
and hence, it can always be perturbed (by adding an element of $\B$) to
an invertible element of $\Mul(\B)$.  However, a subtlety is that
 the perturbation need not
be a \emph{self-adjoint} invertible element of $\Mul(\B)$.  By Proposition
\ref{prop:SAInvertibleLift}, this is connected with the projection lifting
problem, and this problem is connected to interesting and fundamental problems
in operator theory.  For example, the following is a conjecture of Brown--Pedersen--Zhang
(e.g., see \cite{SZhang} and \cite{BrownPedersen}).  
Let $\B$ be a separable, stable C*-algebra with real rank zero.  
Then are the following statements equivalent?:
\begin{enumerate}    
\item $\Mul(\B)$ has real rank zero.
\item $\Mul(\B)$ has the Weyl--von Neumann theorem for self-adjoint operators.
\item $K_1(\B) = 0$.
\item Every projection in $\C(\B)$ lifts to a projection in $\Mul(\B)$.
\end{enumerate}
Some of the above implications have already been proven, and some for 
special cases (e.g., see \cite{ZhangPacific}, \cite{LinWvN} and
\cite{BrownPedersen}).    
We note that real rank zero and projection-lifting were implicitly used in the ground-breaking work
of Brown--Douglas--Fillmore (\cite{BDF}), 
even though this was before the terminology
``real rank zero" was invented.   Moreover, real rank zero has formal similarities to the SAW* property for corona
algebras which is implied by the Kasparov Technical Lemma, a fundamental result which is used to prove basic
properties of KK.                   
In fact, as mentioned previously, perhaps one big reason, for the success and
smoothness of the theory of spectral flow for $\mathbb{B}(l_2)$ 
and more general semifinite von Neumann algebras (e.g.,
\cite{BenPhillEtAl}),
is that von Neumann algebras have real rank zero.  For example, if 
$\Mul$ is a $II_{\infty}$ factor with separable predual, and if $\K_{\Mul}
\subseteq \Mul$ is its Breuer ideal then, since $\Mul$ has real rank zero,
every projection in $\Mul/\K_{\Mul}$ lifts to a projection in $\Mul$.  

Thus, in light of Proposition \ref{prop:SAInvertibleLift} and 
Theorem \ref{thm:ProjectionLiftingCondition}, 
the conditions of having a trivializating operator (or ``generalized spectral
sections") or having local
 trivializing families, as in Definition \ref{df:WahlSFDefinition}, are, from
the perspective of an analyst, very restrictive conditions but also very interesting and tied to fundamental problems
in operator theory. 
As noted earlier in this remark, this condition is even stronger than having Fredholm
index zero.

Finally, we note that the hypotheses in the second last paragraph of 
the statement of Theorem \ref{thm:ProjectionLiftingCondition}  (i.e.,
the hypotheses $1 \preceq 1_{\geq 0}(\pi(A))$ and 
$1 \preceq 1 - 1_{\geq 0}(\pi(A))$, which are used to 
show (3) $\Leftrightarrow$ (5)) is the same as 
statements used in \cite{LeichtnamPiazzaJFA} to show that (under certain
conditions) the existence of spectral sections is equivalent to the vanishing
of a certain index  (e.g., see \cite{LeichtnamPiazzaJFA} Theorem 2). In fact,
parts of the proof here are also similar.\\    
\end{rmk}

We end this section by proving that in Definition \ref{df:WahlSFDefinition},
one only needs one cell in the partition, i.e., we just need the interval
$[0,1]$.  Towards this, we prove the following preliminary lemma:

\begin{lem} \label{lem:Nov320231AM} 
Let $\C$ be a unital C*-algebra, and let $\J \subset \C$ be a C*-ideal.
Let $\{ p_t \}_{t \in [r, s]}$ be a
norm-continuous path of projections in $\C$, and suppose that $q \in \C$ is
a projection for which $p_r - q \in \J$.

Then we can find a norm-continuous path $\{ q_t \}_{t \in [r,s]}$ of projections
in $\C$ such that 
$q_r = q$ and $p_t - q_t \in \J$ for all $t \in [r,s]$. 
\end{lem}

\begin{proof} 
By \cite{WeggeOlsen} Proposition 5.2.6, we can find a norm-continuous path
$\{ u_t \}_{t \in [r,s]}$ of unitaries in $\C$ 
with $u_r = 1_{\C}$ such that $u_t p_r u_t^* = p_t$ for all $t \in [r, s]$.
Define $q_t =_{df} u_t q u_t^*$ for all $t \in [r,s]$. 
Note that for all $t \in [r,s]$,
$p_t - q_t = u_t(p_r - q) u_t^* \in \J$.\\   
\end{proof}

We now show that, in Definition \ref{df:WahlSFDefinition}, we only need one
cell in the partition, i.e., $\{ A_t \}_{t \in [0,1]} \in 
\Mul(C[0,1] \otimes \B)$
can have a single trivializing operator in $C[0,1] \otimes \B$.  This also
shows the similarity between Definition \ref{df:WahlSFDefinition} and the
older versions of spectral flow in \cite{DaiZhang}, \cite{Wu} and
\cite{LeichtnamPiazzaJFA}  (see Subsection \ref{subsection:DZWLPDefinition},
Definition \ref{df:DZWLPSpectralFlow}).

\begin{prop}  \label{prop:WahlSFSingleCell}
Let $I =_{df} [0,1]$.  Let $\B$ be a separable stable C*-algebra, and suppose
that $\{ A_t \}_{t \in I}$ is a norm-continuous path   
of self-adjoint Fredholm operators in $\Mul(\B)$ for which local 
trivializing families exist.

Then we have a norm-continuous path $\{ b_t \}_{t \in I}$ of self-adjoint
operators in $\B$ such that
$\{ b_t \}_{t \in I}$ is a trivializing family for
$\{ A_t \}_{t \in I}$.  

As a consequence, if  $b'_0$ and $b'_1$, in $\B_{SA}$,  
are trivializing operators for 
$A_0$ and $A_1$ respectively, then
$$Sf(\{ A_t \}_{t \in I}, b'_0, b'_1) 
= Ind(A_0, b'_0, b_0)  + Ind(A_1, b_1, b'_0).$$

Equivalently, from the above, $\{ Q_t =_{df} 
1_{\geq 0}(A_t + b_t) \}_{t \in [0,1]}$
is a norm-continuous path of projections in $\Mul(\B)$, and 
$$Sf(\{ A_t \}_{t \in I}, b'_0, b'_1) = 
[P_0 : Q_0] + [Q_1 : P_1]$$
where $P_0 =_{df} 1_{\geq 0}(A_0 + b'_0)$ and
$P_1 =_{df} 1_{\geq 0}(A_1 + b'_1)$.  (Note the similarity with
Definition \ref{df:DZWLPSpectralFlow}.)   
\end{prop}

\begin{proof}
Following Definition \ref{df:WahlSFDefinition},
by the hypothesis of the existence of local trivializing families,
we can find a partition 
$$0 = t_0 < t_1 < ... < t_n = 1$$
of $I$ such that for all $0 \leq j \leq n-1$, we can find 
$\{ c^j_t \}_{t \in [t_j, t_{j+1}]}$ (continuous 
family in $\B_{SA}$) which is a
trivializing family for $\{ A_t \}_{t \in [t_j, t_{j+1}]}$.   
Hence, by Lemma \ref{lem:1geq0Continuous}, for all $0 \leq j \leq n-1$,
$\{ 1_{\geq 0}(A_t + c^j_t) \}_{t \in [t_j, t_{j+1}]}$ is a norm-continuous
path of projections in $\Mul(\B)$.    
By repeatedly applying Lemma \ref{lem:Nov320231AM}, we can find a 
norm-continuous path $\{ R_t \}_{t \in [0,1]}$ of projections in $\Mul(\B)$
such that 
$R_0 = 1_{\geq 0}(A_0 + c^0_0)$ 
and
$$R_t - 1_{\geq 0}(A_t + c^j_t) \in \B \makebox{ for all  } 0 \leq j \leq n-1
\makebox{  and  } t \in [t_j, t_{j+1}].$$
Hence, for all $0 \leq j \leq n-1$ and for all $t \in [t_j, t_{j+1}]$,
\[
\pi(R_t) = \pi(1_{\geq 0}(A_t + c^j_t)) = 1_{\geq 0}(\pi(A_t + c^j_t)) =
1_{\geq 0}(\pi(A_t)).
\]
(Recall that for all $t \in I$,  since $A_t$ is self-adjoint
and Fredholm, $\pi(A_t)$
is self-adjoint and  invertible.)
So for all $t \in I$, $\pi(R_t)$ is the support projection of  
$\pi(A_t)_+$.  Hence, 
the projection $\{ \pi(R_t) \}_{t \in I}$, in $\C(C[0,1] \otimes \B)$,
is the support projection of the positive part of the
 self-adjoint invertible element
$\{ \pi(A_t) \}_{t \in I}$ of $\C(C[0,1] \otimes \B)$. 
Hence, since $\{ \pi(R_t) \}_{t \in I}$ lifts to a projection
$\{ R_t \}_{t \in I}$ in $\Mul(C[0,1] \otimes \B)$,  by Theorem 
\ref{thm:ProjectionLiftingCondition} (1) $\Leftrightarrow$ (3),
we have that $\{ A_t \}_{t \in I}$ has a trivializing operator in 
$(C[0,1] \otimes \B)_{SA}$. I.e., we can find a norm-continuous
path $\{ b_t \}_{t \in I}$ of self-adjoint operators in $\B$ such that
$\{ b_t \}_{t \in I}$ is trivialing family for $\{ A_t \}_{t \in I}$.
The remaining statements follow from      
Lemma \ref{lem:1geq0Continuous} and 
Definition \ref{df:WahlSFDefinition}.\\  
\end{proof}

\subsection{The general definition}

Towards a concrete and more general definition for spectral flow with
hypotheses that are relatively general and easy to check   (and which
does not require the canonical ideal
 to have a nonzero projection),
let us first
prove some preliminary technical lemmas, one of which essentially says that the simple homotopy projection lifting result of 
Lemma \ref{lem:HomotopyProjectionLift} gives us the projection-lifting that we need.\\

\begin{lem}
\label{lem:TechnicalPreDef1}  
Let $I \subset \mathbb{R}$ be a compact interval, and 
let $\B$ be a separable stable C*-algebra.  
Let $\{ A_t \}_{t \in I}$ be a norm-continuous path of self-adjoint Fredholm operators in 
$\Mul(\B)$ such that $A_0$ is invertible in $\Mul(\B)$.

Then for all $t \in I$,
$A_t$ has a trivializing 
operator in $\B$.
\end{lem}   

\begin{proof}
Fix an arbitrary $t' \in I$.  We will prove that $A_{t'}$
has a trivializing operator in $\B$.  

We have that $\{ \pi(A_t) \}_{t \in I}$ is a norm-continuous 
path of self-adjoint
invertible elements of $\C(\B)$.  
By Lemma \ref{lem:1geq0Continuous},  $\{ 1_{\geq 0}(\pi(A_t)) \}_{t \in I}$ is 
a norm-continuous path of projections in $\C(\B)$.
But since $A_0$ is a self-adjoint invertible in $\Mul(\B)$,
$1_{\geq 0}(A_0)$ is a projection in $\Mul(\B)$ and, by the continuous      
functional calculus, 
$\pi(1_{\geq 0}(A_0)) = 1_{\geq 0}(\pi(A_0))$.  Hence, 
$1_{\geq 0}(\pi(A_{t'}))$ is homotopy equivalent to a projection in $\C(\B)$
which can be lifted to a projection in $\Mul(\B)$. Hence,
Hence, by Theorem \ref{thm:ProjectionLiftingCondition}  (1) $\Leftrightarrow$ (4),
$A_{t'}$ has a trivializing operator in $\B$.\\ 
\end{proof}

\begin{lem} (Cf. \cite{WahlSpectralFlow}  Lemma 3.12)  
\label{lem:TechnicalPreDef2}
Let $I \subset \mathbb{R}$ be a compact interval, and let
$\B$ be a separable stable C*-algebra.
Let $\{ A_t \}_{t \in I}$ be a norm-continuous path of Fredholm operators
in $\Mul(\B)$.
\begin{enumerate}
\item \label{lem:TechnicalPreDef2Item1}  
For all $t' \in I$, if  $A_{t'}$ has a trivializing operator $b \in \B_{SA}$
then there is an open neighbourhood $O \ni t'$ such that $A_{t}+ b$ is invertible
for all $t \in O$.
\item \label{lem:TechnicalPreDef2Item2}  
As a consequence, if $A_t$ has a trivializing operator for all $t \in I$,
then there are local trivializing families for 
$\{ A_t \}_{t \in I}$. 
\end{enumerate}
\end{lem}  

\begin{proof}
Statement (\ref{lem:TechnicalPreDef2Item2}) follows immediately from statement
(\ref{lem:TechnicalPreDef2Item1}).   Statement (\ref{lem:TechnicalPreDef2Item1})
follows from the fact that in a unital C*-algebra, the set of invertible elements
forms a (norm topology) open set (e.g., see \cite{WeggeOlsen}
Lemma 4.2.1).\\ 
\end{proof}

The preliminary technical lemmas lead to 
the following result, which is essentially
our general definition of spectral flow:

\begin{prop}   \label{July1120231AM}
Let $\B$ be a separable stable C*-algebra and let $\{ A_t \}_{t \in [0,1]}$
be a norm-continuous path of self-adjoint Fredholm operators in $\Mul(\B)$ such that 
$A_0$ and $A_1$ are invertible in $\Mul(\B)$.

Then $Sf(\{ A_t  \}_{t \in [0,1]}, 0, 0)$, as in    
equation (\ref{equ:WahlSfDefinition}) of Definition \ref{df:WahlSFDefinition},
exists and is well-defined in $K_0(\B)$. 
\end{prop}

\begin{proof}
Firstly, by Lemmas \ref{lem:TechnicalPreDef1}  and  \ref{lem:TechnicalPreDef2},
$\{ A_t \}_{t \in [0,1]}$, as a norm-continuous path of
Fredholm operators in $\Mul(\B)$, has local trivializing families.
Hence, we can apply Definition \ref{df:WahlSFDefinition}.   

Well-definedness follows from Proposition \ref{prop:WahlSFWellDefined}.\\ 
\end{proof}

Our general definition of spectral flow now follows from
Proposition \ref{July1120231AM}:  

\begin{df}  \label{df:SpectralFlow} 
Let $\B$ be  a separable stable C*-algebra and let $\{ A_t \}_{t \in [0,1]}$
be a norm-continuous path of self-adjoint Fredholm operators in $\Mul(\B)$ such that
$A_0$ and $A_1$ are invertible in $\Mul(\B)$.

Then  the \emph{spectral flow} of  $\{ A_t \}_{t \in [0,1]}$ is defined to 
be 
$$Sf(\{ A_t \}_{t \in [0,1]}) =_{df} 
Sf(\{ A_t \}_{t \in [0,1]}, 0, 0)
\in K_0(\B).$$\\  
\end{df}

\begin{rmk}
\label{rmk:Nov820231AM}
We reemphasize that in the above definition, the canonical ideal is an arbitrary separable stable C*-algebra (and need not have
a nonzero projection), and we are starting with an arbitrary norm-continuous path of Fredholm operators with invertible endpoints, which is a quite general,  simple 
and easy to check set of conditions.
\end{rmk}

\begin{rmk}
As an interesting side remark, we note that Theorem 
\ref{thm:ProjectionLiftingCondition} 
and  Lemma 
\ref{lem:TechnicalPreDef2} imply the following:  Let $\B$ be a separable stable C*-algebra, and let 
$\{ A_t \}_{t \in [0,1]}$ be a norm-continuous path of self-adjoint Fredholm operators in $\Mul(\B)$.
Then $\{ A_t \}_{t \in [0,1]}$ has local trivializing families if and only if for all $t \in [0,1]$, the projection
$1_{\geq 0}(\pi(A_t)) \in \C(\B)$ lifts to a projection in $\Mul(\B)$.

We have certainly used the above principle implicitly in our arguments, and this is a feature which we have emphasized in the present paper.\\ 

\end{rmk}

\begin{prop} \label{prop:SecondExample}
Let $\B$ be a separable stable C*-algebra.  
Let $P, Q \in \Mul(\B)$ be projections such that
$P - Q \in \B$.

Let $\{ A_t \}_{t \in [0,1]}$ be a (norm-) continuous path of
self-adjoint Fredholm operators in $\Mul(\B)$ for which
$$A_t =_{df} (1 - t) (2P - 1) + t (2Q - 1) \makebox{  for all  }
t \in [0,1].$$   

Then 
$$Sf(\{ A_t \}_{t \in [0,1]}) = [P : Q].$$ 
\end{prop} 

\begin{proof}
This follows from Definition \ref{df:SpectralFlow}, following the same computation
as that of Proposition \ref{prop:FirstExample}.\\ 

\end{proof}

Finally, we end this subsection by showing that our standard minimal assumptions
for spectral flow
(norm-continuous paths of self-adjoint Fredholm operators with invertible endpoints)
imply the vanishing of the global index.

\begin{prop}  \label{prop:VanishGlobalIndex}
Let $\B$ be a separable stable C*-algebra, and let $\{ A_t \}_{t \in [0,1]}$
be a norm-continuous path of self-adjoint Fredholm operators in $\Mul(\B)$
with $A_0$ and $A_1$ being invertible. 
By Lemma \ref{lem:1geq0Continuous}, $\{ 1_{\geq 0}(\pi(A_t)) \}_{t \in [0,1]}$ 
is a norm-continuous path of projections in $\C(\B)$.

Then there exists a norm-continuous path $\{ P_t \}_{t \in [0,1]}$ of 
projections in $\Mul(\B)$ such that 
$$\pi(P_t) = 
1_{\geq 0}(\pi(A_t))  \makebox{ for all  }t \in [0,1].$$
As a consequence, (viewing $\{ A_t \}_{t \in [0,1]}$ as a self-adjoint
Fredholm operator in $\Mul(C[0,1] \otimes \B)$),
$$Ind_1(\{ A_t \}_{t \in [0,1]}) = 0 \makebox{  in  } K_1(C[0,1] \otimes \B).$$
\end{prop}

\begin{proof}
By Lemmas \ref{lem:TechnicalPreDef1} and \ref{lem:TechnicalPreDef2},
$\{ A_t  \}_{t \in [0,1]}$ is a norm-continuous
path of self-adjoint Fredholm operators in $\Mul(\B)$ (with invertible endpoints) for which
local trivializing families exist  
(in $\B_{SA}$). 
By Proposition \ref{prop:WahlSFSingleCell}, 
there is a norm-continuous path $\{ b_t \}_{t \in [0,1]}$
of  self-adjoint operators in $\B$ such that
$\{ b_t \}_{t \in [0, 1]}$ is a trivializing family for
$\{ A_t  \}_{t \in [0,1]}$ (so $A_t  + b_t$
is self-adjoint and
invertible in $\Mul(\B)$ for all $t \in [0,1]$), and
$\{ P_t =_{df} 1_{\geq 0}(A_t 
 + b_t) \}_{t \in [0,1]}$ is a norm-continuous
path of projections in $\Mul(\B)$.
By the continuous functional calculus,
for all $t \in [0,1]$,
$$\pi(P_t) = \pi(1_{\geq 0}(A_t + b_t)) = 1_{\geq 0}(\pi(A_t + b_t))
= 1_{\geq 0}(\pi(A_t)),$$
and it is not hard to see that viewing $\{ P_t \}_{t \in [0,1]}$
and $\{ 1_{\geq 0}(\pi(A_t)) \}_{t \in [0,1]}$ 
as projections in $\Mul( C[0,1] \otimes \B)$  and
$\C(C[0,1] \otimes \B)$ respectively,
$$\pi(\{ P_t \}_{t \in [0,1]}) = 
\{ 1_{\geq 0}(\pi(A_t)) \}_{t \in [0,1]}.$$
Here, to keep things simple, we use the same notation $\pi$ to denote
all quotient maps $\Mul(\B) \rightarrow \C(\B)$,
$\Mul(C[0,1] \otimes \B) \rightarrow \C(C[0,1] \otimes \B)$,
$\Mul(\B) \rightarrow \C(\B)$, and 
$\Mul(C[0,1] \otimes \B) \rightarrow 
\C(C[0,1] \otimes \B)$. 
 
By \cite{WeggeOlsen} Theorem 10.2, $K_0(\Mul(C[0,1] \otimes \B)) = 0$.
Hence,
$[\{ P_t \}_{t \in [0,1]}] = 0$ in $K_0(\Mul(C[0,1] \otimes \B))$.
Hence, 
$[\{ \pi(P_t) \}_{t \in [0,1]}] = 
[\{ 1_{\geq 0}(\pi(A_t)) \}_{t \in [0,1]}] = 0$ 
in $K_0(\C(C[0,1] \otimes \B))$.
Hence,
$$Ind_1(\{ A_t \}_{t \in [0,1]}) = 
\partial_1 ([ \{ 1_{\geq 0}(\pi(A_t)) \}_{t \in [0,1]}]) = 0 
\makebox{  in   } K_1(C[0,1] \otimes \B).$$\\   
 \end{proof}

By a slight modification of the argument of Proposition \ref{prop:VanishGlobalIndex}, one can show that under the hypotheses of Proposition
\ref{prop:VanishGlobalIndex}, $Ind_1(A_t) =0$ in $K_1(\B)$ for all $t \in [0,1]$, i.e., we have vanishing of the local index.\\

\section{Functorial properties of spectral flow}
\label{sec:FunctorialAxioms}

In this section, we will state some properties for spectral flow, and we will show later in Subsection \ref{subsect:Axiomatization} that
some of these properties give an axiomatization of spectral flow under appropriate extra hypotheses.

To avoid confusion, we here mention that we use the notation $Sf$ (capital $S$) to denote spectral flow, as defined 
in Definition \ref{df:SpectralFlow}, and we use the notation $sf$ (lower case $s$) to simply denote a general functor which may
or may not be $Sf$.

To state the Functoriality Axiom of spectral flow,
recall that if $\B$ and $\D$ are nonunital C*-algebras and $\phi : \B 
\rightarrow \D$ is a *-homomorphism that brings approximate units to  
approximate units, then $\phi$ has a unique strictly continuous extension to a 
unital  
*-homomorphism $\Mul(\B) \rightarrow \Mul(\D)$, which  
we will also denote by $\phi$  (see \cite{JensenThomsenBook} Corollary 1.1.15). We also get an induced unital *-homomorphism
$\overline{\phi} : \C(\B) \rightarrow \C(\D)$. 
 This is the basis of the Functoriality
Axiom in Definition \ref{df:SFAxioms} and Propositon 
\ref{prop:SfSatisfiesAxioms}.  

To state other functorial properties of spectral flow,  
we will also need some more notions which we put together in the following
definition:

\begin{df} \label{df:DfGather} 
Let $\B$ be a separable stable C*-algebra. 
\begin{enumerate}
\item Let
$\Ff(\Mul(\B))$ denote the set of Fredholm operators in $\Mul(\B)$,
i.e., the set of $A \in \Mul(\B)$
such that $\pi(A)$ is an invertible element of $\C(\B)$  (see Subsection \ref{subsect:GeneralizedFredholm}).
Let $\Ff_{SA}(\Mul(\B))$ denote the set of self-adjoint elements
of $\Ff(\Mul(\B))$.    $\Ff(\Mul(\B))$ and $\Ff_{SA}(\Mul(\B))$ are given the restriction
of the norm topology from $\Mul(\B)$.
Oftentimes, when the context is clear,
we drop the $\Mul(\B)$ and just write $\Ff$ and $\Ff_{SA}$ respectively. 
Also, we let $C([0,1], \Ff_{SA})$ denote the collection of all continuous maps
$[0,1] \rightarrow \Ff_{SA}$, i.e., 
the norm-continuous paths in $\Ff_{SA}$.
\item 
Let $\{ A_t \}_{t \in [0,1]}$ and $\{ B_t \}_{t \in [0,1]}$ be
two norm continuous paths of self-adjoint Fredholm operators
in $\Mul(\B)$ (i.e., two continuous paths 
in $\Ff_{SA}$), where $A_0, A_1, B_0$ and $B_1$
are invertible elements of $\Mul(\B)$.
We say that $\{ A_t \}_{t \in [0,1]}$ and $\{ B_t \}_{t \in [0,1]}$
are \emph{homotopic} (and write $\{ A_t \}_{t \in [0,1]} \sim_h 
\{ B_t \}_{t \in [0,1]}$) if 
there exists a norm-continuous family
 $\{ A_{s,t} \}_{(s,t) \in [0,1]\times [0,1]}$ 
in $\Ff_{SA}$ with $A_{s,0}$ and $A_{s,1}$ invertible
in $\Mul(\B)$ for all $s \in [0,1]$ such that 
$$A_{0,t} = A_t \makebox{ and  } A_{1,t} = B_t \makebox{  for all  }
t\in [0,1].$$
\item
Let $\{ A_t \}_{t \in [0,1]}$ and
$\{ B_t \}_{t \in [0,1]}$ be two norm-continuous paths of self-adjoint
Fredholm operators in $\Mul(\B)$ with invertible endpoints 
such that $A_1 = B_0$.
A \emph{concatenation} of the above two paths is defined to be a  
norm-continuous path $\{ C_t \}_{t \in [0,1]}$  
such that there exist 
$t_0 \in (0,1)$,  a  
homeomorphism $\omega_0 : [0, t_0] \rightarrow [0,1]$ and 
a homeomorphism $\omega_1 : [t_0, 1] \rightarrow [0,1]$ with $\omega_0(0) = 0 =
\omega_1(t_0)$ and $\omega_0(t_0) = 1 = \omega_1(1)$ such that  
\[
C_t = 
\begin{cases}
A_{\omega_0(t)}  & t \in [0, t_0] \\
B_{\omega_1(t)} & t \in (t_0, 1].
\end{cases}
\]
    
It is a short exercise to prove that $\{ C_t \}_{t \in [0,1]}$ is, up to 
homotopy 
(as defined in the previous item), 
independent of the choices of $t_0$, $\omega_0$ and $\omega_1$.
I.e., the concatenation $\{ C_t \}_{t \in [0,1]}$ is unique up to homotopy,
and we denote it by 
$$\{ C_t \}_{t \in [0,1]} = \{ A_t \}_{t \in [0,1]} * \{ C_t \}_{t \in 
[0,1]}.$$\\
\end{enumerate}
\end{df}

We are now ready to state our axioms for spectral flow.
Again, we remind the reader that we use the notation $sf$ (lower case $s$) to denote a general functor, as opposed 
to the notation $Sf$ (upper case $S$) which always denotes spectral flow as defined in Definition
\ref{df:SpectralFlow}.

\begin{df} \label{df:SFAxioms}
Suppose that for every separable stable C*-algebra $\B$, we have a map
$$sf : \{ \{ A_t \}_{t \in [0,1]} \in C([0,1],\Ff_{SA}) : A_0 \makebox{  and  } A_1 \makebox{  are
invertible in  } \Mul(\B) \} \rightarrow K_0(\B).$$  

Then $sf$ is said to satisfy Axiom SFj or (SFj) if it satisfies the following 
property (with the specific integer j) for every separable stable C*-algebra
$\B$:      
\begin{enumerate}
\item[(SF1)] (Functoriality in $\B$)  Suppose that
 $\D$ is another separable stable C*-algebra
and $\phi : \B \rightarrow \D$ is a *-homomorphism which brings approximate
units to approximate units. If $\{ A_t \}_{t \in [0,1]}$ is a
(norm-) continuous path in $\Ff_{SA}(\Mul(\B))$
with invertible endpoints then $\{ \phi(A_t) \}_{t \in [0,1]}$ is a
(norm-) continuous path in 
$\Ff_{SA}(\Mul(\D))$ with invertible endpoints,
and $$[\phi] \circ sf(\{ A_t \}_{t \in [0,1]}) = sf(\{ \phi(A_t) \}_{t \in [0,1]}) 
\in K_0(\D).$$ 
\item[(SF2)] (Homotopy Axiom) 
Suppose that $\{ A_{0, t} \}_{t \in [0,1]}$ and 
$\{ A_{1,t} \}_{t \in [0,1]}$ are two homotopic (norm-) continuous
paths in $\Ff_{SA}$ with invertible endpoints.  
Then
$$sf(\{ A_{0,t} \}_{t \in [0,1]}) = sf(\{ A_{1,t} \}_{t \in [0,1]}) 
\makebox{  in  } K_0(\B).$$ 
\item[(SF3)] (Concatenation)  Let $\{ A_t \}_{t \in [0,1]}$  and 
$\{ B_t \}_{t \in [0,1]}$ be two (norm)-continuous paths of self-adjoint
Fredholm operators in $\Mul(\B)$ with invertible endpoints such that
$A_1 = B_0$.
Then 
$$sf(\{ A_t \}_{t \in [0,1]} * \{ B_t \}_{t \in [0,1]}) =
sf(\{ A_t \}_{t \in [0,1]}) + sf(\{ B_t \}_{t \in [0,1]}).$$
Here, $\{ A_t \}_{t \in [0,1]} * \{ B_t \}_{t \in [0,1]}$ 
is \emph{concatenation} of paths.  See Definition \ref{df:DfGather} above. 
\item[(SF4)] (Normalization) Let $P, Q \in \Mul(\B)$ be projections such 
that $P - Q \in \B$ and $P \sim 1 \sim 1 - P$.
Then $$sf(\{ (1-t) (2P-1) + t (2Q - 1) \}_{t \in [0,1]})  = [P:Q].$$\\ 
\end{enumerate}
\end{df}

\begin{rmk}
By Lemma \ref{lem:AShortComputation}, the projection $Q$, in the
Normalization Axiom (SF4) above, must satisfy that
$$Q \sim 1 \sim 1 - Q.$$\\
\end{rmk}

\begin{prop}  \label{prop:TrivialityPrinciple}
(Triviality Principle; c.f. \cite{LeschUniqueness} Lemma 5.3)
Suppose that for every separable stable C*-algebra $\B$ we have a map
$$sf : \{ \{ A_t \}_{t \in [0,1]} \in C([0,1], \Ff_{SA}) : A_0 \makebox{ and  }
A_1 \makebox{  are invertible in } \Mul(\B) \} \rightarrow K_0(\B).$$

Suppose that $sf$ satisfies the Homotopy Axiom (SF2) and the Concatenation
Axiom (SF3).     

Then for any separable stable C*-algebra $\B$, for any norm-continuous path
$\{ C_t \}_{t \in [0,1]}$ of self-adjoint invertible elements of $\Mul(\B)$,
$$sf(\{ C_t \}_{t \in [0,1]}) = 0 \makebox{  in  } K_0(\B).$$
\end{prop}

\begin{proof}
The argument is already in \cite{LeschUniqueness} Lemma 5.3.  We provide
the short proof for the convenience of the reader.\\  

\noindent Case 1: Suppose that $\{ C_t \}_{t \in [0,1]}$ is constant, i.e., suppose
that $C_t = C_0$ for all $t \in [0,1]$, where $C_0 \in \Mul(\B)$ is a 
self-adjoint invertible.

Then, $\{ C_t \}_{t \in [0,1]} * \{ C_t \}_{t \in [0,1]}
= \{ C_t \}_{t \in [0,1]}$ (and this is independent of our choices for the concatenation).  Hence, by (SF3),
$$sf(\{ C_t \}_{t \in [0,1]}) = sf(\{ C_t \}_{t \in [0,1]} * \{ C_t \}_{t \in [0,1]}) = sf(\{ C_t \}_{t \in [0,1]}) + sf(\{ C_t \}_{t \in [0,1]})$$
in $K_0(\B)$.  
Then $$sf(\{ C_t \}_{t \in [0,1]}) = 0  \makebox{  in  } K_0(\B).$$\\

\noindent Case 2 (General Case):  Suppose that $\{ C_t \}_{t \in [0,1]}$ is
an arbitrary norm-continuous path of self-adjoint invertible elements of
$\Mul(\B)$.
$\{ C_t \}_{t \in [0,1]}$ is homotopic 
to the constant path $\{ D_t \}_{t \in [0,1]}$,
where $D_t =_{df} C_0$ for all $t \in [0,1]$.  
Hence, by Axiom (SF2) and by Case 1,
$$sf(\{ C_t \}_{t \in [0,1]}) = sf(\{ D_t \}_{t \in [0,1]}) = 0 \makebox{  in  }
K_0(\B).$$\\
\end{proof}

To prove that $Sf$ satisfies the Functoriality Axiom, we need the following
lemma:      

\begin{lem} \label{lem:FunctorialityOfInd}
Let $\B$ and $\D$ be separable stable C*-algebras, and let
$\phi : \B \rightarrow \D$ be a *-homomorphism that brings an approximate
of $\B$ to an approximate unit of $\D$.

Let $A \in \Mul(\B)$ be a self-adjoint Fredholm operator. 
Then $\phi(A)$ is a self-adjoint Fredholm operator in $\Mul(\D)$.

Suppose, in addition, that  
$b_1, b_2 \in \B_{SA}$ are trivializing operators for $A$.
     
Then $\phi(b_1)$ and $\phi(b_2)$ are trivializing operators for $\phi(A)$
and  
$$Ind(\phi(A), \phi(b_1), \phi(b_2)) = [\phi] (Ind(A, b_1, b_2))
\makebox{  in  } K_0(\D).$$ 
\end{lem}

\begin{proof}
Recall that we are using the notation
$\phi : \Mul(\B) \rightarrow \Mul(\D)$
to also denote the unique unital strictly continuous *-homomorphism which extends $\phi : \B \rightarrow \D$, and
$\overline{\phi} : \C(\B) \rightarrow \C(\D)$ is the induced 
unital *-homomorphism. 

Since $\pi(A)$ is invertible in $\C(\B)$ and $\overline{\phi}$ is unital, 
$\overline{\phi}(\pi(A)) = 
\pi \circ \phi(A)$ is invertible in $\C(\D)$, and hence,
$\phi(A)$ is a self-adjoint Fredholm operator in $\Mul(\D)$.

Suppose that for $j = 1,2$, $b_j \in \B_{SA}$ and
  $A + b_j$ is a self-adjoint invertible element of 
$\Mul(\B)$.  Then $\phi(b_j) \in \D_{SA}$, and since $\phi$ is unital,
$\phi(A + b_j) = \phi(A) + \phi(b_j)$ is a self-adjoint invertible element
of $\Mul(\D)$.  Hence, $\phi(b_j) \in \D$ is a trivializing operator
for $\phi(A)$.

By the Functoriality of essential codimension (Proposition \ref{prop:ECProperties} (EC1)),
$$[\phi](Ind(A, b_1, b_2)) = [\phi]([1_{\geq 0}(A + b_1) : 1_{\geq 0}(A + 
b_2)]) = [\phi(1_{\geq 0}(A + b_1)) : \phi(1_{\geq 0}(A + b_2))].$$
But by the continuous functional calculus,
$$[\phi(1_{\geq 0}(A + b_1)) : \phi(1_{\geq 0}(A + b_2))] =
[1_{\geq 0}(\phi(A + b_1)) : 1_{\geq 0}(\phi(A + b_2))]    
= Ind(\phi(A), \phi(b_1), \phi(b_2)).$$\\ 
\end{proof}

\begin{prop} \label{prop:SfSatisfiesAxioms}
The spectral flow $Sf$, as defined in 
Definition \ref{df:SpectralFlow}, satisfies Axioms (SF1) to (SF4).
\end{prop}

\begin{proof}
Let $\B$ be a separable stable C*-algebra.\\

\emph{Proof of (SF1) Functoriality:} 
Suppose that $\D$ is another separable stable C*-algebra, 
and suppose
that $\phi : \B \rightarrow \D$ is a *-homomorphism that brings an approximate
unit of $\B$ to an approximate unit of $\D$.   
Let $\{ A_t \}_{t \in [0,1]}$ be a norm-continuous path of self-adjoint 
Fredholm operators
in $\Mul(\B)$ such that $A_0$ and $A_1$ are invertible elements 
of $\Mul(\B)$.                          
By Lemma \ref{lem:FunctorialityOfInd},  
$\{ \phi(A_t) \}_{t \in [0,1]}$ is a norm-continuous path of self-adjoint
 Fredholm operators
in $\Mul(\D)$. (Recall that $\phi$ extends uniquely to a strictly
continuous unital *-homomorphism $\Mul(\B) \rightarrow \Mul(\D)$, which we
also denote by $\phi$.)
Also, since $A_0$ and $A_1$ are invertible in $\Mul(\B)$,
$\phi(A_0)$ and $\phi(A_1)$ are invertible in $\Mul(\D)$.   
So $Sf(\{ \phi(A_t) \}_{t \in [0,1]})$ is defined by Definition \ref{df:SpectralFlow}  (see also
Proposition \ref{July1120231AM}).  

Hence, let $$0 = t_0 < t_1 < ... < t_n = 1$$
be a partition of $[0,1]$ and for all $0 \leq j \leq n-1$, let 
$\{ b^j_t \}_{t \in [t_j, t_{j+1}]}$ (a continuous path) in 
$\B_{SA}$
 be a trivializing family for $\{ A_t \}_{t
\in [t_j, t_{j+1}]}$. By Lemma \ref{lem:FunctorialityOfInd},
for all $j$ and $t \in [t_j, t_{j+1}]$,  
$\phi(A_t) +\phi(b^j_t)$ 
is      
invertible in $\Mul(\D)$.  Moreover, for all $j$, the map
$[t_j, t_{j+1}] \rightarrow  \Mul(\D) : t \mapsto 
\phi(A_t)       + \phi(b^j_t)$
is norm-continuous. 
Hence, for all $j$, $\{ \phi(b^j_t) \}_{t \in [t_j, t_{j+1}]}$ is a trivializing
family for $\{ \phi(A_t)   \}_{t \in [t_j, t_{j+1}]}$. 

Hence, we have that 
\begin{eqnarray*}
& & [\phi] \circ Sf(\{ A_t \}_{t\in [0,1]})\\
& = & [\phi] \circ Ind(A_0, 0, b^0_0) + 
[\phi] \circ Ind(A_1, b^{n-1}_1, 0) \\
& & + \sum_{j=1}^{n-1} [\phi] \circ Ind(A_{t_j},
b^{j-1}_{t_j},
b^j_{t_j}) \makebox{  (by Definition \ref{df:SpectralFlow}).}\\
& = &  Ind(\phi(A_0), 0, \phi(b^0_0)) 
+  Ind(\phi(A_1),  
\phi(b^{n-1}_1), 0) \\
& & + \sum_{j=1}^{n-1} Ind(\phi(A_{t_j}), 
\phi(b^{j-1}_{t_j}),
\phi(b^j_{t_j})) 
\makebox{  (by Lemma \ref{lem:FunctorialityOfInd}).}\\   
& = & Sf(\{ \phi(A_t) \}_{t \in [0,1]}).\\ 
\end{eqnarray*}

\emph{Proof of (SF2) Homotopy Axiom:}   
Suppose that $\{ A_{s,t} \}_{(s,t) \in [0,1] \times [0,1]}$ is 
a norm-continuous 
family of self-adjoint Fredholm operators in $\Mul(\B)$ such that for all
$s \in [0,1]$, $A_{s,0}$ and $A_{s,1}$ are invertible elements of 
$\Mul(\B)$.     

For all $t \in [0,1]$, let $A_t =_{df} \{ A_{s, t} \}_{s \in [0,1]}$.
Then  $\{ A_t \}_{t \in [0,1]}$ is a norm-continuous path of self-adjoint
Fredholm operators in $\Mul(C[0,1] \otimes \B)$, and 
$A_0$ and $A_1$ are invertible elements of $\Mul(C[0,1] \otimes \B)$.
For all $s \in [0,1]$, let $ev_s : C[0,1] \otimes \B \rightarrow \B : f 
\mapsto f(s)$ be the evaluation at $s$  *-homomorphism.  Clearly, $ev_s$
brings approximate units to approximate units.  
Also, for all $s, t \in [0,1]$, $ev_s(A_t) = A_{s,t}$.  
For all $s \in [0,1]$, we
have the induced group homomorphism $[ev_s] : K_0(C[0,1] \otimes \B) \rightarrow K_0(\B)$.
By Axiom (SF1), the Functoriality Axiom, we have that   
$$[ev_s]\circ Sf(\{ A_t \}_{t \in [0,1]}) = Sf(\{ A_{s,t} \}_{t \in [0,1]})
\makebox{  in  } K_0(\B).$$
Since $\{ ev_s \}_{s \in [0,1]}$ is a (pointwise-norm continuous) homotopy
between the  maps $ev_0$ and $ev_1$ and since $K_0$ is homotopy-invariant, we have that 
$$Sf(\{ A_{0, t} \}_{t \in [0,1]}) = Sf( \{ A_{1,t} \}_{t \in [0,1]}).$$\\

\emph{Proof of (SF3) Concatenation Axiom:}  This is straightforward from
the definitions.\\

\emph{Proof of (SF4) Normality Axiom:}  
This follows from Proposition \ref{prop:SecondExample}.\\

\end{proof}

\section{The Pre Spectral Flow isomorphism} 
\label{sect:PreSFIsomorphism}

A central result in the paper \cite{AtiyahSingerSkew} is that
$\Ff_{SA, *}(\mathbb{B}(l_2))$ (see Definition \ref{df:FfSADefinitions})
is a  classifying space for 
the functor 
$K^1$.  There are, of course, other realizations for a classifying space
for $K^1$ (e.g., $U(\infty)$), but $\Ff_{SA, *}$ is important, for instance,
for the index theorem for families of self-adjoint elliptic operators
(e.g., see \cite{AtiyahPatodiSinger3} Theorem 3.4 and the discussion 
before that).  A consequence of the result from \cite{AtiyahSingerSkew} is
that $\pi_1(\Ff_{SA,*}(\mathbb{B}(l_2))) \cong \mathbb{Z} = K_0(\K)$, which
we call the ``Pre Spectral Flow Isomorphism". And 
in \cite{AtiyahPatodiSinger3} Section 7 (just before Theorem 7.4; see
also \cite{Phillips1996} the last theorem in the paper), 
it was shown that the above isomorphism is induced by spectral flow,  
and we call this the ``Spectral Flow Isomorphism Theorem".  

In the interesting papers \cite{PereraPaper} and \cite{PereraThesis} (using ideas
from \cite{AtiyahSingerSkew} and \cite{AtiyahPatodiSinger3}), 
the Pre Spectral Flow
Isomorphism was generalized to type $II_{\infty}$ factors with separable
predual and multiplier algebras of the form
$\Mul(\B)$ where $\B$ is a separable stable
C*-algebra with an approximate unit consisting of projections; we focus 
on the latter $\Mul(\B)$ result. 
For the above Pre Spectral Flow isomorphism 
$PSf: \pi_1(\Ff_{*,\infty}(\Mul(\B))) \cong K_0(\B)$, on the
basis that in \cite{AtiyahPatodiSinger3}, it was shown that $PSf$ was 
induced by spectral flow for the case $\B = \K$, 
and also on the basis of the computation of some
intuitively reasonable concrete examples, it was proposed, 
in \cite{PereraThesis}, that $PSf$ (with some modification to include
paths that are not loops) constitutes a viable definition of spectral
flow for $\Mul(\B)$. 
While this proposal has many merits and will certainly be useful and
important for future progress, 
a defect is that $PSf$ is defined abstractly, using Bott Periodicity
among other things, and does not capture the original             
Atiyah--Lusztig--Patodi--Singer concrete intuition that spectral  
flow roughly 
measures the ``net mass of the part of the spectrum that passes through zero" 
as we move along the continuous path of self-adjoint Fredholm operators.
 
  In this section, we will generalize the Pre Spectral Flow Isomorphism
to multiplier algebras $\Mul(\B)$ where $\B$ need not have a nonzero projection.
Then, in the next section, we will prove that the above Pre Spectral Flow
Isormorphism is induced by our concrete definition of Spectral Flow given
in Definition \ref{df:SpectralFlow} (Spectral Flow Isomorphism Theorem), recovering the original concrete
intuition of
\cite{AtiyahPatodiSinger3}. In the process, we will show  
that for every invertible $X \in \Ff_{SA, \infty}(\Mul(\B))$,
$\Omega_X \Ff_{SA, \infty}(\Mul(\B))$
is a classifying space for the functor $K \mapsto K_0(C(K) \otimes \B)$, 
generalizing the result from
\cite{AtiyahSingerSkew}.  Our proofs use ideas from 
the semifinite factor case of 
\cite{PereraPaper}, \cite{PereraThesis}, \cite{AtiyahSingerSkew}, 
but with many substantial modifications.
A certain amount of 
homotopy theory is present in this part of the argument. Our goal is to provide a clean and
logically complete proof of the Spectral Isomorphism Theorem in a 
general context, which is 
amenable for students from analysis.\\\\

\textbf{\emph{We fix two  notations that will apply for this 
whole section:
Throughout this whole section, $\B$ will be a separable stable C*-algebra,
and $P_0 \in \Mul(\B)$ will be a projection such that   
$P_0 \sim 1_{\Mul(\B)} \sim 1_{\Mul(\B)} - P_0$. We refer to these conventions
by ``($\Lambda$)".}}\\\\

\subsection{Fredholm operators, topological groups of invertible operators,
 and Grassmannian spaces}

Following \cite{AtiyahSingerSkew}, \cite{PereraPaper} and
\cite{PereraThesis}, we begin with an analysis of the path-components of the space of self-adjoint
Fredholm operators in a multiplier algebra.   
We then study
 the relationships between the relevant space of Fredholm operators, a group of invertible
operators   
and a certain Grassmannian space, culminating in the first step in the construction of
the Spectral flow isomorphism.\\

Recall that $\Ff$ or $\Ff(\Mul(\B))$ denotes the set of Fredholm operators in $\Mul(\B)$, i.e., 
$\Ff  =_{df} \{ A \in \Mul(\B) : \pi(A) \makebox{ is invertible in } \C(\B) \}$
(see Subsection \ref{subsect:GeneralizedFredholm}).
Recall that 
$\Ff_{SA}$ or $\Ff_{SA}(\Mul(\B))$ consists of the self-adjoint elements of $\Ff$, i.e., 
$\Ff_{SA}$ consists of self-adjoint Fredholm operators in $\Mul(\B)$.  All these spaces
are given the restriction of the norm topology from $\Mul(\B)$.
Now also recall that for a C*-algebra $\C$ and $x \in \C$, $sp(x)$ is 
our notation for the spectrum of $x$.
We now introduce four more definitions:
\begin{df}  \label{df:FfSADefinitions}
$$\Ff_{SA, +} = \Ff_{SA, +}(\Mul(\B))
 =_{df} \{ A \in \Ff_{SA} : sp(\pi(A)) \subset (0, \infty) \}.$$
$$\Ff_{SA, -} = \Ff_{SA, -}(\Mul(\B))  
 =_{df} \{ A \in \Ff_{SA} : sp(\pi(A)) \subset (-\infty, 0) \}.$$
$$\Ff_{SA, *} = \Ff_{SA, *}(\Mul(\B)) =_{df} \Ff_{SA} - (\Ff_{SA, +} \cup \Ff_{SA, -}).$$
$$\Ff_{SA, \infty} = \Ff_{SA, \infty}(\Mul(\B)) =_{df}
\{ A \in \Ff_{SA, *} : 1_{\geq 0}(\pi(A)) \sim 1_{\C(\B)}
\sim 1_{\geq 0}(-\pi(A)) \}.$$ 
All the above spaces are given the restriction of the norm topology 
from $\Mul(\B)$.\\ 
\end{df}

Clearly, we have a (pairwise) disjoint union
$$\Ff_{SA} = \Ff_{SA, +} \sqcup \Ff_{SA, -} \sqcup \Ff_{SA, *}.$$ 
In the works of \cite{AtiyahSingerSkew}, \cite{BenPhillEtAl} 
and \cite{PereraPaper},  the object of interest is $\Ff_{SA,*}$.  In fact, the above decomposition partially generalizes 
\cite{AtiyahSingerSkew} Theorem B.  However,
in this paper, we will instead focus on the smaller topological space $\Ff_{SA, \infty}$.   This is because
the multiplier algebras in \cite{AtiyahSingerSkew}, \cite{BenPhillEtAl}
and \cite{PereraPaper}, $\mathbb{B}(l_2)$ and $\Mul$ (where $\Mul$ is a type $II_{\infty}$ factor with separable predual), have relatively simple
structure -- for example,  each of $\mathbb{B}(l_2)$ and $\Mul$ has a unique proper nontrivial ideal.  On
the other hand, a general multiplier algebra $\Mul(\B)$ can have infinitely 
many ideals.  Moreover, unlike the case of $\mathbb{B}(l_2)$, for general
multiplier algebras, $\Ff_{SA,*}$ need not be a path-connected component
of $\Ff_{SA}$ (see this paper Lemma 4.6 and compare it with
\cite{AtiyahSingerSkew} Theorem B).  All these (and other) reasons
 necessitate working
instead with the object $\Ff_{SA, \infty}$.  (E.g., see the generalized
Spectral Flow Isomorphism Theorem \ref{thm:GeneralSFIsomorphism}.)

Recall also, that for a unital C*-algebra $\C$, 
$GL(\C)$ denotes its group of invertibles (or
``general linear group").  
We also let $GL(\C)_{SA}$ denote the self-adjoint elements of 
$GL(\C)$ (i.e., the self-adjoint invertible elements of $\C$), and 
$GL(\C)_+$ and $GL(\C)_-$ the positive and negative elements of $GL(\C)$
respectively.
We let $GL(\C)_{SA,*}$ denote the $x \in GL(\C)_{SA}$ such that $sp(x)$ contains both strictly
positive and strictly negative real numbers.
Finally, $GL(\C)_{SA, \infty}$ denotes the elements $x \in GL(\C)_{SA}$
such that $1_{\geq 0}(x) \sim 1_{\C} \sim 1_{\geq 0}(-x)$.   
All the above are topological spaces with the restriction of the norm topology of $\C$.\\

 The 
first several results are basic exercises, but we nonetheless provide some proofs 
for the convenience
of the reader.

\begin{lem}
Let $\C$ be a unital C*-algebra.
Then $GL(\C)_+$ and $GL(\C)_-$ are both (norm-) contractible topological spaces which
are closed in $GL(\C)$ and clopen in $GL(\C)_{SA}$.    

$GL(\C)_{SA, \infty}$ is (norm-) closed in $GL(\C)$ and clopen in 
$GL(\C)_{SA}$.  

$GL(\C(\B))_{SA, \infty}$ is a (norm-) path-connected topological space which is
closed in $GL(\C(\B))$ and clopen in $GL(\C(\B))_{SA}$.   
\label{lem:May220231AM}  
\end{lem} 
\begin{proof}
We prove the second and third paragraphs of the lemma, leaving the first paragraph (which is basic operator theory) as a warm-up 
exercise for the reader.

Let $\{ X_n \}$ be a sequence in $GL(\C)_{SA, \infty}$ and $X \in GL(\C)$ such that $X_n \rightarrow X$.
Since $X_n$ is self-adjoint for all $n$, $X$ is self-adjoint;  so 
$X   \in GL(\C)_{SA}$.  
Also, by Lemma \ref{lem:Aug520231am},  $1_{\geq 0}(X_n) \rightarrow 1_{\geq 0}(X)$.   Hence, by \cite{WeggeOlsen} Propositions 5.2.6 and 5.2.10,
 for sufficiently large $n$,
$1_{\geq 0}(X) \sim 1_{\geq 0}(X_n) \sim 1$ and $1 - 1_{\geq 0}(X) \sim 1 - 1_{\geq 0}(X_n)
\sim 1$.   Hence, $X \in GL(\C)_{SA, \infty}$.  Since $\{ X_n \}$ and $X$ were arbitrary,
$GL(\C)_{SA, \infty}$ is closed in $GL(\C)$.

Since $GL(\C)_{SA, \infty}$ is closed in $GL(\C)$, $GL(\C)_{SA, \infty}$ is closed 
in $GL(\C)_{SA}$ (the set of self-adjoint elements of $GL(\C)$).  

That $GL(\C)_{SA, \infty}$ is open in $GL(\C)_{SA}$ follows immediately from Lemma
\ref{lem:Aug520232am} and from \cite{WeggeOlsen} Propositions 5.2.6 and
5.2.10.

Let $X, Y \in GL(\C(\B))_{SA, \infty}$.   Let $P =_{df} 1_{\geq 0}(X)$ and $Q =_{df} 1_{\geq 0}(Y)$.  (So
$1 - P = 1_{\geq 0}(-X)$ and $1 - Q =_{df} 1_{\geq 0}(-Y)$.)  By applying the continuous functional calculus,
we see that $X \sim_h P + -(1-P)$ and $Y \sim_h Q-(1-Q)$  in $GL(\C(\B))_{SA, \infty}$.  
By the definition of $GL(\C(\B))_{SA, \infty}$, $P \sim 1 \sim Q \sim 1-P \sim 1 - Q$.  Hence, let $V, W \in \C(\B)$ be partial isometries
such that $V^* V = P$, $V V^* = Q$, $W^* W = 1 - P$ and $W W^* = 1 - Q$.   Then $U' =_{df} V + W$ is 
a unitary in $\C(\B)$ such that $U' (P + -(1-P)){U'}^* = Q + -(1 - Q)$.  Since $P \sim 1$ and $\C(\B)$ is properly
infinite, let $u \in P \C(\B) P$ be a unitary such that $[U'] = [u^*]$ in $K_0(\C(\B))$.
Let $U =_{df} U' (u + (1 - P)) \in U(\C(\B))$.  Then $[U] = 0$ in $K_1(\C(\B))$.   Since $\C(\B)$ is 
$K_1$-injective (see Lemma \ref{lem:CoronaK1InjAndSurj} and Subsection
\ref{subsect:CoronaK1InjAndSurj} in general), 
$U \sim_h 1$ in $U(\C(\B))$. 
Also, 
\begin{eqnarray*}
U (P- (1-P)) U^* & = & U' (u + (1-P))(P - (1-P)) (u^* + (1-P)){U'}^* \\
 & = & U' (P - (1-P)) {U'}^*   \\  
& = & Q - (1-Q).
\end{eqnarray*} 
Hence, $P + -(1-P) \sim_h Q + -(1-Q)$ in $GL(C(\B))_{SA, \infty}$.
Hence, $X \sim_h Y$ in $GL(\C(\B))_{SA, \infty}$.   Since $X, Y$ are arbitrary, $GL(\C(\B))_{SA, \infty}$ is
(norm-) path-connected.\\ 
\end{proof}

\begin{rmk}
The arguments of Lemma \ref{lem:May220231AM} actually show that
$GL(\C)_{SA, \infty}$ is an open subset of $\C_{SA}$. Similarly,  
$GL(\C)_+$ and  $GL(\C)_-$ are open subsets of $\C_{SA}$.\\ 
\end{rmk}

\begin{lem} (Cf. \cite{AtiyahSingerSkew} Theorem B,  \cite{PereraPaper}
Proposition 3.1 and \cite{PereraThesis} Lemma 3.3.3)
 \label{lem:May220232AM}  $\Ff_{SA, +}$ and $\Ff_{SA,-}$ are (norm-) 
contractible topological spaces which are closed in $\Ff$ and clopen in 
$\Ff_{SA}$.
\end{lem}  
\begin{proof}
We prove the lemma for $\Ff_{SA,+}$, leaving the (similar) arguments for $\Ff_{SA,-}$ as an exercise for the reader.

Let $F : \Ff_{SA, +} \times [0,1] \rightarrow \Mul(\B)_{SA}$ be the continuous map given by
$$F(A,t)=_{df}  (1-t) A + t 1_{\Mul(\B)} \makebox{  for all  } t \in [0,1] \makebox{  and  }  A \in 
\Ff_{SA, +}.$$

Firstly, for all $t \in [0,1]$ and $A \in \Ff_{SA,+}$,
$$\pi(F(A,t)) =  (1-t) \pi(A) + t 1_{\C(\B)}$$
which is a positive invertible element of $\C(\B)$, since 
$A \in \Ff_{SA, +}$.   Hence, for all $t \in [0,1]$ and $A \in \Ff_{SA,+}$, 
$F(A,t) \in \Ff_{SA, +}$.      Hence, $ran(F) \subseteq \Ff_{SA, +}$.

Also, note that $F(.,0) = id_{\Ff_{SA, +}}$ and $F(., 1) = 1$.  Hence, $F$ is a homotopy between the 
map $id_{\Ff_{SA, +}}$ and the constant map $\Ff_{SA,+} \rightarrow \Ff_{SA, +} : C \mapsto 1$.
I.e., $\Ff_{SA,+}$ is contractible.

Let $A \in \Ff_{SA, +}$ be arbitrary.  Then $\pi(A)$ is a positive invertible element of $\C(\B)$.  Hence, choose $r > 0$ so that
$sp(\pi(A)) \subset (r, \infty)$.  Hence, by a standard spectral theory argument, we can choose $\delta > 0$ so that whenever $C \in \C(\B)$ is a self-adjoint element such that $\| C - \pi(A) \| < \delta$, then $sp(C) \subset (\frac{r}{2}, \infty)$; in particular, such a $C$ is positive and invertible in $\C(\B)$.   Let $D \in B(A, \delta)_{SA} $ (i.e., $D$ is self-adjoint and  $\| D - A \| < \delta$) be arbitrary.   Hence,
$\| \pi(D) - \pi(A) \| < \delta$.  Hence, by our choice of $\delta$, $\pi(D)$ is a positive invertible element of $\C(\B)$. Hence,
$D \in \Ff_{SA,+}$.  Since $D$ is arbitrary, $B(A, \delta)_{SA} \subset \Ff_{SA, +}$, i.e., $A$ is an interior point of 
$\Ff_{SA,+}$. Since $A$ is arbitrary, $\Ff_{SA, +}$  is an open subset of $\Ff_{SA}$.

We leave the proofs of the remaining statements as exercises for the reader.\\

\end{proof}

The quotient map $\pi : \Mul(\B) \rightarrow \C(\B)$ restricts to quotient  maps (with the same name) 
$\pi: \Ff \rightarrow GL(\C(\B))$,
$\pi : \Ff_{SA} \rightarrow GL(\C(\B))_{SA}$,  
$\pi : \Ff_{SA,*} \rightarrow GL(\C(\B))_{SA,*}$ and 
$\pi : \Ff_{SA, \infty} \rightarrow GL(\C(\B))_{SA, \infty}$, which are
all (norm-) continuous.\\

\begin{lem}
\label{lem:May320232AM}  
The quotient maps $\pi : \Ff \rightarrow GL(\C(\B))$,
$\pi : \Ff_{SA} \rightarrow GL(\C(\B))_{SA}$,  
$\pi : \Ff_{SA,*} \rightarrow GL(\C(\B))_{SA,*}$, and
$\pi : \Ff_{SA, \infty} \rightarrow GL(\C(\B))_{SA, \infty}$ are all  
homotopy equivalences, where all the spaces are given the restriction of the norm topology from the ambient 
C*-algebra ($\Mul(\B)$ or $\C(\B)$).  
\end{lem}

\begin{proof}
This argument is similar to that of \cite{AtiyahSingerSkew} Lemma 2.3
(see also \cite{PereraPaper} Lemma 3.3 and \cite{PereraThesis} Lemma 3.3.4). 
We will prove that the map
$\pi : \Ff_{SA, \infty} \rightarrow GL(\C(\B))_{SA,\infty}$
is a homotopy equivalence.  The proofs, for the other maps, are similar
and easier.   

Since $\pi : \Mul(\B) \rightarrow \C(\B)$ is a surjective continuous linear map, by the
Bartle--Graves Selection Theorem 
(see \cite{Michael} Proposition 7.2; see also
the corollary, in the introduction, to \cite{Michael} Theorem 3.2''), 
$\pi$ has a continuous right inverse, i.e., there
is a (norm-) continuous (not necessarily linear) map 
$s_1 : \C(\B) \rightarrow \Mul(\B)$
for which $\pi \circ s_1 = id_{\C(\B)}$.    
Let $s : \C(\B) \rightarrow \Mul(\B)$ be the (norm-) continuous map given by
$$s(x) =_{df} \frac{s_1(x) + s_1(x)^*}{2} \makebox{  for all  } x \in \C(\B).$$
Clearly, $s$ restricts to a map (with the same name)
$s : GL(\C(\B))_{SA, *} \rightarrow \Ff_{SA,*}$  
and
$\pi \circ s = id_{GL(\C(\B))_{SA,*}}$. 
And it is also clear, by the definition of $\Ff_{SA, \infty}$,
 that $s$ restricts further to a map (with the same
name)
$s : GL(\C(\B))_{SA, \infty} \rightarrow \Ff_{SA, \infty}$ and
$$\pi \circ s = id_{GL(\C(\B))_{SA, \infty}}.$$

\noindent \emph{Claim:}  There is a homotopy 
$F : \Ff_{SA,\infty} \times [0,1] \rightarrow \Ff_{SA,\infty}$ 
between $id_{\Ff_{SA,\infty}}$ 
and $s \circ \pi$ which is given by 
$$F(A, t) =_{df} (1 - t) A + t s \circ \pi(A) \makebox{  for all  } t\in [0,1] 
\makebox{  and  } A \in \Ff_{SA, \infty}.$$ 
\noindent\emph{Proof of Claim:}   
It suffices to prove that $ran(F) \subseteq \Ff_{SA, \infty}$. 
(The remaining items are
trivial.)  

Clearly, $ran(F) \subseteq \Mul(\B)_{SA}$.   
Let $A \in \F_{SA, \infty}$ and $t \in [0,1]$ be given.  
Hence, $\pi(A) \in GL(\C(\B))_{SA,\infty}$.   
Also, since $s$ is a right inverse of $\pi$,
$\pi \circ s \circ \pi(A) = \pi(A)$.
So   
\begin{eqnarray*}
\pi(F(A,t)) & = & \pi((1-t)A + t s \circ \pi(A)) \\
& = & (1-t) \pi(A) + t \pi \circ s \circ \pi(A) \\
& = & (1-t)\pi(A) + t \pi(A)\\
& = & \pi(A) \in GL_{SA, \infty}(\C(\B)).
\end{eqnarray*}
Hence, $F(A, t) \in \Ff_{SA, \infty}$.
Since $A$ and $t$ were arbitrary, $ran(F) \subseteq \Ff_{SA, \infty}$.\\ 
This completes the proof of the Claim and the whole result.\\
\end{proof}

The next result again partly generalizes \cite{AtiyahSingerSkew} Theorem B. 
Note that unlike the case of $\mathbb{B}(l_2)$, for a general multiplier
algebra $\Mul(\B)$, $\Ff_{SA, *}(\Mul(\B))$ need not be a path-connected
component of $\Ff_{SA}(\Mul(\B))$.

\begin{lem} \label{lem:Aug520235AM}
$\Ff_{SA, \infty}$ is a path-connected topological space which is
closed in $\Ff$ and clopen in $\Ff_{SA}$.
\end{lem}

\begin{proof}
By Lemma \ref{lem:May220231AM},  $GL(\Mul(\B))_{SA, \infty}$ is (norm-) path-connected.   
By Lemma \ref{lem:May320232AM}, $GL(\Mul(\B))_{SA, \infty}$ is homotopy-equivalent to
$\Ff_{SA, \infty}$.   Hence, $\Ff_{SA, \infty}$ is path-connected.

Let $\{ X_n \}$ be a sequence in $\Ff_{SA, \infty}$ and $X \in \Ff$ such that $X_n \rightarrow X$.
By a standard argument, $X$ is self-adjoint.
Now $\pi(X_n) \rightarrow \pi(X)$ and $\pi(X)$ is a self-adjoint invertible element of $\C(\B)$.
By Lemma \ref{lem:Aug520231am},  
$$1_{\geq 0}(\pi(X_n)) \rightarrow  1_{\geq 0}(\pi(X))$$
and
$$1 - 1_{\geq 0}(\pi(X_n)) \rightarrow 1 - 1_{\geq 0}(\pi(X)).$$
But since $X_n \in \Ff_{SA, \infty}$, $$1_{\geq 0}(\pi(X_n))  \sim 1 \sim 1 - 
1_{\geq 0}(\pi(X_n))$$ for all $n$.    
Hence,
$$1_{\geq 0}(\pi(X))  \sim 1 \sim 1 - 
1_{\geq 0}(\pi(X)).$$
Hence, $X \in \Ff_{SA, \infty}$.   Since $\{ X_n \}$ and $X$ are arbitrary, $\Ff_{SA, \infty}$ is closed in $\Ff$.
Hence, $\Ff_{SA, \infty}$ is also closed in $\Ff_{SA}$.

Let $X \in \Ff_{SA, \infty}$ be arbitrary.   Hence, $\pi(X)$ is a self-adjoint invertible element of $\C(\B)$ such that
$$1_{\geq 0}(\pi(X)) \sim 1 \sim 1 - 1_{\geq 0}(\pi(X)).$$ 
By Lemma \ref{lem:Aug520232am},  choose $\delta > 0$ so that if $y \in \C(\B)$ is self-adjoint and
$\| \pi(X) - y \| < \delta$ then $y$ is invertible and $\| 1_{\geq 0}(\pi(X)) - 1_{\geq 0}(y) \| < \frac{1}{2}$.  

Now suppose that $Y \in B(X, \delta)_{SA}$ (i.e., $Y$ is self-adjoint and $\| X - Y \| < \delta$) is given.
Hence, $\| \pi(X) - \pi(Y) \| < \delta$.  Hence, by our choice of $\delta$, $\pi(Y)$ is invertible and 
$\| 1_{\geq 0}(\pi(X)) - 1_{\geq 0}(\pi(Y)) \| < \frac{1}{2}$.  Hence,
$$1_{\geq 0}(\pi(Y)) \sim 1_{\geq 0}(\pi(X)) \sim 1$$
and
$$1 - 1_{\geq 0}(\pi(Y)) \sim 1 - 1_{\geq 0}(\pi(X)) \sim 1.$$
Hence, $Y \in \Ff_{SA, \infty}$.  Since $Y$ is arbitrary, $B(X, \delta)_{SA} \subset \Ff_{SA, \infty}$, i.e., 
$X$ is an interior point of $\Ff_{SA, \infty}$.  Since $X$ is arbitrary, $\Ff_{SA, \infty}$ is an open subset
of $\Ff_{SA}$.\\
\end{proof}

\begin{rmk}  \label{rmk:Aug5202310PM}
The proofs of Lemma \ref{lem:May220232AM} and Lemma \ref{lem:Aug520235AM} actually show that
$\Ff_{SA,+}$, $\Ff_{SA, -}$ and $\Ff_{SA, \infty}$ are all (norm-) open subsets of 
$\Mul(\B)_{SA}$ (i.e., the set of self-adjoint elements of $\Mul(\B)$).\\
\end{rmk}

For a unital C*-algebra $\C$,
recall that $U(\C)$ denotes the unitary group of $\C$, with 
the restriction of the norm topology from $\C$.  We have notation 
for various topological subspaces of $U(\C)$, analogous to those for
$GL(\C)$.  In more detail, we consider the following:     
Let $U(\C)_{SA}$ denote the self-adjoint elements of $U(\C)$ (i.e., the self-adjoint unitaries
in $\C$).   Let $U(\C)_{SA,*}$ denote the $x \in U(\C)_{SA}$ such that $sp(x)$ contains both
$1$ and $-1$.  
Let $U(\C)_{SA, \infty}$ denote the $x \in U(\C)_{SA}$ such that 
$1_{\geq 0}(x) \sim 1_{\C} \sim 1_{\geq 0}(-x)$. 
By arguments similar to the $GL(\C)$ case,  $U(\C)_{SA}$, $U(\C)_{SA,*}$ and $U(\C)_{SA, \infty}$ are (norm-) closed 
subsets of $U(\C)$.  Similarly,  by arguments similar to 
the $GL(\C)$ case, 
$U(\C)_{SA, \infty}$ is an (norm-) open subset of $U(\C)_{SA}$. 
Note also that all the above spaces are given the
restriction of the norm topology from $\C$.

Recall also that by convention ($\Lambda$) (stated at the beginning of this section 4), $P_0 \in \Mul(\B)$ is a projection
such that $P_0 \sim 1 \sim 1 - P_0$.\\

\begin{df}
We define
$$GL_{P_0}(\Mul(\B)) =_{df} \{ X \in GL(\Mul(\B)) : X P_0 - P_0 X \in \B \}$$
and 
$$U_{P_0}(\Mul(\B)) =_{df} \{ U \in U(\Mul(\B)) : U P_0 - P_0 U \in \B \},$$
and both objects are topological groups (with the restriction of the norm topology from $\Mul(\B)$).\\
\end{df}

\begin{rmk}   \label{rmk:Aug520238PM} Let
$$\Mul(\B)_{P_0} =_{df} \{ Z \in \Mul(\B) : Z P_0 - P_0 Z \in \B \}.$$
Then $\Mul(\B)_{P_0}$, with the norm and *-algebraic operations inherited from $\Mul(\B)$, is a 
C*-algebra.  Moreover, 
$$GL_{P_0}(\Mul(\B)) = GL(\Mul(\B)_{P_0}) \makebox{  and  }
U_{P_0}(\Mul(\B)) = U(\Mul(\B)_{P_0}).$$\\
\end{rmk}

\begin{lem}  \label{lem:May320235AM}
$U_{P_0}(\Mul(\B))$ is a deformation retract of 
$GL_{P_0}(\Mul(\B))$.

Let $\C$ be a unital C*-algebra.    
$U(\C)_{SA,*}$ is  a deformation retract of $GL(\C)_{SA,*}$;
and $U(\C)_{SA, \infty}$ is also a deformation retract of
$GL(\C)_{SA, \infty}$. 
In particular, the map 
\begin{equation} \label{equ:Aug520237PM} GL(\C)_{SA, \infty} \rightarrow U(\C)_{SA, \infty} :  X 
\mapsto 1_{\geq 0}(X) - 1_{\geq 0}(-X) 
\end{equation}
 is a retraction onto 
$U(\C)_{SA, \infty}$ which 
is homotopic to the identity map $id_{GL(\C)_{SA,\infty}} : 
GL(\C)_{SA, \infty} \rightarrow GL(\C)_{SA, \infty}$.

In the above, all spaces are given the restriction of the norm topology (from the ambient 
C*-algebra which is either $\Mul(\B)$ or $\C$).  
\end{lem}

\begin{proof}
In all of the spaces mentioned in the statement of Lemma \ref{lem:May320235AM},  all the deformation retraction maps 
(including the map (\ref{equ:Aug520237PM})) are the standard maps using polar decomposition; i.e., they all
have the form
\begin{equation} \label{equ:Aug520239PM}
X \mapsto X |X|^{-1}.
\end{equation}

By Remark \ref{rmk:Aug520238PM},  $GL_{P_0}(\Mul(\B))$  and $U_{P_0}(\Mul(\B))$ are the general linear group and 
unitary group, respectively, of the unital C*-algebra $\Mul(\B)_{P_0}$.   It is a standard result that the map
$GL(\Mul(\B)_{P_0}) \rightarrow U(\Mul(\B)_{P_0})$, given by (\ref{equ:Aug520239PM}), is a deformation retraction (e.g., 
see \cite{WeggeOlsen} Lemma 4.2.3). 

Clearly, the map (\ref{equ:Aug520237PM}) is the identity map when restricted to $U(\C)_{SA, \infty}$.  
By an argument using the continuous functional calculus, it is not hard to see that the map (\ref{equ:Aug520237PM}) is an 
instant of the map  (\ref{equ:Aug520239PM}),  and also, 
$$|X| 1_{\geq 0} (X) = 1_{\geq 0}(X) |X| \makebox{  and  } |X|  1_{\geq 0} (-X) = 1_{\geq 0}(-X) |X|, 
\makebox{  for all  } X \in GL(\C)_{SA, \infty}.$$  
Consider the (norm-) continuous map
$$F : GL(\C)_{SA, \infty} \times [0,1] \rightarrow GL(\C)$$
given by
$$F(X, t) =_{df} (1_{\geq 0}(X) - 1_{\geq 0 }(-X))((1-t) |X| + t 1_{\C}) \makebox{ for all  } X \in GL(\C)_{SA, \infty}
\makebox{  and  } t \in [0,1].$$
By an argument using the continuous functional calculus, for all $X \in GL(\C)_{SA, \infty}$
  and   $t \in [0,1]$, $F(X, t) \in GL(\C)_{SA}$,    and
$$1_{\geq 0}(F(X, t)) = 1_{\geq 0}(X) \sim 1 \makebox{  and  } 1_{\geq 0} (-F(X,t)) = 1_{\geq 0}(-X) \sim 1.$$
Hence, $ran(F) \subset GL(\Mul(\B))_{SA, \infty}$.  Also, it is not hard to see that
$F(., 0) = id_{GL(\C)_{SA, \infty}}$, and $F(.,1)$ is the map (\ref{equ:Aug520237PM}). 
Hence, $F$ is a homotopy between the map $id_{GL(\C)_{SA, \infty}}$ and  (\ref{equ:Aug520237PM}).
Hence,  (\ref{equ:Aug520237PM}) is a deformation retraction map
 of $GL(\C)_{SA, \infty}$ onto $U(\C)_{SA, \infty}$.

By a similar argument,  (\ref{equ:Aug520237PM}) is a deformation
retraction of $GL(\C)_{SA,*}$ onto $U(\C)_{SA,*}$.\\ 
\end{proof}

Recall that 
for a C*-algebra $\C$, we let $Proj(\C)$ denote the collection of projections 
in $\C$.  

\begin{df}  \label{df:ProI}   Let $\C$ be a unital C*-algebra.
$$\ProI(\C) =_{df} \{ p \in Proj(\C) :  p \sim 1_{\C} 
\sim 1_{\C} - p \}.$$ 

We give $\ProI(\C)$ the restriction of the norm topology from $\C$.
\end{df}

Often, one refers to $\ProI(\C)$ as a 
type of \emph{Grassmannian space}.\\

Basic operator theory immediately
 gives us an alternate  characterization of $\ProI(\C(\B))$:     

\begin{lem} \label{lem:ProIChar}
$$\ProI(\C(\B)) = \{  p \in Proj(\C(\B)) : \exists P \in Proj(\Mul(\B)) \makebox{  s.t.  }   P \sim 1_{\Mul(\B)} \sim 
1_{\Mul(\B)} -P \makebox{  and  }  \pi(P) = p \}.$$
\end{lem}

\begin{proof}

This follows immediately from Definition \ref{df:ProI} and Lemma \ref{lem:Aug520233pm}.\\

\end{proof}

\begin{lem}  \label{lem:May320237AM}   Let $\C$ be a unital C*-algebra.
The map  $$U(\C)_{SA, \infty} \rightarrow 
\ProI(\C) : u \mapsto \frac{u+1}{2}$$
is a homeomorphism, where all the spaces are given the restriction of the norm topology from $\C$. 
\end{lem}

\begin{proof}
This is an excellent basic exercise for the reader.\\
\end{proof}

We are now ready for the map which is the first component in the Pre Spectral Flow Isomorphism.
\begin{lem}  \label{lem:Aug520239AM}
The map 
$$\Ff_{SA, \infty} \rightarrow \ProI(\C(\B)) : A \mapsto 1_{\geq 0}(\pi(A))$$ 
is a homotopy equivalence,   where  all the spaces are given the restriction of the norm topology from 
the ambient C*-algebra (which will be either $\Mul(\B)$ or $\C(\B)$). 

As a consequence, $\ProI(\C(\B))$ is a path-connected topological space.  
\end{lem}

\begin{proof}
The first part follows immediately from  
Lemmas \ref{lem:May320232AM}, \ref{lem:May320235AM} and \ref{lem:May320237AM}.
In fact, the map in the present lemma is an obvious composition of
the maps from the above lemmas.

The last statement follows from Lemma \ref{lem:Aug520235AM}, the first part of the present lemma and the fact that
path-connectedness is a homotopy invariant.
\end{proof}

\subsection{A certain fibration}  
In this subsection, following \cite{PereraThesis},
we will define a certain Hurewicz fibration 
which will lead to  
the second major component in the definition of the Pre Spectral
Flow Isomorphism. (Hurewicz fibrations are often called ``fibrations" (without
``Hurewicz");  see  \cite{HatcherTopologyBook} bottom of page 375, just before Theorem 4.41; 
see also \cite{Spanier} page 66.)\\

\begin{rmk} \label{rmk:NormedSpaceIsCW} 
In what follows, we will repeatedly use the following result:
Every open subset of a 
normed linear space is homotopy equivalent to a 
CW complex.  
The statement and proof can be found in \cite{LundellWeingram} Chapter 4, Corollary 5.5.   (For the separable case, there is a stronger (and older) result:
Every convex subset of a separably metrizable    
locally convex topological vector space is homotopy equivalent to a CW
complex.  See \cite{MilnorCWHomotopyType} Theorem 1 and \cite{Hu} 
Corollary II.14.2 and Theorem II.3.1.) \\ 
\end{rmk}

\begin{lem} \label{lem:ProIAndU_PCWComplexes}
For a unital C*-algebra $\C$, $GL(\C)$ and $U(\C)$ are homotopy equivalent to CW complexes.
Also, the spaces $\ProI(\C(\B))$,  $GL_{P_0}(Mul(\B))$ and $U_{P_0}(\Mul(\B))$ 
are  homotopy equivalent to CW complexes.
\end{lem}

\begin{proof}
We will use the result that an open subset of a normed linear space is homotopy equivalent to a CW complex
(see Remark \ref{rmk:NormedSpaceIsCW}).  

Let $\C$ be a unital C*-algebra.   It is well-known that the general linear group (i.e., group of invertibles)
$GL(\C)$, of $\C$, is an open subset of $\C$ (\cite{WeggeOlsen} Lemma 4.2.1), 
and that $U(\C)$ is a deformation retract
of $GL(\C)$ (\cite{WeggeOlsen} Lemma 4.2.3).   
Since $\C$ is a normed linear space,  both $GL(\C)$ and $U(\C)$ are
homotopy equivalent to CW complexes. 

By Remark \ref{rmk:Aug520238PM},  there is a unital C*-algebra $\Mul(\B)_{P_0}$ for which
$$GL_{P_0}(\Mul(\B)) = GL(\Mul(\B)_{P_0}) \makebox{  and  }
U_{P_0}(\Mul(\B)) = U(\Mul(\B)_{P_0}).$$
Hence, by the previous paragraph, $GL_{P_0}(\Mul(\B))$ and $U_{P_0}(\Mul(\B))$ are both homotopy
equivalent to CW complexes.

By Lemma \ref{lem:Aug520239AM} and Remark \ref{rmk:Aug5202310PM},
$\ProI(\C(\B))$ is homotopy equivalent to an open subset of $\Mul(\B)_{SA}$, which is a (real) normed
linear space.   
Hence, $\ProI(\C(\B))$ is homotopy equivalent to a CW complex.\\  
\end{proof}

\begin{lem} \label{lem:Aug18202310PM}
Fix a Fredholm operator $X \in \Mul(\B)$, 
i.e., $\pi(X) \in GL(\Mul(\B))$. 
Let $E$ be the set of elements of $GL(\M_2 \otimes \Mul(\B))$
that have the form  
\[
\left[ \begin{array}{cc} X +a  & b \\ c &  Y \end{array} \right],
\]
where $a, b, c \in \B$ and $Y \in \Mul(\B)$.  Here,
$E$ is given the restriction of the norm topology from $\M_2 \otimes \Mul(\B)$.   

Then $E$ is homotopy equivalent to a CW complex.  
\end{lem}

\begin{proof}
We will again use the result that an open subset of a normed linear space is homotopy-equivalent to a CW complex
(see Remark \ref{rmk:NormedSpaceIsCW}).  

Let $\E$ be the C*-subalgebra of $\M_2 \otimes \Mul(\B)$ given by 
\begin{eqnarray*}
\E & =_{df} &  \left[  \begin{array}{cc}   \B & \B \\
\B & \Mul(\B) \end{array} \right]\\
& = & \left\{ \left[  \begin{array}{cc}   a & b \\
c & Y \end{array} \right] :  a, b, c, \in \B, 
\makebox{ } Y \in \Mul(\B) \right\}.
\end{eqnarray*}

Let $E_1 \subset \M_2 \otimes \Mul(\B)$ be defined by 
\begin{eqnarray*}
E_1 & =_{df} &  \left[  \begin{array}{cc} X+   \B & \B \\
\B & \Mul(\B) \end{array} \right]\\
& = & \left\{ \left[  \begin{array}{cc} X+  a & b \\
c & Y \end{array} \right] :  a, b, c, \in \B,
\makebox{ } Y \in \Mul(\B) \right\}.   
\end{eqnarray*}
We give $E_1$ the restriction of the norm topology from $\M_2 \otimes
\Mul(\B)$.               

Note that the map
$$\Gamma : \E \rightarrow E_1 : \left[  \begin{array}{cc}   a & b \\
c & Y \end{array} \right]  \mapsto 
\left[  \begin{array}{cc} X +  a & b \\
c & Y \end{array} \right]$$
is a (nonlinear) homeomorphism. 
(Recall also that $X$ is fixed.)

Since $GL(\M_2 \otimes \Mul(\B))$ is an (norm-) open subset of 
the C*-algebra $\M_2 \otimes \Mul(\B)$ (see \cite{WeggeOlsen} Lemma
4.2.1),  
$E_1 \cap GL(\M_2 \otimes \Mul(\B))$ is an (norm-) open subset of 
$E_1$. Hence, since $\Gamma$ is a homeomorphism, 
$\Gamma^{-1}(E_1 \cap GL(\M_2 \otimes \Mul(\B)))$ is an open subset of 
$\E$.  Since $\E$ is a normed linear space,
$\Gamma^{-1}(E_1 \cap GL(\M_2 \otimes \Mul(\B)))$ is homotopy-equivalent to
a CW complex (see Remark \ref{rmk:NormedSpaceIsCW}).   
Hence, $E_1 \cap GL(\M_2 \otimes \Mul(\B))$ is homotopy-equivalent 
to a CW complex.  But this completes the proof, since
$$E = E_1 \cap GL(\M_2 \otimes \Mul(\B)).$$\\   
 
\end{proof}

\vspace*{3ex}
\textbf{\emph{Recall, from our standing convention ($\Lambda$) for the whole of (this) Section 4,  that $P_0 \in \Mul(\B)$ is a projection for which 
$P_0 \sim 1 \sim 1 - P_0$ (see the beginning of this section 4).\\\\}}

Following \cite{PereraThesis}, we now define a certain fibration which is the key to the second component
in the definition of the Pre Spectral Flow Isomorphism:    

\begin{df} \label{df:AFibration}
We define the map $\alpha_{P_0} : U(\Mul(\B)) \rightarrow \ProI(\C(\B))$ 
by
$$\alpha_{P_0}(U) =_{df} \pi(U P_0 U^*) \makebox{  for all  } U \in
U(\Mul(\B)).$$\\ 
\end{df}

\begin{thm} \label{thm:AFibration}
The map $\alpha_{P_0} : U(\Mul(\B)) \rightarrow \ProI(\C(\B))$ 
has the homotopy lifting property for all topological spaces (see \cite{HatcherTopologyBook} page 375); i.e., $\alpha_{P_0}$ is a Hurewicz fibration
(see \cite{HatcherTopologyBook} page 377).  In fact, $\alpha_{P_0}$ is 
a principal fiber bundle.   Moreover,
the fiber
$$\alpha_{P_0}^{-1}(\pi(P_0)) = U_{P_0}(\Mul(\B)).$$
\end{thm}

\begin{proof}

We follow and modify the argument of \cite{PereraThesis} Theorem 2.4.1, p45,
which is for the case of type $II_{\infty}$ factors.
See also \cite{PereraPaper}  Theorem 3.9.

Firstly, that $\alpha_{P_0}$ has range in $\ProI(\C(\B))$ and is
(norm-) continuous is clear.\\

\emph{Surjectivity of $\alpha_{P_0}$:}  If $p \in \ProI(\C(\B))$, then
there exists a projection $P \in \Mul(\B)$ with $P \sim 1 \sim 1 - P$ for
which $\pi(P) = p$ (see Lemma \ref{lem:ProIChar}).  Then we can find a unitary $U \in \Mul(\B)$
such that $U P_0 U^* = P$.  Hence, $\alpha_{P_0}(U) = P$.  Since $p$
was arbitrary, $\alpha_{P_0}$ is surjective.\\

\emph{$\alpha_{P_0}$ is an open map:}
Let $U \in \Mul(\B)$ and let $\epsilon >0$ be given.
We want to prove that $\alpha_{P_0}(B_{U(\Mul(\B))}(U, \epsilon))$ is an open
subset of $\ProI(\C(\B))$.
(Recall that $B_{U(\Mul(\B))}(U, \epsilon) =_{df} \{ U' \in U(\Mul(\B)) :
\| U - U' \| < \epsilon \}$ is the open ball about $U$ with radius
$\epsilon$.)

So let $p \in \alpha_{P_0}(B_{U(\Mul(\B))}(U, \epsilon))$ be arbitrary.
Hence, choose $V \in U(\Mul(\B))$ such that $\| V - U \| < \epsilon$
and $p = \pi(V P_0 V^*)$.
Now choose $0<\delta < \frac{1}{2}$ so that
\begin{equation} \label{equ:May520233AM}  \delta + \| V - U \|
< \epsilon. \end{equation}
By \cite{WeggeOlsen}  Proposition 5.2.6 (its proof), let $\delta' > 0$
be such that for any unital C*-algebra $\C$, for any projections
$r, r' \in \C$, if $\| r - r' \| < \delta'$ then there exists a unitary
$x \in \C$ such that $r' = x r x^*$ and $\| x - 1 \| < \delta$.

Now suppose that $q \in \ProI(\C(\B))$ is such that
$$\| q - p \| < \delta'.$$
Then by our choice of $\delta'$, there exists a unitary $w \in \C(\B)$
such that $q = w p w^*$ and $\| w - 1 \| < \delta$.
Then by Lemma \ref{lem:May520235PM},  let $W \in \Mul(\B)$ be a
unitary such that $\| W - 1 \| < \delta$ and $\pi(W) = w$.
Now
$$\alpha_{P_0}(WV) = \pi(WVP_0 V^* W^*) = \pi(W) \pi(V P_0 V^*) \pi(W)^* =
w p w^* = q.$$
Also, since $\| W - 1 \| <
\delta$ and  by (\ref{equ:May520233AM}),
$$\| WV - U \| \leq \| WV - V \| + \| V - U \|
\leq \| W - 1 \| + \| V - U \| < \delta + \| V - U \| < \epsilon.$$
So $WV \in B_{U(\Mul(\B))}(U, \epsilon)$.  So
$q \in \alpha_{P_0}(B_{U(\Mul(\B))}(U, \epsilon))$.
Since $q$ was arbitrary, $B_{\ProI(\C(\B))}(p, \delta') \subset
\alpha_{P_0} (B_{U(\Mul(\B))}(U, \epsilon))$.  Since $p$ was arbitrary, every point
in $\alpha_{P_0} (B_{U(\Mul(\B))}(U, \epsilon))$ is an interior point of
$\alpha_{P_0} (B_{U(\Mul(\B))}(U, \epsilon))$, i.e., $\alpha_{P_0} (B_{U(\Mul(\B))}(U, \epsilon))$ is
open.  Since $U, \epsilon$ were arbitrary, $\alpha_{P_0}$ is an open map.\\

\emph{$\alpha_{P_0}$ has local cross-sections (See \cite{Steenrod} 7.4 on
page 30):} By the Bartle--Graves Selection
theorem  (see \cite{Michael} Proposition 7.2; 
see also the corollary, in the
introduction, to \cite{Michael} Theorem 3.2"), 
let $s_1 : \C(\B) \rightarrow \Mul(\B)$ be a (norm-)
 continuous map such that
$\pi \circ s_1 = id_{\C(\B)}$.
By replacing $s_1$ by $s_1 + b$ for an appropriate $b \in \B$ if necessary,
we may assume that $s_1(1_{\C(\B)}) = 1_{\Mul(\B)}$.
Hence, since $s_1$ is (norm-) continuous, 
we can find a $\delta_1 > 0$ so that 
$s_1(B_{\C(\B)}(1, \delta_1)) \subseteq  B_{\Mul(\B)}(1, \frac{1}{2})$.
Hence, by  \cite{WeggeOlsen} Lemma 4.2.1, every operator in 
$s_1(B_{\C(\B)}(1, \delta_1))$ is an invertible element of $\Mul(\B)$.
Hence,
we have a (norm-) continuous map
$s_2 : B_{\C(\B)}(1, \delta_1) \rightarrow U(\Mul(\B))$ given by
$$s_2(z) = s_1(z) |s_1(z)|^{-1} \makebox{  for all  } 
z \in B_{\C(\B)}\left(1, \delta_1 \right).$$
Moreover, since $\pi \circ s_1 = id_{\C(\B)}$,
for every $u \in U(\C(\B))$ for which $\| u - 1 \| < \delta_1$,
we have that
\begin{equation}  \label{equ:May520238HalfPM}
\pi \circ s_2 (u) = u.
\end{equation}

Let $p \in \ProI(\C(\B))$ be arbitrary.  We want to construct a local cross-section of $\alpha_{P_0}$ which is
defined on an open neighborhood of $p$ in $\ProI(\C(\B))$.
By \cite{WeggeOlsen}  Proposition 5.2.6, there is a $\delta > 0$ and  a
(norm-) continuous map
$$Proj(B_{\C(\B)}(p, \delta)) \rightarrow U(\C(B)) : q \mapsto u_q$$
such that $u_p = 1$ and for every projection $q \in B_{\C(\B)}(p, \delta)$,
$$u_q p u_q^* = q.$$
Note that, as a consequence, $Proj(B_{\C(\B)}(p, \delta)) \subset \ProI(\C(\B))$.
Contracting $\delta$ if necessary, we may assume that for all
$q \in Proj(B_{\C(\B)}(p, \delta))$, $\| u_q - 1 \| < \delta_1$, and hence,
$u_q$ is in the domain of $s_2$.
Also, since we have already proven that $\alpha_{P_0}$ is surjective (see above), let $U \in U(\Mul(\B))$ be such that $\alpha_{P_0}(U) = \pi(U P_0 U^*)
= p$. 
Now define
$$s : Proj(B_{\C(\B)}(p, \delta)) \rightarrow U(\Mul(\B)) : q \mapsto
  s_2(u_q) U$$
which is clearly continuous.  Now let 
$r \in Proj(B_{\C(\B)}(p, \delta))$.
So $u_r \in \C(\B)$ is a unitary such that
$$r = u_r p u_r^* \makebox{  and  } \| u_r - 1 \| < \delta_1.$$
Then
\begin{eqnarray*}
\alpha_{P_0}(s(r)) & = & \alpha_{P_0}(s_2(u_r) U)\\
& = & \pi(s_2(u_r) U P_0 U^* s_2(u_r)^*)\\
& = & \pi(s_2(u_r)) \pi(U P_0 U^*) \pi(s_2(u_r)^*) \\
& = & u_r p u_r^* \makebox{  (by  (\ref{equ:May520238HalfPM}))}\\
& = & r \makebox{  as required.}
\end{eqnarray*}
Hence, $s$ is a local cross-section of $\alpha_{P_0}$ on the
open neighbourhood $Proj(B_{\C(\B)}(p, \delta)) =
B_{\ProI(\C(\B))}(p, \delta)$ of $p$.  Since $p$ was arbitrary, $s$ has
local cross-sections.\\

To summarize the above, $\alpha_{P_0}:   U(\Mul(\B)) \rightarrow 
\ProI(\C(\B))$ is a continuous open surjection
and has local cross-sections.  Hence, by \cite{Steenrod} Sections 7.3 and
7.4 (on pages 30-31), $\alpha_{P_0}$ is a principal fiber bundle. (In fact,
$\alpha_{P_0}$ is a principal $U_{P_0}(\Mul(\B))$-bundle.) 
But by \cite{Engelking} Chapter 5 Section 1 Corollary 1 (on page 211), 
since $\ProI(\C(\B))$ is metrizable (its 
topology being the restriction of the the norm topology from $\C(\B)$), it is a paracompact
topological space. Hence, by \cite{Dold} Theorem 4.8 (see also
\cite{Spanier} Chapter 2 Section 7 Theorem 14),
$\alpha_{P_0}$ has the homotopy lifting property with respect to 
all topological spaces, i.e., it is a Hurewicz fibration.

The expression for the fiber is straightforward.\\    
\end{proof}

To continue, let us recall some more notation.  Let  $X$ be a topological space. For a  continuous map
(or path) $\omega : [0,1] \rightarrow X$,
$\omega^{-1}$ is the path that is ``going in the opposite direction", i.e., 
$$\omega^{-1}(t) =_{df} \omega(1 - t) \makebox{  for all  } t \in [0,1].$$
Also, for $x_0 \in X$, $\Omega_{x_0}(X)$ or $\Omega_{x_0}X$ is the space of loops in $X$ that are based at $x_0$, i.e.,
$$\Omega_{x_0}(X) =_{df} \{ f \in C(S^1, X) : f(1) = x_0 \}  
=   \{ f \in C([0,1], X) : f(0) = f(1) = x_0 \}.$$
$\Omega_{x_0}(X)$ is given the \emph{compact open topology} (e.g., see
\cite{Engelking} Chapter 3, section 3, on page 121).  When $X$ is a metric
space (e.g., a normed space),  the compact open topology on $\Omega_{x_0}(X)$ is the same as the topology of uniform convergence over $S^1$ (see
\cite{Engelking} Theorem 8.2.3; see also \cite{HatcherTopologyBook} page 395
as well as the Appendix of \cite{HatcherTopologyBook} starting on page 529).    

To simplify notation, when the base point $x_0$ is clear, we simply write
$\Omega(X)$ (or $\Omega X$) in place of $\Omega_{x_0}(X)$ 
(or $\Omega_{x_0}X$).\\

We now define a map whose ``homotopy inverse" will be the second  
component of the Pre Spectral Flow Isomorphism.

\begin{df} \label{df:PereraIsomInverseRealSecondComp} 
Let $E$, $B$ be topological spaces with $E$ contractible, and suppose that
$\alpha : E \rightarrow B$ is a Hurewicz fibration.   
Let $b_0 \in B$ and $e_0 \in F =_{df} \alpha^{-1}(b_0)$. 

Since $E$ is contractible, fix a contraction of $E$ to $e_0$
(i.e., a homotopy between
$id_E : E \rightarrow E$ and the constant map $E \rightarrow \{ e_0 \}$). 

We define  
$$\kappa : F \rightarrow \Omega_{b_0} B$$
as follows:  
For each $e \in F$, the contraction gives a continuous path $\omega$ from 
$e$ to $e_0$ (so $\omega(0) = e$ and $\omega(1) = e_0$); and we define  
$$\kappa(e) =_{df} \alpha \circ \omega \in \Omega_{b_0} B.$$\\  
\end{df}

\begin{rmk} \label{rmk:Sept520231AM}      
In Definition \ref{df:PereraIsomInverseRealSecondComp}, the definition of
the map $\kappa(e)$ (for $e \in F$) 
is independent of the choice of the contraction of
$E$, up to homotopy in $\Omega_{b_0} B$.   This is because since $E$ is 
contractible, it is simply connected;  and in a simply connected topological
space, two continuous paths with the same endpoints are homotopic (with 
a homotopy that fixes the endpoints); and hence, any two such paths (from
different contractions) will 
yield homotopic loops in $\Omega_{b_0} B$.\\  
\end{rmk}

Our strategy of argument resembles more the interesting and important work of \cite{PereraPaper}
and \cite{PereraThesis} for the case of type $II_{\infty}$ factors  (see also \cite{AtiyahSingerSkew}) than for
the case of the general multiplier algebra $\Mul(\B)$, and as a consequence, 
our results are stronger, and  
our
computations are more direct, concrete and complete (see the paragraphs before
Definition \ref{df:PereraIsomRealSecondComp}). 
We also note that, for the type $II_{\infty}$ factor case,
\cite{PereraThesis} uses a ``homotopy inverse" of $\kappa$
as opposed to $\kappa$ (as in Definition 
\ref{df:PereraIsomInverseRealSecondComp}). As a result,  
the proof, in \cite{PereraPaper} and 
\cite{PereraThesis}, of (this paper's) Theorem 
\ref{thm:May520237PM} (for the $II_{\infty}$ factor case), 
 is considerably more difficult and actually
missing details which are not easy to fill in (see the paragraphs before
Definition \ref{df:PereraIsomRealSecondComp}).  In contrast, since in this paper, we are 
using $\kappa$ instead, our proof of Theorem \ref{thm:May520237PM} is 
very different and easier.  We are also able to provide
a complete proof.   
In addition, since again the argument in \cite{PereraPaper} and
\cite{PereraThesis} (for the
$II_{\infty}$ factor case) uses a ``homotopy
inverse" of $\kappa$ (as opposed to $\kappa$), it is more difficult to
prove that the induced map $\pi_0(\Omega B) \rightarrow \pi_0(F)$ is a
group homomorphism.  In fact, \cite{PereraPaper} and 
\cite{PereraThesis} \emph{do not} prove that
this map is a group homomorphism (and do not even show that
$\pi_0(\Omega B)$ and $\pi_0(F)$ are groups), 
and this is another significant gap in
their argument.  (See the two paragraphs after Theorem \ref{thm:PereraIsoIsIso}.)

Suppose that we have $E$, $B$, $F$, $\alpha : E \rightarrow B$, $e_0$
and $b_0$ as in Definition \ref{df:PereraIsomInverseRealSecondComp}.
In our context (the context of interest), 
we will additionally have that $F$ is a topological
group.  Also, by \cite{Switzer} Example 2.15 ii), 
$\Omega_{b_0} B$ is an H-group (see \cite{Switzer} 2.13).
Among other things, this implies that $\pi_0(\Omega_{b_0}B)$
and $\pi_0(F)$ are both groups (e.g., see \cite{Switzer} 
Propositon 2.14).  
We will not prove that the the map $\kappa$ (as in Definition 
\ref{df:PereraIsomInverseRealSecondComp}) is an H-homomorphism
(see \cite{Spanier} Chapter 1, Section 5, just
before Theorem 4 (on page 35);  
Spanier uses the terminology ``homomorphism"). 
However, we will prove that the map $\kappa$ will respect the H-group
structures of $\Omega_{b_0} B$ and $F$ so that the induced map
$\kappa_* : \pi_0(F) \rightarrow 
\pi_0(\Omega_{b_0} B)$ is a group homomorphism. 
This will be enough for our purposes of showing that $\kappa$
is a homotopy equivalence and, ultimately, finishing the definition
of the Pre Spectral Flow Isomorphism.\\

\begin{lem}  \label{lem:kappaGivesPi0Hom}       
Let $E, B$ be topological spaces with $E$ contractible, and suppose
that $\alpha : E \rightarrow B$ is a Hurewicz fibration.
Let $b_0 \in B$ and $e_0 \in F =_{df} \alpha^{-1}(b_0)$.

Suppose, in addition, that $E$ is a topological group, $F$ is
a topological subgroup of $E$, $e_0 = 1_E = 1_F$ (the unit of 
the group), and 
for all $e \in E$, $\alpha(e F) = \{ \alpha(e) \}$.  

Let $\kappa : F \rightarrow \Omega_{b_0} B$ be defined as 
in  Definition \ref{df:PereraIsomInverseRealSecondComp}.   

Then for all $f, f' \in F$,
$\kappa(ff')$ is homotopic to $\kappa(f)\kappa(f')$ in $\Omega_{b_0}B$.
As a consequence, the induced map
$$\kappa_* : \pi_0(F) \rightarrow \pi_0(\Omega_{b_0} B)$$
is a group homomorphism. 
\end{lem}

\begin{proof}
Since $E$ is contractible, fix a contraction from $E$ to $1_E$.
(See also Remark \ref{rmk:Sept520231AM}.) 

Let $f, f' \in F$ be arbitrary. 
Let $\omega_0 : [0,1] \rightarrow E$ be the continuous path
with $\omega_0(0) = ff'$ and $\omega_0(1) = 1_E$ given by the 
contraction as in Definition \ref{df:PereraIsomInverseRealSecondComp}.
Let $\omega, \omega' : [0,1] \rightarrow E$ be the continuous paths
with $\omega(0) = f$, $\omega'(0) = f'$ and $\omega(1) = \omega'(1) 
= 1_E$ that are also given by the contraction as in Definition
\ref{df:PereraIsomInverseRealSecondComp}.
(So by Definition \ref{df:PereraIsomInverseRealSecondComp},
$\kappa(ff') =_{df} \alpha \circ \omega_0$, $\kappa(f) =_{df}
 \alpha \circ 
\omega$ and $\kappa(f') =_{df} \alpha \circ \omega'$.)  
     
Now in $\Omega_{b_0} B$, 
$\kappa(f) \kappa(f') = \kappa(f) * \kappa(f')$ 
where $*$ is concatenation of paths (see Definition \ref{df:DfGather} (3)).
So 
\begin{equation} \label{equ:Sept520232AM}
\kappa(f) \kappa(f') = \alpha(\omega)*\alpha(\omega').  
\end{equation}  

Let $\omega_1 : [0,1] \rightarrow E$ be the continuous path that 
is given by 
\[
\omega_1(t) =
\begin{cases}
\omega(2t) f' & t\in [0,\frac{1}{2}]\\
 \omega'(2t - 1) & t \in [\frac{1}{2}, 1],  
\end{cases} 
\]
i.e., $\omega_1$ is the concatenation $(\omega f') * \omega'$.
Hence, 
\[
\alpha \circ \omega_1 = (\alpha(\omega f))* (\alpha( \omega')).
\]

Hence, 
by (\ref{equ:Sept520232AM}) and since we have assumed that
$\alpha(eF) = \{ \alpha(e) \}$ for all $e \in E$, 
we must have that  
\begin{equation} \label{equ:Sept520233AM}
\alpha \circ \omega_1 = \kappa(f)\kappa(f')
\end{equation}

Now $\omega_1$ is a continuous path in $E$ connecting $ff'$ with
$1_E$, i.e., it has the same endpoints as $\omega_0$.  
But since $E$ is contractible, $E$ is simply connected.  So any
two continuous paths in $E$, with the same endpoints, must be homotopic
(with a homotopy fixing the two endpoints).  So 
$\omega_0$ and $\omega_1$ 
are homotopic in $E$ (with a homotopy fixing the endpoints).
So $\kappa(ff') = \alpha \circ \omega_0$ is homotopic to 
$\alpha\circ \omega_1 = \kappa(f) \kappa (f')$ in $\Omega_{b_0}B$.
(See (\ref{equ:Sept520233AM}) and the definition of $\omega_0$.)\\
\end{proof}

\begin{thm} \label{thm:May520237PM}
Let $E =_{df} U(\Mul(\B))$, $B =_{df} \ProI (\C(\B))$, $b_0 =_{df}
\pi(P_0) \in B$, $\alpha =_{df} \alpha_{P_0} : E \rightarrow B$ 
be the Hurewicz fibration
from Definition \ref{df:AFibration} (see also Theorem 
\ref{thm:AFibration}),   
and $e_0 =_{df} 1_{\Mul(\B)} \in F =_{df} \alpha^{-1}(\pi(P_0))
= U_{P_0}(\Mul(\B))$.  

Since $E = U(\Mul(\B))$ is (norm-) contractible (e.g., see
\cite{WeggeOlsen} Theorem 16.8), let
$$\kappa : F \rightarrow \Omega_{b_0} B$$
be defined as in Definition \ref{df:PereraIsomInverseRealSecondComp}.

Then $\kappa$ is a homotopy equivalence between $F$ and 
$\Omega_{b_0} B$.

Moreover, $F$ and $\Omega_{b_0} B$ have obvious H-group structures (see
\cite{Switzer} 2.13)  and 
the induced map
$$\kappa_* : \pi_0(F) \rightarrow \pi_0(\Omega_{b_0} B)$$
is a group isomorphism.  
\end{thm}

\begin{proof}
By Lemmas \ref{lem:HatcherPi_nIsom} and \ref{lem:kappaPi0Bijection},
for all $n \geq 1$, for all $f \in F$, the induced map
$$\kappa_* : \pi_n(F, f) \rightarrow \pi_n(\Omega_{b_0} B , \kappa(f))$$
is a group isomorphism, and
the induced map
$$\kappa_* : \pi_0(F) \rightarrow \pi_0(\Omega_{b_0} B)$$
is a bijection.
But by Lemma \ref{lem:ProIAndU_PCWComplexes} and   
\cite{MilnorCWHomotopyType} Theorem 3,
$F$ and $\Omega_{b_0} B$ are homotopy equivalent to CW complexes.
Hence, by Whitehead's Theorem (\cite{Spanier} Chapter 7, Section 6, Corollary 
24; see also \cite{HatcherTopologyBook} Theorem 4.5 for a special case),
the map
$$\kappa : F \rightarrow \Omega_{b_0} B$$
is a homotopy equivalence.

Now by \cite{Switzer} Examples 2.15, both $F$ and $\Omega_{b_0}B$ have 
obvious H-group structures (in fact, $F$ is already a topological group).
In fact, our $E$, $B$, $\alpha$, $b_0$, $e_0$ and $F$ satisfy the hypotheses
of Lemma \ref{lem:kappaGivesPi0Hom}.  Hence, by Lemma 
\ref{lem:kappaGivesPi0Hom} and since $\kappa$ is a homotopy equivalence
(also can use Lemma \ref{lem:kappaPi0Bijection}),
$$\kappa_* : \pi_0(F) \rightarrow \pi_0(\Omega_{b_0} B)$$
is a group isomorphism.\\   
\end{proof}

As mentioned in the paragraph after Remark \ref{rmk:Sept520231AM},
our construction of the Pre Spectral Flow Isomorphism resembles more
the interesting and important construction in the type $II_{\infty}$ factor case in \cite{PereraPaper}
and \cite{PereraThesis} than the construction for a multiplier algebra 
$\Mul(\B)$ case (with $\B$ having an approximate unit consisting of projections)
in \cite{PereraThesis}.   We note that for the $\Mul(\B)$ case,
\cite{PereraThesis} does not prove that $\kappa$ is a homotopy equivalence
(as in our Theorem \ref{thm:May520237PM})) -- instead, \cite{PereraThesis}
only claims to prove that $\kappa$ induces group isomorphisms
between all homotopy groups $\kappa_* : \pi_k(\Omega B) \cong \pi_k(F)$ (see
\cite{PereraThesis} Theorem 3.3.8 (c)).  

Our construction more resembles the interesting and important construction of \cite{PereraPaper}
and \cite{PereraThesis} for the case of $II_{\infty}$ factors. 
For $II_{\infty}$ factors, 
because \cite{PereraThesis} uses a ``homotopy inverse" of $\kappa$ (instead of
using $\kappa$), the proof of (our) Theorem \ref{thm:May520237PM} in 
\cite{PereraThesis} (for $II_{\infty}$ factors) is considerably more difficult and has significant 
missing details in it  (see \cite{PereraThesis} page 45, after the proof
of Theorem 2.4.1).    
One example, is that the argument in
\cite{PereraThesis} page 45 (after the proof of Theorem 2.4.1)
 appeals to an exercise in 
\cite{Switzer} (\cite{Switzer} Exercise 4.22) which, at least to an analyst,
is not a straightforward exercise.  (It was also not immediate 
to a homotopy theorist that we consulted.) Moreover, it was also
not immediate to us that  
a solution to \cite{Switzer} Exercise 4.22 would automatically 
fill in the step in the argument of \cite{PereraThesis}.  
In contrast, our paper's proof of Theorem \ref{thm:May520237PM},
which includes the argument of Subsection \ref{subsect:KappaIsHomotopyEquivalence}, is a completely different proof which is elementary, much easier, and
logically complete.\\

\begin{df} \label{df:PereraIsomRealSecondComp}
Suppose that we have the same assumptions as that of
Theorem \ref{thm:May520237PM}.

We define $\gamma : \Omega_{b_0} B \rightarrow F$ to be a homotopy
inverse of $\kappa$.
I.e., $\gamma$ is a continuous map for which   
$$\gamma \circ \kappa \simeq id_F  \makebox{  and  }
\kappa \circ \gamma \simeq id_{\Omega_{b_0} B}$$
where $\simeq$ is the binary relation for two maps being
homotopic.\\

Note that by Theorem \ref{thm:May520237PM},
$$\gamma_* : \pi_0(\Omega_{b_0} B) \rightarrow \pi_0(F)$$
is a group isomorphism.\\       
\end{df}

\subsection{Completing the Pre Spectral Flow Isomorphism}

In this subsection, we will define the last component of the Pre Spectral
Flow Isomorphism  
and thus complete its definition.\\ 

\textbf{Recall also
 the standing assumption ($\Lambda$), from the beginning of
(this) Section 4, that $\B$ is a separable stable 
C*-algebra and $P_0 \in \Mul(\B)$ is a projection for which
$P_0 \sim 1_{\Mul(\B)} \sim 1_{\Mul(\B)} - P_0$.}\\  

The key result of this section is Theorem \ref{thm:LastComponentPerera}, 
where we prove that the map 
$$GL_{P_0}(\Mul(\B)) \rightarrow GL(\C(P_0\B P_0)): 
X \mapsto P_0 X P_0$$ 
is a homotopy equivalence (which induces a group
homomorphism in $\pi_0$).
This result gives the last component of the
Pre Spectral Flow Isomorphism
$PSf$ and allows us to complete the construction of $PSf$.\\ 

We will use the following observation of Carey and Phillips (which we 
here state for the convenience of the reader):

\begin{prop}  \label{prop:CareyPhillipsFibration}
Let $V, W$ be Banach spaces and $\phi : V \rightarrow W$ a continuous linear
surjection.
Let $U \subseteq V$ be an open subset such that $O =_{df} \phi(U) \subseteq W$
is an open subset.  

Then the restricted map $\phi : U \rightarrow O$ is a Hurewicz fibration. 
\end{prop}

\begin{proof}
This is \cite{CareyPhillipsClifford}  Proposition A.14 (in the appendix of
the paper).\\   
\end{proof}

The next result generalizes \cite{CareyPhillipsClifford} Theorem 3.10 and 
Corollary 3.11, which is for the type $II_{\infty}$ factor case.  
(See also \cite{PereraPaper}  Theorem 2.2.1.)

\begin{thm}  \label{thm:LastComponentPerera} 
The map 
$$\Phi : GL_{P_0}(\Mul(\B)) \rightarrow GL(\C(P_0 \B P_0))  : X \mapsto
\pi(P_0 X P_0)$$
is a homotopy equivalence.  
\end{thm}

\begin{proof}
Let   
$$\widehat{V} =_{df} P_0 \Mul(\B) P_0 + (1 - P_0) \Mul(\B) (1 - P_0) + \B.$$
Using the Carey--Phillips matrix notation
(\cite{CareyPhillipsClifford}), 
$$\widehat{V} = 
\left\{ \left[\begin{array}{cc}  X' & a \\ b & Y' \end{array} \right] :
X' \in P_0\Mul(\B) P_0, \makebox{ } Y' \in (1- P_0) \Mul(\B) (1 - P_0),
\makebox{  } a \in 
P_0 \B (1 - P_0),
\makebox{ } b \in (1 - P_0) \B P_0 \right\}.$$ 
Note that $\widehat{V}$, with the restriction of the
 norm from $\Mul(\B)$, is a Banach space.

Let 
$$\widehat{W} =_{df} \C(P_0 \B P_0)$$
which is a Banach space, since it is a C*-algebra.
Clearly, the map
$$\Phi : \widehat{V} \rightarrow \widehat{W} : 
X \mapsto \pi(P_0 X P_0)$$
(or, in Carey--Phillips matrix notation,
$\Phi : \left[ \begin{array}{cc}  X' & a \\ b & Y' \end{array} \right] 
\mapsto \pi(X')$) 
is a surjective continuous linear map.

Now, $GL_{P_0}(\Mul(\B)) = \widehat{V} \cap GL(\Mul(\B))$.
Hence, since $GL(\Mul(\B))$ is a (norm-) open subset of $\Mul(\B)$
(see \cite{WeggeOlsen} Lemma 4.2.1), 
$GL_{P_0}(\Mul(\B))$ is open in $\widehat{V}$.
It is also clear that $GL(\C(P_0 \B P_0))$ is an open subset of 
$\widehat{W} = C(P_0 \B P_0)$ (again see \cite{WeggeOlsen} Lemma 4.2.1).\\

\noindent \emph{Claim 1:}   
$\Phi(GL_{P_0}(\Mul(\B))) = GL(\C(P_0 \B P_0)).$\\

\noindent \emph{Proof of Claim 1:}  Now,  
$\pi(GL_{P_0}(\Mul(\B))) \subset   \pi(P_0) \C(\B) \pi(P_0) + \pi(1- P_0) 
\C(\B)\pi(1 - P_0)$.  If $X  \in GL_{P_0}(\Mul(\B))$, then 
$\pi(X)$ is invertible in $\pi(\Mul(\B)) = \C(\B)$;  hence, 
$\pi(P_0 X P_0)$ must be invertible in $\pi(P_0) \C(\B) \pi(P_0)$, i.e., 
$\Phi(X) \in GL(\C(P_0 \B P_0))$.  Hence, $\Phi(GL_{P_0}(\Mul(\B)))
\subseteq GL(\C(P_0 \B P_0))$.

We now prove the reverse inclusion.   
Let $y \in GL(\C(P_0 \B P_0))$ be arbitrary.   
Let $y = v |y|$ be the polar decomposition of $y$ in $\C(P_0 \B P_0)$.  So
$v$ is a unitary in $\C(P_0 \B P_0)$ and $|y|$ is a positive invertible 
element of $\C(P_0 \B P_0)$.  Let $a_v \in P_0 \Mul(\B) P_0$ be a contractive
element such that $\pi(a_v) = v$.  
Since $|y|$ is homotopic to $\pi(P_0)$ in $GL(\C(P_0 \B P_0))$, we can
find an invertible element $\widetilde{y} \in \Mul(P_0 \B P_0)$ such that 
$\pi(\widetilde{y}) = |y|$ (see \cite{WeggeOlsen} Corollary 4.3.3).
In full (nonCarey--Phillips) matrix notation, let 
$$Z_1 =_{df} \left[ \begin{array}{cc} a_v & \sqrt{1 - a_v a_v^*} \\
\sqrt{1 - a_v^* a_v} & - a_v^* \end{array} \right] 
\in U_{P_0}(\Mul(\B)) \subset  GL_{P_0}(\Mul(\B)).$$
(Note that $\sqrt{1 - a_v a_v^*} \in \B$ since $\pi(a_v) = v$ is a unitary.
Similarly, $\sqrt{1 - a_v^* a_v} \in \B$.  Also, since
$P_0 \sim 1 \sim 1 - P_0$, we may view $-a_v^*$ as being an element of
$(1 - P_0) \Mul(\B) (1 - P_0)$.  In fact, we may identify 
$\Mul(\B) \cong \mathbb{M}_2(P_0 \Mul(\B) P_0)$.)  

Let $$Z_2 =_{df} 
\left[ \begin{array}{cc} \widetilde{y} & 0 \\ 0 & 1 - P_0 \end{array}
\right].$$
Then $Z_1 Z_2 \in GL_{P_0}(\Mul(\B))$ and 
$$\pi(P_0 Z_1 Z_2 P_0) = y, \makebox{ i.e., } \Phi(Z_1 Z_2) = y.$$
Since $y$ was arbitrary, this completes the proof of the reverse inclusion
and hence the Claim.\\
\noindent \emph{End of proof of Claim 1.}\\

By Claim 1, we let $\Phi$ also denote the restricted map
$GL_{P_0}(\Mul(\B)) \rightarrow GL(\C(P_0 \B P_0))$.
From Claim 1 and the above, it follows, by 
Proposition \ref{prop:CareyPhillipsFibration}, that the map 
$\Phi : GL_{P_0}(\Mul(\B)) \rightarrow GL(\C(P_0 \B P_0))$
is a Hurewicz fibration.  We denote this statement by ``(!)".

\vspace*{2ex}
\noindent \emph{Claim 2:}  For every $X \in GL_{P_0}(\Mul(\B))$,
the fiber $\Phi^{-1}(\Phi(X))$ (which is a subspace of the total
space $GL_{P_0}(\Mul(\B))$)
is contractible.\\

\noindent \emph{Proof of Claim 2:}   
Fix $X \in GL_{P_0}(\Mul(\B))$.  
Note that $X$ is an invertible element of $\Mul(\B)$ that has the
form (in the Carey--Phillips notation)   
$$X = \left[ \begin{array}{cc} X_0 & b \\ c & Y_0 \end{array} \right]$$
where $X_0 \in P_0 \Mul(\B) P_0$, $b \in P_0 \B (1 - P_0)$,
$c \in (1 - P_0) \B P_0$ and $Y_0 \in (1-P_0) \Mul(\B) (1 - P_0)$.

Towards proving the Claim 2, let us first use a more convenient full
 two-by-two
matrix representation of $\Phi^{-1}(\Phi(X))$ (different from the 
Carey--Phillips matrix notation).  
Now since $P_0 \sim 1 \sim 1 - P_0$, we have an obvious  homeomorphism 
\begin{equation}  \label{equ:May620231AM}
\Phi^{-1}(\Phi(X)) \cong GL\left(\left[ \begin{array}{cc} X_0  + \B & \B \\
\B & \Mul(\B) \end{array} \right]\right)  \subseteq GL(\M_2(\Mul(\B)))
\end{equation} 
(so under the Murray--von Neumann
 equivalence $P_0 \sim 1$, we view $X_0$ as an element
of $\Mul(\B)$).  
For convenience, for the proof of this Claim, we identify 
$\Phi^{-1}(\Phi(X))$ with
the (norm-) closed 
subset of $GL(\M_2(\Mul(\B)))$, with the restriction of the
 norm topology from $\M_2(\Mul(\B))$,
 that is given above.
In other words, we identify $\Phi^{-1}(\Phi(X))$ with all invertible
matrices of the form
$$\left[\begin{array}{cc}  X_0  + a  & b \\
c & Y \end{array} \right]  \in GL(\M_2(\Mul(\B)))
\makebox{ where } a, b,c \in \B, 
\makebox{ and } Y \in \Mul(\B).$$        

By Lemma \ref{lem:Aug18202310PM}, $\Phi^{-1}(\Phi(X))$ is 
homotopy-equivalent to a CW complex.  
Hence, to prove that $\Phi^{-1}(\Phi(X))$ is contractible, it suffices to prove that
for every compact Hausdorff topological space $K$,
$C(K \rightarrow \Phi^{-1}(\Phi(X)))$ is (uniform norm-) path-connected.
(Short summary of the argument:  Suppose that we have proven the above 
path-connectedness property for every compact Hausdorff space $K$.
By \cite{Switzer} Definition 2.60, all homotopy groups vanish, i.e.,
$\pi_n(\Phi^{-1}(\Phi(X))) = 0$ for every $n \geq 0$. Since 
$\Phi^{-1}(\Phi(X))$ has the homotopy type of a CW complex, we can apply 
Whitehead's Theorem (e.g., \cite{HatcherTopologyBook} Theorem 4.5) to a 
constant map $\Phi^{-1}(\Phi(X)) \rightarrow \{ z_1 \}$.)  
We denote the above statements by ``(V)".

So let $K$ be an arbitrary compact Hausdorff space, and let
$X_1, X_2 \in C(K \rightarrow \Phi^{-1}(\Phi(X)))$ be arbitrary.    
We want to show that $X_1$ and $X_2$ are connected by a (uniform norm-)
continuous path in $C(K \rightarrow \Phi^{-1}(\Phi(X)))$.  
Using the identification (\ref{equ:May620231AM}), 
\begin{equation} \label{equ:Aug1920231AM}
X_1 = \left[ \begin{array}{cc} X_0 + a_1 & b_1 \\ 
c_1 & Y_1 \end{array} \right] \makebox{  and  }
X_2 = \left[ \begin{array}{cc} X_0 + a_2 & b_2 \\ 
c_2 & Y_2 \end{array} \right] 
\end{equation}
\noindent where $a_j, b_j, c_j : K \rightarrow \B$ and
$Y_j : K \rightarrow \Mul(\B)$ are all continuous maps ($j =1,2$). 
For $j=1, 2$,
since $X_j \in GL(C(K) \otimes \M_2 \otimes \Mul(\B))$ (by 
the definitions of $\Phi^{-1}(\Phi(X))$ and $X_j$), we have that
$\pi(X_0), \pi(Y_j) 
\in GL(C(K) \otimes \C(\B))$.     
(Here, we identify  $X_0$ with the constant map $1_{C(K)} \otimes X_0$ 
in $C(K) \otimes \Mul(\B)$.)
Hence, for $j = 1,2$, in $K_1(C(K) \otimes \C(\B))$,
\begin{equation} \label{equ:Aug2820231AM}
[\pi(X_j)] = [\pi(X_0)] + [\pi(Y_j)].
\end{equation}
\noindent(Here, we are using the common notation $\pi$ for all the 
the quotient maps 
$\Mul(\B) \rightarrow \C(\B)$, 
$\M_2(\Mul(\B)) \rightarrow \M_2(\C(\B))$, 
$C(K) \otimes \Mul(\B) \rightarrow
C(K) \otimes \C(\B)$, $C(K) \otimes \M_2(\Mul(\B)) \rightarrow C(K)
\otimes \M_2(\C(\B))$, etc.)    
Now by \cite{WeggeOlsen} Corollary 16.7, 
$K_1(C(K) \otimes \Mul(\B)) = 0$.  Hence, by (\ref{equ:Aug2820231AM}),
since $X_j$ is invertible in $C(K) \otimes \M_2 \otimes  \Mul(\B)$,   
we must have that in $K_1(C(K) \otimes \C(\B)))$, 

\[
0 = [\pi(X_j)] = [\pi(X_0)] + [\pi(Y_j)]  \makebox{ }(j= 1, 2).
\]
Hence,
in $K_1(C(K) \otimes \C(\B))$,
\[
[\pi(Y_1)] = - [\pi(X_0)] = [\pi(Y_2)].
\]
Hence, since $C(K) \otimes \C(\B)$ is $K_1$-injective (see 
Lemma \ref{lem:CoronaK1InjAndSurj}), 
$\pi(Y_1)$ and $\pi(Y_2)$ are connected by a (norm-) continuous
path in $GL(C(K) \otimes \C(\B))$.   
Hence, there exist $d'_1, d'_2, ..., d'_m \in C(K) \otimes \C(\B)$
such that 
\begin{equation} \label{equ:Aug1920232AM}
\pi(Y_2) = exp(d'_1) exp(d'_2) \cdot \cdot \cdot exp(d'_m) \pi(Y_1).
\end{equation} 
For $l =1,..., m$, let $d_l \in C(K) \otimes \Mul(\B)$ be such that 
\begin{equation} \label{equ:Aug1920233AM} 
\pi(d_l) = d'_l. 
\end{equation}
Hence, if for all $t \in [0,1]$, we define
$$Z_t =_{df} exp(td_1) exp(td_2) \cdot \cdot \cdot exp(td_m)
\in GL(C(K) \otimes \Mul(\B))$$
and we define
$$X_{1,t} =_{df} \left[ \begin{array}{cc} 1 & 0 \\ 0 & Z_t \end{array}
\right] X_1,$$
then by (\ref{equ:Aug1920231AM}), (\ref{equ:Aug1920232AM}), 
(\ref{equ:Aug1920233AM}) and the definition of 
$\Phi^{-1}(\Phi(X))$, $\{ X_{1,t} \}_{t \in [0,1]}$ is a      
(norm-) continuous path in $C(K \rightarrow \Phi^{-1}(\Phi(X)))$
such that   
$$X_{1,0} = X_1 \makebox{  and  } X_{1,1} - X_2 \in C(K) \otimes 
 \M_2 \otimes \B.$$
We denote the above statement by ``(+)".

From (+), it follows that there are (norm-) continuous maps
$a_3, b_3, c_3, d_3 : K \rightarrow \B$ such that 
the matrix $$A =_{df} \left[ \begin{array}{cc} 1 + a_3 & b_3 \\ c_3 & 1 + d_3
\end{array} \right] \in GL(C(K) \otimes \M_2 \otimes \Mul(\B))$$
and
$$X_2 = X_{1,1} A.$$ 
We denote the above statement by ``(++)".

Now since $\B$ is stable, let $\{ p_j \}_{j=1}^{\infty}$ be a sequence of
pairwise orthogonal projections in $\Mul(\B)$ such that 
$p_j \sim 1_{\Mul(\B)}$ for all $j$ and 
\begin{equation} \label{equ:May620235AM} 1_{\Mul(\B)} 
= \sum_{j=1}^{\infty} p_j \end{equation}  
where the sum converges strictly in $\Mul(\B)$.  
Since $K$ is compact, 
we can find an integer $N \geq 1$ and (norm-) continuous maps
$\widetilde{a}_3, \widetilde{b}_3, \widetilde{c}_3, \widetilde{d}_3 
: K \rightarrow 
\left(\sum_{j=1}^N p_j\right) \B \left( \sum_{j=1}^N p_j \right)$ 
such that if 
$$\widetilde{A} =_{df} 
\left[ \begin{array}{cc} 1_{\Mul(\B)} + \widetilde{a}_3 & 
\widetilde{b}_3 \\
\widetilde{c}_3 & 1_{\Mul(\B)} + \widetilde{d}_3 \end{array} \right],$$
then $\widetilde{A}$ is an invertible element   
in $GL(1_{\M_2 \otimes \Mul(\B)}  
+ C(K) \otimes \M_2 \otimes \B)$ for which
$\widetilde{A}$ is (norm-) close to $A$ and
$\widetilde{A}$ is homotopic to $A$ in 
$GL( 1_{\M_2 \otimes \Mul(\B)}
+ C(K) \otimes \M_2 \otimes \B)$. (Note that      
the ambient C*-algebra is really 
$\mathbb{C}1_{\M_2 \otimes \Mul(\B)}            
+ C(K) \otimes \M_2 \otimes \B$, but all operators (including
inverses) can be made to reside in the (nonC*-algebra)
$1_{\M_2 \otimes \Mul(\B)}            
+ C(K) \otimes \M_2 \otimes \B$.)  
  We denote the above statements 
by ``(*)".

Let $\{ e_{j,k} \}_{1 \leq j,k \leq 2}$ be the standard system of matrix
units for $\M_2$.  
By \cite{WeggeOlsen} Lemma 16.2, let $W \in \M_2 \otimes \Mul(\B)$
be a unitary such that 
\begin{enumerate}
\item $W (\sum_{j=N+2}^{\infty} 1_{\M_2} \otimes p_j) 
= (\sum_{j=N+2}^{\infty} 1_{\M_2} \otimes p_j) W
= \sum_{j=N+2}^{\infty} 1_{\M_2} \otimes p_j$,
\item $W (\sum_{j=1}^N 1_{\M_2} \otimes p_j) W^* = 
e_{2,2} \otimes p_{N+1}$, and 
\item $W (1_{\M_2} \otimes p_{N+1}) W^* = 
e_{1,1} \otimes p_{N+1}  +  
\sum_{j=1}^N 1_{\M_2} \otimes p_j$.   
\end{enumerate} 
(Note that by \cite{WeggeOlsen} Lemma 16.2, all the 
above projections and sums of projections are Murray--von Neumann
equivalent to $1_{\M_2 \otimes \Mul(\B)}$ in 
$\M_2 \otimes \Mul(\B)$.)
Identifying $W$ with the constant function $1_{C(K)} \otimes W$
in $C(K) \otimes \M_2 \otimes \Mul(\B)$ (and identifying 
$1_{\Mul(\B)}$ with
$1_{C(K) \otimes \Mul(\B)}$),
\begin{equation}
W \widetilde{A} W^* = 
\left[ \begin{array}{cc} 1_{\Mul(\B)} & 0 \\
0 & 1_{\Mul(\B)} + f \end{array} \right]     
\end{equation}
where $f \in C(K) \otimes \B$.  
By \cite{WeggeOlsen} Corollary 16.7, 
$U(C(K) \otimes \M_2 \otimes \Mul(\B))$ is 
(norm-) path-connected.  Hence, let $\{ W_t \}_{t \in [0,1]}$ be 
a (norm-) continuous path of unitaries in $U(C(K) \otimes 
\M_2 \otimes \Mul(\B))$
such that $$W_0 = 1  \makebox{  and  } W_1 = W.$$
We have that 
for all $t \in [0,1]$,
$$W_t \widetilde{A} W_t^* \in GL(1_{C(K) \otimes \M_2 \otimes \Mul(\B)}
+ C(K) \otimes \M_2 \otimes \B).$$
Hence, $\widetilde{A}$ is homotopic to 
$\left[\begin{array}{cc} 1_{\Mul(\B)} & 0 \\ 0 & 
1_{\Mul(\B)} + f \end{array} \right]$  in\\ 
$GL(1_{C(K) \otimes \M_2 \otimes \Mul(\B)}
+ C(K) \otimes \M_2 \otimes \B)$.   
From this and (*), $A$ is homotopic to $\left[\begin{array}{cc} 1_{\Mul(\B)} & 0 \\ 0 &
1_{\Mul(\B)} + f \end{array} \right]$  in
$GL(1_{C(K) \otimes \M_2 \otimes \Mul(\B)}
+ C(K) \otimes \M_2 \otimes \B)$.
Noting that by definition (e.g., of $\Phi^{-1}(\Phi(X))$), 
$$C(K \rightarrow \Phi^{-1}(\Phi(X))) 
GL(1_{C(K) \otimes \M_2 \otimes \Mul(\B)}
+ C(K) \otimes \M_2 \otimes \B) \subset C(K \rightarrow
\Phi^{-1}(\Phi(X))),$$ 
it follows by (++) that 
in $C(K \rightarrow \Phi^{-1}(\Phi(X)))$, 
$X_2$ is homotopic to 
$$X_{1,1} \left[ \begin{array}{cc} 1 & 0 \\ 0 & 1 + f 
\end{array} \right].$$ 
We denote the above statement by ``(**)". 

Now by \cite{WeggeOlsen} Corollary 16.7, 
$GL(C(K) \otimes \Mul(\B))$ is (norm-) path-connected. So since 
$1 +f \in GL( C(K) \otimes \Mul(\B))$, $1 +f$ is (norm-) path-connected
to $1$ in $GL(C(K) \otimes \Mul(\B))$.  Hence,
$\left[\begin{array}{cc} 1 & 0 \\ 0 & 1 + f \end{array}\right]$ is
(norm-) path-connected to $\left[\begin{array}{cc} 1 & 0 \\ 0 & 1 
\end{array}\right]$ in 
$$e_{1,1} \otimes 1  + e_{2,2} \otimes 
GL(C(K) \otimes \Mul(\B))  
= \left[\begin{array}{cc} 1 & 0 \\ 0 & GL(C(K) \otimes \Mul(\B)) 
\end{array} \right].$$
So since (by the definition of $\Phi^{-1}(\Phi(X))$),
$$C(K \rightarrow \Phi^{-1}(\Phi(X))) \left[\begin{array}{cc} 1 & 0 \\ 0 & GL(C(K) \otimes \Mul(\B))
\end{array} \right] \subset C(K \rightarrow \Phi^{-1}(\Phi(X))),$$
it follows, by (**), that in $C(K \rightarrow \Phi^{-1}(\Phi(X)))$,
$X_2$ is homotopic to $X_{1,1}$.  Hence, by (+),             
in $C(K \rightarrow \Phi^{-1}(\Phi(X)))$,
$X_2$ is homotopic to $X_{1,0} = X_1$. 

Since $X_1, X_2 \in C(K \rightarrow \Phi^{-1}(\Phi(X)))$ were arbitrary,
$C(K \rightarrow \Phi^{-1}(\Phi(X)))$ is (uniform norm-) path-connected.
Hence, since $K$ was arbitrary, by (V), 
$\Phi^{-1}(\Phi(X))$ is (norm-) contractible.  
Since $X \in GL_{P_0}(\Mul(\B))$ was arbitrary, we have completed
the proof of Claim 2.\\ 
\noindent \emph{End of proof of Claim 2.}\\

Now by statement (!), $\Phi :  GL_{P_0}(\Mul(\B))
\rightarrow GL(\C(P_0 \B P_0))$ is a Hurewicz fibration.
Let $X \in GL_{P_0}(\Mul(\B))$ be arbitrary.  
Hence, by \cite{Spanier} Chapter 7, Section 2, Theorem 10
(see also \cite{HatcherTopologyBook} Theorem 4.41),
there is a long exact sequence of homotopy groups
\begin{equation} \label{equ:May720231AM}
... \rightarrow \pi_n(F, x_0) \rightarrow \pi_n(E, x_0)
\overset{\Phi_*}{\rightarrow} \pi_n(B, b_0) 
\rightarrow \pi_{n-1}(F, x_0) \rightarrow ... 
\rightarrow \pi_0(B, b_0) 
\end{equation}   
where $E = GL_{P_0}(\Mul(\B))$, $B = GL(\C(P_0 \B P_0))$, 
$F = \Phi^{-1}(\Phi(X))$, $x_0 = X$, and
$b_0 = \Phi(X)$.

By Claim 2, we have that 
\begin{equation} \label{equ:Aug2920236PM} 
\pi_n(\Phi^{-1}(\Phi(X)), X) = \pi_n(F, x_0) = 0 \makebox{  for all  } 
n \geq 0.
\end{equation} 
Hence, by (\ref{equ:May720231AM}), for all $n \geq 1$,
\begin{equation} \label{equ:Aug2920237PM}\Phi_* : \pi_n(GL_{P_0}(\Mul(\B)), X) \rightarrow 
\pi_n(GL(\C(P_0 \B P_0)), \Phi(X)) 
\makebox{   is a group isomorphm.}
\end{equation}  

But also, since $E$ and $B$ are topological groups,
$\pi_0(E, x_0)$ and $\pi_0(B, b_0)$ are groups,  
and it is not hard to check that 
$\Phi_* : \pi_0(E, x_0) \rightarrow \pi_0(B, b_0)$ is a  group
homomorphism.  Hence, by 
(\ref{equ:Aug2920236PM}) and (\ref{equ:May720231AM}),    
\begin{equation} \label{equ:Aug2920235PM} 
\Phi_* :  \pi_0(GL_{P_0}(\Mul(\B)), X) \rightarrow
\pi_0(GL(\C(P_0 \B P_0)), \Phi(X)) \end{equation}   
is injective.   But since the map $\Phi :      
GL_{P_0}(\Mul(\B)) \rightarrow GL(\C(P_0 \B P_0))$ is surjective (see Claim 
1), it follows that the $\Phi_*$ in (\ref{equ:Aug2920235PM}) is also
surjective and hence a group isomorphism.

Hence, for all $n \geq 0$, the induced map
$$\Phi_* : \pi_n(GL_{P_0}(\Mul(\B)), X) \rightarrow
\pi_n(GL(\C(P_0 \B P_0)), \Phi(X))$$
is a group isomorphism.
Since $X$ was arbitrary, the above statement is true for all
$X \in GL_{P_0}(\Mul(\B))$.  
But by Lemma \ref{lem:ProIAndU_PCWComplexes}, $GL_{P_0}(\Mul(\B))$ and
$GL(\C(P_0 \B P_0))$ are both homotopy equivalent to CW complexes.
Hence, by Whitehead's Theorem (e.g., see \cite{HatcherTopologyBook}
Theorem 4.5), $\Phi$ is a homotopy equivalence
between $GL_{P_0}(\Mul(\B))$  
and $GL(\C(P_0 \B P_0))$.\\                    
\end{proof}

We are now ready to define the Pre Spectral Flow Isomorphism.  
By Lemma \ref{lem:Aug520239AM}, Theorem \ref{thm:May520237PM}, 
Definition \ref{df:PereraIsomRealSecondComp},  
Lemma \ref{lem:May320235AM} and Theorem \ref{thm:LastComponentPerera}, 
we have a sequence of homotopy
equivalences
\begin{equation}  \label{equ:May720232AM} \end{equation}
\begin{align*} 
\Omega_{2P_0 - 1} \Ff_{SA, \infty}(\Mul(\B)) 
 \rightarrow \Omega_{\pi(P_0)} \ProI(\C(\B)) \rightarrow 
\alpha_{P_0}^{-1}(\pi(P_0)) = U_{P_0}(\Mul(\B))\\  
\rightarrow GL_{P_0}(\Mul(\B)) 
 \rightarrow GL(\C(P_0 \B P_0)).
\end{align*}
This induces bijections of the $\pi_0$ ($0$th homotopy set)
of the above spaces:
\begin{equation} \label{equ:Aug2920238PM} \end{equation}
\begin{align*}
\pi_0(\Omega_{2P_0 - 1} \Ff_{SA, \infty})   
\rightarrow \pi_0(\Omega_{\pi(P_0)} \ProI)  \rightarrow 
\pi_0(U_{P_0}(\Mul(\B))) \rightarrow \pi_0(GL_{P_0}(\Mul(\B)))\\
\rightarrow \pi_0(GL(\C(P_0\B P_0))).
\end{align*}    
Now all the topological spaces in (\ref{equ:May720232AM}) are 
H-groups (see, for example, \cite{Spanier} Chapter 1, Section 5, pages 34-35 (before Theorem 4)
as well as \cite{Spanier} Chapter 1, Section 5, Theorems 1, 4 and 8;  see
also \cite{Switzer} Examples 2.15).
Hence,  all the $\pi_0$ sets in (\ref{equ:Aug2920238PM}) are 
groups ($0$th homotopy groups). We need to check that
 all the maps in (\ref{equ:Aug2920238PM}) are group homomorphisms (and
thus group isomorphisms).  The first map $\Omega \Ff_{SA, \infty}
\rightarrow \Omega \ProI(\C(\B))$, in (\ref{equ:May720232AM}), is
an H-homomorphism    
(see \cite{Spanier} Chapter 1, Section 5, just before Theorem 4 on page 35;
Spanier uses the terminology ``homomorphism"), since this map has the form
$\Omega \sigma$ for some map $\sigma : \Ff_{SA, \infty} \rightarrow  
\ProI(\C(\B))$. Hence, by \cite{Spanier} Chapter 1, Section 5, Theorem 7,
the first induced map in (\ref{equ:Aug2920238PM}) 
is a group homomorphism (and hence isomorphism) of $\pi_0$ groups.  
The third and fourth (last) maps in (\ref{equ:May720232AM})  -- i.e., 
$U_{P_0}(\Mul(\B)) \rightarrow GL_{P_0}(\Mul(\B))$ and
$GL_{P_0}(\Mul(\B)) \rightarrow GL(\C(P_0 \B P_0))$  -- are both group
homomorphisms and hence H-homomorphisms.  Hence, by \cite{Spanier} Chapter 1,
Section 5, Theorem 7 again, the third and fourth induced maps in 
(\ref{equ:Aug2920238PM}) are group homomorphisms (and hence isomorphisms of the $\pi_0$ groups).
Finally,  by Definition \ref{df:PereraIsomRealSecondComp} and 
Theorem \ref{thm:May520237PM},  the second map of (\ref{equ:May720232AM})
is also a group homomorphism (and hence isomorphism).   
Hence, composing all the maps in (\ref{equ:Aug2920238PM}), we have 
a group isomorphism
\begin{equation} \label{equ:Aug2920239PM}   
\pi_0(\Omega_{2P_0 - 1} \Ff_{SA, \infty}(\Mul(\B)))
\rightarrow 
\pi_0(GL(\C(P_0 \B P_0))).\\  
\end{equation}

Note that since $\B$ is stable and since $P_0 \sim 1$,
$P_0 \B P_0 \cong \B$.
Also, since $\B$ is stable,
$\C(\B)$ is both $K_1$-injective and $K_1$-surjective (see 
Lemma \ref{lem:CoronaK1InjAndSurj}).  Hence, 
by the definition of $\pi_0$,  
$$\pi_0(GL(\C(P_0\B P_0))) = \pi_0(GL(\C(\B))) = K_1(\C(\B)).$$
Also, by the paragraph before Definition \ref{df:FredholmIndex}, 
the index map (from the six term
exact sequence) 
\begin{equation}  \label{equ:May720233AM}
\partial:  K_1(\C(\B)) \rightarrow K_0(\B)
\end{equation} is a 
group isomorphism.\\

Finally, note that, by definition (e.g., \cite{Switzer} Definition 2.60),  
\begin{equation} \label{equ:Aug29202311PM}
\pi_1(\Ff_{SA, \infty}(\Mul(\B)), 2P_0 - 1) 
\cong \pi_0(\Omega_{2P_0 - 1} \Ff_{SA, \infty}(\Mul(\B))).  
\end{equation}

\begin{df} \label{df:PereraIsomorphism}
The Pre Spectral Flow Isomorphism 
$$PSf : \pi_1(\Ff_{SA, \infty}, 2P_0 - 1) \rightarrow K_0(\B)$$ 
is obtained by composing the maps 
(\ref{equ:Aug29202311PM}),  
(\ref{equ:Aug2920239PM}) and (\ref{equ:May720233AM}).\\   
\end{df}

The following is immediate from the definition of $PSf$: 

\begin{thm} \label{thm:PereraIsoIsIso}
The Pre Spectral Flow Isomorphism 
$$PSf : \pi_1(\Ff_{SA, \infty}, 2P_0 - 1) \rightarrow K_0(\B)$$
is indeed a group isomorphism.\\  
\end{thm}

We belatedly note (though it is already implicit in some arguments above) that for a topological space $K$, the $0$th homotopy set 
$\pi_0(K)$ need \emph{not} be a group; and even if $K$ and $K'$ are 
topological spaces for which $\pi_0(K)$ and $\pi_0(K')$ are groups, 
a continuous map $g : K \rightarrow K'$ \emph{need not} induce a group
homomorphism $g_* : \pi_0(K) \rightarrow \pi_0(K')$.  (This is in contrast
with the \emph{higher homotopy groups} $\pi_n(K)$ and $\pi_n(K')$ for all
$n \geq 1$.)  Thus, it is important, in our construction of the Pre Spectral
Flow Isomorphism $PSf$, 
to check that various $\pi_0$ sets are actually groups and
that the relevant maps between topological spaces yield group homomorphisms
at the level of $\pi_0$.  (E.g., see the paragraph after equation 
(\ref{equ:Aug2920238PM})).  

In the argument of the interesting and important \cite{PereraPaper} and \cite{PereraThesis} (for the
$II_{\infty}$ factor case, which we generalize), the above points  are
not taken into account at all.

Before we end this section, we make two final remarks.  Firstly, from 
Lemma \ref{lem:MakingTrivialPaths}, we see that any two 
invertible elements of $\Ff_{SA, \infty}$ are path connected.  Hence, 
in Theorem \ref{thm:PereraIsoIsIso}, the invertible $2P_0 -1$ can replaced
with any invertible element of $\Ff_{SA, \infty}$.  Secondly,  
composing the homotopy equivalences in (\ref{equ:May720232AM}) gives
us a homotopy equivalence $\Omega_X \Ff_{SA, \infty} (\Mul(\B))
\rightarrow GL(P_0 \C(\B) P_0)$, where $X \in \Ff_{SA, \infty}$
is any invertible.  Hence, we have the following result,
which generalizes a result from \cite{AtiyahSingerSkew}:

\begin{thm}  Let $\B$ be a separable stable C*-algebra and 
$X \in \Ff_{SA, \infty}(\Mul(\B))$ an invertible element.
Then  
$\Omega_X \Ff_{SA, \infty}(\Mul(\B))$ is a classifying space
for the functor $K \rightarrow K_0(C(K) \otimes \B)$. I.e., for
every compact Hausdorff space $K$,
$$[K, \Omega_X \Ff_{SA, \infty}] \cong K_0(C(K) \otimes \B).$$ \\ 
\end{thm}

\section{The Spectral Flow Isomorphism Theorem and
the Axiomatization of Spectral Flow}

\subsection{The Spectral Flow Isomorphism Theorem}
\label{subsect:SFIsomorphism}

In this subsection, we show that the Pre Spectral Flow Isomorphism $PSf$ is 
induced by spectral flow $Sf$, generalizing a result of 
Atiyah--Patodi--Singer  (see \cite{AtiyahPatodiSinger3} (7.3) and \cite{AtiyahSingerSkew};  see also \cite{Phillips1996} the last theorem which
is stated on page 464,     and also \cite{PhillipsVictoria} Theorem 2.9 for the von Neumann factor case).\\\\

\textbf{Throughout (this) Subsection 5.1, we will follow the same 
conventions as in (the whole of) Section 4:  $\B$ is a separable stable
C*-algebra and $P_0 \in \Mul(\B)$ is a projection for which 
$P_0 \sim 1 \sim 1 - P_0$.  Following Section 4, we will refer
to these as the ``standing conventions ($\Lambda$)". 
(Note that these conventions will no longer be
assumed in Subsection 5.2 and afterwards.)}\\\\

Next, we need to compute a Fredholm index (Lemma \ref{lem:May1020231AM}).
What we need  should follow
from concrete, elementary computations involving essential 
codimension, which can be found in    
\cite{LoreauxNgSutradharEC}.  However, since the 
paper \cite{LoreauxNgSutradharEC} has not 
yet been published, 
we use a short KK theory argument (proof of Lemma \ref{lem:May1020231AM}), 
appealing to the already extent literature in KK theory.  Since the present paper
is meant to use mainly elementary techniques, while we do not 
explain all details,  we will give precise references, and a reader who is willing
to read some definitions and
 accept some statements from the extent and standard literature will be able to follow the
short argument without trouble.     
  
Let $KK_F$ denote the ``Fredholm picture" of KK theory introduced by 
Higson in \cite{HigsonKK} (e.g., see \cite{HigsonKK} Definition 2.2).  
For $\A$ and $\D$ C*-algebras with $\A$ separable and $\D$ $\sigma$-unital,
if we continue to let $KK(\A, \D)$ denote the generalized homomorphism
picture of KK (defined with $KK_h$-cycles) as presented in 
(\ref{equdf:KK_h}), there is 
a group isomorphism 
\begin{equation} \label{equ:Aug29202311:15PM} 
KK(\A, \D) \rightarrow KK_D(\A, \D) : 
[\phi, \psi] \mapsto [\phi, \psi, 1]. 
\end{equation} 
(See \cite{HigsonKK} Lemma 3.6.)

Next, we refer the reader to Subsection \ref{subsect:GeneralizedFredholm} 
to recall
the definitions of (generalized) Fredholm operator and (generalized)
 Fredholm index.

Finally, the reader should recall the definition of essential
codimension from Definition \ref{df:EssentialCodimension}.\\

\begin{lem} \label{lem:May1020231AM}
Let $P, Q \in \Mul(\B)$ be projections such that 
$P - Q \in \B$ and $P \sim 1 \sim 1- P$.

Say that $U \in \Mul(\B)$ is a unitary such that
$U Q U^* = P$.

Then $Q U^* Q$ is a Fredholm operator in $\Mul(Q \B Q)$ and its 
(generalized)  Fredholm index is the essential codimension $[Q:P]$.  
\end{lem}

\begin{proof}
Firstly, since $UQU^* = P$, $Q U^* = U^* P$.
So since $P - Q \in \B$,
$\pi(Q) \pi(U^*)  = \pi(U^*)\pi(Q)$ and $\pi(Q U^* Q)$ is an invertible
(actually unitary) element of $\C(Q \B Q)$.  Hence, $Q U^* Q$ 
is a Fredholm operator in 
$\Mul(Q \B Q) = Q \Mul(\B) Q$.  
 
Now let $\phi, \psi : \mathbb{C} \rightarrow \Mul(\B)$ be the *-homomorphisms
given by $\phi(1) =_{df} Q$ and $\psi(1) =_{df} P$.
So $[Q:P] =_{df} [\phi, \psi]$ in $KK(\mathbb{C}, \B)$ (see
Definition \ref{df:EssentialCodimension}).
So by the remarks before this lemma,
$[P:Q] = [\phi, \psi, 1]$ in $KK_F(\mathbb{C}, \B)$ (see
(\ref{equ:Aug29202311:15PM})). 
Note that since $P - Q \in \B$,  the (norm-) continuous path  
$$t \in [0,1] \mapsto (\phi, \psi, (1 - t)1 + tPQ)$$
gives a homotopy (see \cite{HigsonKK} Definition 2.1)  between 
$(\phi, \psi, 1)$ and $(\phi, \psi, PQ).$
Hence,
$$[\phi, \psi, 1] = [\phi, \psi, PQ] \makebox{  in  } KK_F(\mathbb{C}, \B).$$

Now since $Q \sim 1$ (see Lemma \ref{lem:AShortComputation}),  
let $S \in \Mul(\B)$ be  
an isometry such that $SS^* = Q$.
Let $T =_{df} US$.  Then $T$ is an isometry and $TT^* = USS^*U^* = UQU^* = P$.
By \cite{HigsonKK} Lemma 2.3,
$$[S^* \phi(.) S, T^* \psi(.) T, T^* S ] 
= [\phi(.), \psi(.), TT^* SS^*] = [\phi(.), \psi(.), PQ] \makebox{  in  } KK_F(
\mathbb{C}, \B).$$    
Noting that 
$$S^*\phi(1) S = S^* Q S = 1 = T^* P T = T^* \psi(1) T,$$
the above implies that
$[Q:P]$ is the Fredholm index of $T^*S = S^* U^* S$ as a Fredholm
operator in $\Mul(\B)$.
Hence, since $SS^* = Q$,  $[Q:P]$ is the Fredholm index of $QU^*Q$ 
as a Fredholm operator in $\Mul(Q \B Q)$.\\  
\end{proof}

We now construct a family of loops in $\Ff_{SA, \infty} = 
\Ff_{SA, \infty}(\Mul(\B))$ which will be 
important for the proof that the Pre Spectral Flow Isomorphism $PSf$ is induced by
spectral flow $Sf$, up to a sign difference.    

\begin{df}  \label{df:ThirdExample} 
Let $P, Q \in \Mul(\B)$ be projections such that 
$P - Q \in \B$ and $P \sim 1 \sim 1 - P$.  
(Recall that, by Lemma \ref{lem:AShortComputation}, 
this implies that $Q \sim 1 \sim 1 - Q$.)

Since $P \sim Q$ and $1-P \sim 1 - Q$, 
let $U \in U(\Mul(\B))$ be such that 
$U(2Q-1)U^* = 2P-1$.  
Since $U(\Mul(\B))$ is (norm-) contractible (by 
\cite{WeggeOlsen} Theorem 16.8), let $\{ U_t \}_{t \in [0,1]}$
be a (norm-) continuous path of unitaries in $U(\Mul(\B))$
 for which $U_0 = 1$ and $U_1
 = U$.    
Let $\{ A_t \}_{t \in [0,1]}$ be the (norm-) continuous path of invertible
(and hence Fredholm) 
operators in $\Mul(\B)$ that is given by $$A_t =_{df} U_t (2Q-1) U_t^* 
\makebox{ for all } t\in [0,1].$$

Let $\omega_{Q, P} : [0, 1] \rightarrow \Ff_{SA, \infty}(\Mul(\B))$ 
(also denoted 
$\{ \omega_{Q, P, t } \}_{t \in [0,1]}$)  be the (norm-) continuous 
loop,
based at  
$2Q-1$,  which is defined by the concatenation of paths (see 
Definition \ref{df:DfGather}):   
\[
\omega_{Q, P} =_{df} \{ A_t \}_{t \in [0,1] } * \{ (1-t) (2P - 1) + t (2 Q - 1) \}_{t
\in [0,1]}.\\   
\]
\end{df}

\vspace*{2ex}

\begin{rmk}  \label{rmk:omegaQPNotWellDefined}
Note that the path $\omega_{Q,P}$, as in Definition \ref{df:ThirdExample},
is not well-defined in the sense that
it is dependent on the choices of the unitary
$U$ and the continuous path of unitaries $\{ U_t \}$.  The path
$\omega_{Q, P}$ is also defined using the operation of concatenation,
which is only well-defined up to homotopy (see 
Definition \ref{df:DfGather}). 
               
However, the important thing for us is that whichever the choices used to
define $\omega_{Q,P}$, the computations of spectral flow and 
the Pre Spectral Flow Isomorphism will always
yield the same quantity. This is the content of the next two lemmas.
This will then show that spectral flow $Sf$ induces the Pre Spectral Flow Isomorphism $PSf$.\\
\end{rmk}

\begin{lem} \label{lem:ThirdExampleSF}
Let $P, Q \in \Mul(\B)$ be projections such that $P - Q \in \B$ and
$P \sim 1 \sim 1 - P$. 

Let $\{ \omega_{Q, P, t}  \}_{t \in [0,1]}$ be the (norm-) continuous loop of 
self-adjoint Fredholm operators, based at $2Q-1$, which is defined as in 
Definition \ref{df:ThirdExample}.    

Then the spectral flow 
$$Sf(\{ \omega_{Q, P, t} \}_{t \in [0,1]}) = [P:Q].$$      
\end{lem}

\begin{proof}
By Definition \ref{df:ThirdExample}, 
 $\omega_{P, Q}$ is defined as the concatenation 
$$\omega_{P,Q} =_{df} \{ A_t \}_{t \in [0,1]} *
\{ (1 - t)(2P - 1) + t (2Q - 1) \}_{t \in [0,1]}$$
where for all $s \in [0,1]$, $A_s$ is defined as in Definition 
\ref{df:ThirdExample}. (See also 
Remark \ref{rmk:omegaQPNotWellDefined}.)   
Now by the Concatenation Axiom (SF3)  and the Triviality Principle 
(see Propositions \ref{prop:SfSatisfiesAxioms} and \ref{prop:TrivialityPrinciple}), we have that   
\begin{eqnarray*}
Sf(\omega_{Q, P}) & = & Sf(\{ A_t \}_{t \in [0,1]} * 
 \{ (1 - t)(2P - 1) + t (2Q - 1) \}_{t \in [0,1]} ) \\
& = &   Sf(\{ A_{t} \}_{t \in [0,1]})  
+ Sf( \{ (1 - t)(2P - 1) + t (2Q - 1) \}_{t \in [0,1]} )\\
& = & Sf( \{ (1 - t)(2P - 1) + t (2Q - 1) \}_{t \in [0,1]} ).
\end{eqnarray*}
But by Proposition \ref{prop:SecondExample},
$$Sf( \{ (1 - t)(2P - 1) + t (2Q - 1) \}_{t \in [0,1]} ) = [P:Q].$$\\ 
\end{proof}

\begin{lem}  \label{lem:ThirdExamplesfP}
Let $P, Q \in \Mul(\B)$ be projections with $P - Q \in \B$ and
$P \sim 1 \sim 1-P$.

Let $\omega_{Q,P}$ be the (norm-) continuous loop, based at
$2Q - 1$,  defined as in 
Definition \ref{df:ThirdExample}. 
Let $[\omega_{Q,P}]$ be the homotopy class of $\omega_{Q,P}$ 
in the fundamental group $\pi_1(\Ff_{SA, \infty}, 2Q - 1)$.

Then the Pre Spectral Flow Isomorphism applied to $[\omega_{Q,P}]$ is  
$$PSf([\omega_{Q,P}]) = [Q:P] = -[P:Q].$$
\end{lem}

\begin{proof}
We follow the strategy of \cite{PhillipsVictoria} 2.8.

Note that here, when we are applying the Pre Spectral Flow Isomorphism $PSf$,
we are replacing $P_0$ with $Q$ in Definition \ref{df:PereraIsomorphism}.

Following the definition of the Pre Spectral Flow Isomorphism 
in Definition \ref{df:PereraIsomorphism} (see also 
(\ref{equ:May720232AM}), (\ref{equ:Aug2920238PM}), 
(\ref{equ:Aug2920239PM}), (\ref{equ:May720233AM}) and (\ref{equ:Aug29202311PM});  and also
Lemma \ref{lem:Aug520239AM}, Theorem \ref{thm:May520237PM},
Lemma \ref{lem:May320235AM} and 
Theorem \ref{thm:LastComponentPerera}) to 
compute $PSf([\omega_{Q,P}])$, we first
push $\omega_{Q,P}$ to a loop $\{ 1_{\geq 0}(\pi(\omega_{Q,P,t})) \}_{t \in [0,1]}$
in $\ProI(\C(\B))$ which is based at $Q$;     
then, using  Proposition   
\ref{prop:UniqueLifting} and that $\alpha_Q$ is a Hurewicz fibration, 
we need to lift 
$\{ 1_{\geq 0}(\pi(\omega_{Q,P,t}))^{-1} = 1_{\geq 0}(\pi(\omega_{Q, P, t}^{-1}))
 \}_{t \in [0,1]}$
 (along $\alpha_Q$) 
to any (norm-) continuous path in $U(\Mul(\B))$ which
begins (left endpoint $t = 0$) at $1_{\Mul(\B)}$, and evaluate 
at the right endpoint $t = 1$ to get an element $u$ of $U_{Q}(\Mul(\B))$; then
we take the cutdown $QuQ$, which will be Fredholm in $\Mul(Q \B Q)$, and
$PSf(\omega_{Q,P})$ will be the Fredholm index of $Q uQ$.   
Also, note that in  
the definition of the Pre Spectral Flow Isomorphism, we may replace 
$\omega_{Q,P}^{-1}$ by any homotopy equivalent loop based at $2Q - 1$.  

Let $\{ U_t \}_{t \in [0,1]}$ be the (norm-) continuous path of unitaries
in $U(\Mul(\B))$, as in Definition \ref{df:ThirdExample}.         
Up to homotopy equivalence, $\omega_{Q,P}^{-1}$ is the same as the
(norm-) continuous loop $\omega$ in $\Ff_{SA, \infty}$, based at $2Q - 1$, which 
is given by 
\[
\omega_t =_{df} 
\begin{cases}
(1 - 2t)(2Q-1) + 2t (2P-1)  & t \in [0, \frac{1}{2}]\\
U_{2 - 2t} (2Q - 1) U_{2 - 2t}^*  & t \in [\frac{1}{2}, 1]  
\end{cases}
\] 
(Recall, from Definition \ref{df:ThirdExample}, that
$U_0 = 1$ and $U_1 (2Q - 1) U_1^* = 2P - 1$.)
Applying the procedure sketched in the previous paragraph to $\omega$,
firstly, we have that
\[ 
1_{\geq 0}(\pi(\omega_t)) 
= 
\begin{cases}
\pi(Q) & t \in [0, \frac{1}{2}] \\
\pi(U_{2 - 2t}QU_{2 - 2t}^*) = \pi(U_{2 - 2t} U_1^* Q U_1 U_{2 - 2t}^*) & t \in [\frac{1}{2}, 1].
\end{cases}
\]
(Note that $\pi(U_1^* Q U_1) = \pi(U_1^* P U_1) = \pi(Q)$.)
Hence, a lift of $\{ 1_{\geq 0}(\pi(\omega_t)) \}_{t \in [0,1]}$ (along
$\alpha_Q$) is the (norm-) continuous
path $\{ V_t \}_{t \in [0,1]}$, in $U(\Mul(\B))$, given by
\[
V_t =_{df} 
\begin{cases} 
1 & t \in [0, \frac{1}{2}]\\
U_{2 - 2t} U_1^* & t \in [\frac{1}{2}, 1].   
\end{cases}
\]
And finally, $PSf([\omega_{P,Q}])$ is the Fredholm index of
$Q V_1 Q = Q U_0 U_1^*Q = QU_1^*Q$.  (Note that since $U_1Q U_1^* = P$,
$U_1 Q = P U_1$.  Hence, since $P - Q \in \B$, $Q U_1^* Q$ is a Fredholm
operator in $\Mul(Q \B Q)$.)

But by Lemma \ref{lem:May1020231AM}, the Fredholm index of $Q U_1^* Q$
is $[Q : P]$.   
That $[Q:P] = -[P:Q]$ follows from the KK definition of 
essential codimension (Definition \ref{df:EssentialCodimension}) and
\cite{JensenThomsenBook} Proposition 4.1.5.\\ 
\end{proof}

\begin{lem} \label{lem:ExhaustingPi1}  The fundamental group
$$\pi_1(\Ff_{SA, \infty}, 2P_0 - 1) = \{ [\omega_{P_0, P} ] : 
P \in Proj(\Mul(\B)) \makebox{  and  } P - P_0 \in \B \}.$$

Here, for all $P$, $[\omega_{P_0, P}]$ is the homotopy class of
the loop $\omega_{P_0, P}$, based at $2P_0 - 1$, which is defined in 
Definition \ref{df:ThirdExample}.   
\end{lem} 

\begin{proof}

Note that every element of $K_0(\B) = KK(\mathbb{C}, \B)$ is realized
as an essential codimension
$[Q':Q]$  where  $Q, Q' \in \Mul(\B)$ are projections such that 
$Q' - Q \in \B$ (this follows immediately from the KK definition of
essential codimension and the generalized 
homomorphism picture of $KK(\mathbb{C}, \B)$; e.g.,
 see (\ref{equdf:KK_h}) and Definition \ref{df:EssentialCodimension}; see also \cite{JensenThomsenBook} Chapter 4).  
Moreover, we can choose $Q, Q'$ so that 
$Q \sim 1 \sim 1-Q$ and as a consequence, by Lemma 
\ref{lem:AShortComputation},  
$Q' \sim 1 \sim 1-Q'$ (use \cite{JensenThomsenBook} Lemma 4.1.4 and
the definition of addition after it; use also \cite{WeggeOlsen} Lemma
16.2).    
Moreover, for every such pair $(Q', Q)$, 
we can find a unitary $U \in \Mul(\B)$
such that $P_0 = UQ'U^*$.
Since $U(\Mul(\B))$ is (norm-) contractible (\cite{WeggeOlsen} Theorem 
16.8), $(Q', Q)$ and $(U Q' U^*, U Q U^*) = (P_0, U Q U^*)$ correspond
to 
homotopic $KK_h(\mathbb{C}, \B)$-cycles  (see the paragraph before
Definition \ref{df:EssentialCodimension}). 
Hence, $[Q':Q] = [P_0,UQU^*]$ in $KK(\mathbb{C}, \B) = K_0(\B)$.
Since $(Q', Q)$ is arbitrary, 
every element in $K_0(\B)$ can be realized by an essential
codimension of the form $[P_0, Q'']$.

Hence, by Theorem \ref{thm:PereraIsoIsIso}   
and Lemma \ref{lem:ThirdExamplesfP},
$$\pi_1(\Ff_{SA, \infty}, 2P_0 - 1) = 
\{ [\omega_{P_0, P}] : P \in Proj(\Mul(\B)) \makebox{  and  }
P - P_0 \in \B \}.$$\\  
\end{proof}

\begin{thm}  \label{thm:Sf*=sfP}  
Let 
$$Sf_* : \pi_1(\Ff_{SA, \infty}(\Mul(\B)), 2P_0 - 1) \rightarrow K_0(\B)$$
be the group homomorphism induced by the spectral flow $Sf$.

Then $Sf_*$ is a group isomorphism.
In fact,
\begin{equation} \label{equ:Sf*=sfPExceptSign}
 Sf_* = -PSf. \end{equation}  
\end{thm}

\begin{proof}
Since spectral flow  $Sf$ satisfies the Homotopy Axiom (SF2)
(by Proposition \ref{prop:SfSatisfiesAxioms}), 
$Sf$ induces a well-defined map
$Sf_*$ on $\pi_1(\Ff_{SA, \infty}, 2P_0 - 1)$. 

Hence, since the  $PSf : \pi_1(\Ff_{SA, \infty},  2P_0 -1)
\rightarrow K_0(\B)$ is a group isomorphism (see Theorem 
\ref{thm:PereraIsoIsIso}) and by Lemmas \ref{lem:ExhaustingPi1},
 \ref{lem:ThirdExampleSF} and
\ref{lem:ThirdExamplesfP}, 
$$Sf_* = -PSf.$$\\     
\end{proof}

\begin{rmk}
The reason for the negative sign in (\ref{equ:Sf*=sfPExceptSign})
is because of the difference in ``orientation" of our version of
spectral flow and the versions of spectral flow from \cite{AtiyahPatodiSinger3}
and others like \cite{BenPhillEtAl}.   

Recall that, intuitively, our spectral flow measures the ``net mass" of the part
spectrum that passes through zero in the \emph{negative} direction.
In contrast, the spectral flow of Atiyah--Patodi--Singer and others
measures the ``net mass" of the part of the spectrum that passes through zero in
the \emph{positive} direction.  See Remarks 
\ref{rmk:SfIntuitionAndOrientation1} 
and \ref{rmk:SfIntuitionAndOrientation2}.\\  
\end{rmk}

Towards loosening the restriction on the basepoint in 
Theorem \ref{thm:Sf*=sfP}, we require the following technical result:

\begin{lem} \label{lem:MakingTrivialPaths}
Any two invertible elements of $\Ff_{SA, \infty}(\Mul(\B))$
are connected by a (norm-) continuous path of invertible elements
of $\Ff_{SA, \infty}(\Mul(\B))$. 
\end{lem}

\begin{proof}
Let $X_1, X_2 \in \Ff_{SA, \infty}(\Mul(\B))$ both be invertible.
For $j = 1,2$, since $X_j$ is invertible and self-adjoint,
let $P_j =_{df} 1_{\geq 0}(X_j) \in Proj(\Mul(\B))$ (which implies that $1- P_j =  
1_{\geq 0}(-X_j)$).
From the definition of $\Ff_{SA, \infty}$, it follows that 
for $j = 1,2$,
$$\pi(P_j) = 1_{\geq 0}(\pi(X_j)) \sim 1 \sim 1_{\geq 0}(-\pi(X_j))
= \pi(1 - P_j) = 1 - \pi(P_j).$$
By Lemma \ref{lem:Aug520233pm}, for $j = 1,2$, we can find a projection
$Q_j \in \Mul(\B)$ such that 
$$\pi(Q_j) = \pi(P_j) \makebox{  and  }
Q_j \sim 1 \sim 1 - Q_j.$$
Hence, for $j = 1,2$, $Q_j - P_j \in \B$.
It follows, by Lemma \ref{lem:AShortComputation}, that for $j = 1,2$,
$$P_j \sim 1 \sim 1 - P_j.$$
Hence, let $U \in \Mul(\B)$ be a unitary such that 
$$U P_1 U^* = P_2.$$
(As a consequence, $U(1 - P_1) U^* = 1 - P_2$.)
Among other things, note that
$$1_{\geq 0}(UX_1U^*) = P_2 =1_{\geq 0}(X_2) 
\makebox{  and  } 1_{\geq 0}(-UX_1U^*) = 1- P_2 =1_{\geq 0}(-X_2).$$ 

Since $U(\Mul(\B))$ is contractible (by \cite{WeggeOlsen} Theorem 16.8),
let $\{ U_t \}_{t \in [0,\frac{1}{2}]}$ be a (norm-) continuous path
of unitaries in $U(\Mul(\B))$ such that
$$U_0 = 1 \makebox{  and  } U_{\frac{1}{2}} = U.$$
 
Let $\{ Y_t \}_{t \in [0,1]}$ be the (norm-) continuous path of 
invertible elements in $\Ff_{SA, \infty}$ that is given by 
\[
Y_t =_{df} \begin{cases} 
U_t X_1 U_t^* & t \in [0, \frac{1}{2}] \\
(2 - 2t) U X_1 U^* +  (2t -1) X_2  & t \in [\frac{1}{2}, 1]. 
\end{cases}  
\]
Then $$Y_0 = X_1 \makebox{  and  } Y_1 = X_2.$$\\ 
\end{proof}

\begin{thm}  \label{thm:GeneralSFIsomorphism}
Let $X \in \Ff_{SA, \infty}$ be an arbitrary invertible element.

Then the spectral flow $Sf$ induces a group isomorphism
$$Sf_* : \pi_1 (\Ff_{SA, \infty}, X) \rightarrow K_0(\B).$$
\end{thm}

\begin{proof}
Firstly, since spectral flow $Sf$ satisfies the Homotopy Axiom (SF2) (see
Proposition \ref{prop:SfSatisfiesAxioms}),
it induces well-defined maps $Sf_*$ on $\pi_1(\Ff_{SA, \infty}, X)$
and $\pi_1(\Ff_{SA, \infty}, 2P_0 - 1)$.   

By Lemma \ref{lem:MakingTrivialPaths}, let $\omega_X :[0,1] 
\rightarrow \Ff_{SA, \infty}$ be a (norm-) continuous path, consisting
of \emph{invertible} elements, such that 
$$\omega_X(0) = X \makebox{  and  } \omega_X(1) = 2 P_0 - 1.$$

Then we have a group isomorphism
\begin{equation}  \label{equ:Aug29202311:59PM}
\Phi : \pi_1(\Ff_{SA, \infty}, X) \rightarrow \pi_1(\Ff_{SA, \infty},
2P_0 - 1) : [\omega] \mapsto [\omega^{-1}_X  *  \omega * \omega_X].
\end{equation} 
(Recall that $*$ is the concatenation operation; see Definition 
\ref{df:DfGather}.   The inverse to the above map would be 
$[\omega'] \mapsto [\omega_X * \omega' * \omega^{-1}_X]$.)  

Since spectral flow $Sf$ satisfies the Concatenation Axiom (SF3) and the
Triviality Principle (see Propositions \ref{prop:SfSatisfiesAxioms} and
\ref{prop:TrivialityPrinciple}),
for all $\omega \in \pi_1(\Ff_{SA, \infty}, X)$,
\begin{eqnarray*}
Sf_*(\Phi([\omega])) & = & Sf_*([\omega^{-1}_X * \omega * \omega_X])\\
& = & Sf_*([\omega^{-1}_X]) + Sf_*([\omega]) + Sf([\omega_X])\\
& = & Sf_*([\omega]).  
\end{eqnarray*}
Hence,
\begin{equation} \label{equ:Aug29202311:5999}  
Sf_* \circ \Phi = Sf_* \end{equation} 
as maps from $\pi_1(\Ff_{SA, \infty}, X)$ to $K_0(\B)$.  
Note that the ``$Sf_*$" on the left of (\ref{equ:Aug29202311:5999})
is the map $\pi(\Ff_{SA, \infty}, 2P_0 - 1) \rightarrow K_0(\B)$ 
(which, by Theorem \ref{thm:Sf*=sfP}, is a group isomorphism), and the
``$Sf_*$" on the right of (\ref{equ:Aug29202311:5999}) is the map
$\pi_1(\Ff_{SA, \infty}, X) \rightarrow K_0(\B)$.

Now since $\Phi$ is a group isomorphism, by Theorem \ref{thm:Sf*=sfP},
$Sf_* \circ \Phi$ (the map on the left in (\ref{equ:Aug29202311:5999}))
is a group isomorphism.  Hence, by (\ref{equ:Aug29202311:5999}),
the map $Sf_* : \pi_1(\Ff_{SA, \infty}, X) \rightarrow K_0(\B)$ (on the
right in (\ref{equ:Aug29202311:5999})) is a group isomorphism.\\ 
\end{proof}

\subsection{Axiomatization}

\label{subsect:Axiomatization}
\makebox{ }\\

\vspace*{2ex}
\textbf{Recall that from this subsection on, we are NO longer assuming
the standing conventions ($\Lambda$) from Subsection 5.1.}\\ 

Also, we point out to the the reader that, in this subsection, we use ``Sf" (upper case S) to denote spectral flow, as defined
in Definition \ref{df:SpectralFlow}; and we use ``sf" (lower case s) to
denote an arbitrary functor.  We also refer the reader to Section
\ref{sec:FunctorialAxioms} for the definitions of various functorial
properties or axioms.\\\\

We begin by restricting the class of (norm-) continuous paths of Fredholm 
operators to one for which our axiomatization would hold.

\begin{df} \label{df:Pg}  For a separable stable C*-algebra $\B$, let  
$\Pg_{SA, \infty}(\Mul(\B))$  
      denote the collection of (norm-) continuous paths $\{ A_t \}_{t \in [0,1]}$
in  
$\Ff_{SA}$ such that 
$A_0$ and $A_1$ are both invertible elements of $\Ff_{SA, \infty}$.

Often, since the context is clear, we drop the ``$\Mul(B)$" and
simply write $\Pg_{SA, \infty}$ instead of $\Pg_{SA, \infty}(\Mul(\B))$.\\
\end{df}

\begin{rmk} \label{rmk:Aug3120231AM} 
Since $\Ff_{SA, \infty}(\Mul(\B))$ is a path-connected and clopen
subset of $\Ff_{SA}(\Mul(\B))$ (by Lemma \ref{lem:Aug520235AM}),  
$\Pg_{SA, \infty}(\Mul(\B))$ is exactly the class of 
(norm-) continuous paths in 
$\Ff_{SA, \infty}(\Mul(\B))$ with invertible endpoints.\\ 
\end{rmk}

\begin{prop}  \label{prop:SFSatisfiesAxiomsOnPg}
For every separable stable C*-algebra $\B$,
the spectral flow
$Sf$ satisfies Axioms (SF1) to (SF4) on $\Pg_{SA, \infty}(\Mul(\B))$.
\end{prop}

\begin{proof}
This follows immediately from Proposition \ref{prop:SfSatisfiesAxioms}.  
(We are just taking Proposition \ref{prop:SfSatisfiesAxioms} 
and restricting the paths to be in $\Pg_{SA, \infty}$. Note that
the paths in the Normalization Axiom (SF4) are already elements
of $\Pg_{SA, \infty}$.)\\ 
\end{proof}

\begin{prop} (Triviality Principle for $\Pg_{SA, \infty}(\Mul(\B)$)  
\label{prop:RestrictedTrivialityPrinciple}
Suppose that for every separable stable C*-algebra $\B$,
we have a map
$$sf : \Pg_{SA, \infty} \rightarrow K_0(\B)$$
such that on $\Pg_{SA, \infty}$,  $sf$ satisfies
Axioms  (SF2) and (SF3).

Then for any separable stable C*-algebra $\B$, for any norm continuous
path $\{ A_t \}_{t \in [0,1]}$ of invertible elements in 
$\F_{SA, \infty}$,  
$$sf(\{ A_t \}_{t \in [0,1]}) = 0.$$
(See also Remark \ref{rmk:Aug3120231AM}.)  
\end{prop}

\begin{proof}
The proof is exactly the same as that of 
Proposition \ref{prop:TrivialityPrinciple}, except with the observation that
all the paths take values in $\Ff_{SA, \infty}$.\\     
\end{proof}

We are now ready to present our axiomatization of spectral flow $Sf$.

\begin{thm}   \label{thm:LaQuinta}
Suppose that for every separable stable C*-algebra $\B$,
we have a map
$$sf : \Pg_{SA, \infty}(\Mul(\B)) \rightarrow K_0(\B)$$
such that on $\Pg_{SA, \infty}(\Mul(\B))$,  $sf$ satisfies 
Axioms (SF1) to (SF4). 

Then for every separable stable C*-algebra $\B$,
$$Sf = sf$$
where $Sf$ is spectral flow $Sf$ (as in Definition \ref{df:SpectralFlow}) 
restricted to 
$\Pg_{SA, \infty}(\Mul(\B))$.   
\end{thm}

\begin{proof}
We follow the argument of \cite{LeschUniqueness}  Theorem 5.4, which was for the $\mathbb{B}(l_2)$ case and considered
a ``folklore result". 

By Proposition \ref{prop:SFSatisfiesAxiomsOnPg}, the spectral flow
 $Sf$ satisfies
Axioms (SF1) to (SF4) on $\Pg_{SA, \infty}$ for every separable
stable C*-algebra $\B$.

Let $\B$ be a separable
stable C*-algebra, and let $\{ A_t \}_{t \in [0,1]}$ be an arbitrary element of 
$\Pg_{SA, \infty}$.  
We will prove that $Sf(\{ A_t \}_{t \in [0,1]}) = sf(\{ A_t \}_{t \in [0,1]})$.

Let $P_0 \in \Mul(\B)$ be a projection such that $P_0 \sim 1 \sim 1 - P_0$.
By Lemma \ref{lem:MakingTrivialPaths}, let $\{ B_t \}_{t \in [0,1]}$ and
$\{ C_t \}_{t \in [0,1]}$ be (norm-) continuous paths of invertible elements
of $\Ff_{SA, \infty}$ such that 
$$B_0 = C_1 = 2 P_0 - 1, \makebox{ } B_1 = A_0 \makebox{  and  }
C_0 = A_1.$$

Consider the (norm-) continuous loop $\{ D_t \}_{t \in [0,1]}$
in $\Pg_{SA, \infty}$, 
based at $2 P_0 - 1$, given by the concatenation
$$\{ D_t \}_{t \in [0,1]} = \{ B_t \}_{t \in [0,1]} * \{ A_t \}_{t \in [0,1]}
* \{ C_t \}_{t \in [0,1]}.$$

By the Triviality Principle on $\Pg_{SA, \infty}$ (Proposition
\ref{prop:RestrictedTrivialityPrinciple}),
$$Sf(\{ B_t \}_{t \in [0,1]}) = Sf(\{ C_t \}_{t \in [0,1]})
= sf(\{ B_t \}_{t \in [0,1]}) = sf(\{ C_t \}_{t \in [0,1]}) = 0.$$
Hence, since both $Sf$ and $sf$ satisfy the Concatenation Axiom (SF3) on $\Pg_{SA, \infty}$,  
$$Sf(\{ D_t \}_{t \in [0,1]}) = 
Sf(\{ B_t \}_{t \in [0,1]}) + Sf(\{ A_t \}_{t \in [0,1]}) + 
Sf(\{ C_t \}_{t \in [0,1]}) = Sf(\{ A_t \}_{t \in [0,1]})$$    
and 
$$sf(\{ D_t \}_{t \in [0,1]}) = 
sf(\{ B_t \}_{t \in [0,1]}) + sf(\{ A_t \}_{t \in [0,1]}) +
sf(\{ C_t \}_{t \in [0,1]}) = sf(\{ A_t \}_{t \in [0,1]}).$$
Hence, to prove that $Sf(\{ A_t \}_{t \in [0,1]}) = sf(\{ A_t \}_{t \in [0,1]})$,
it suffices to prove that $Sf(\{ D_t \}_{t \in [0,1]}) = sf(\{ D_t \}_{t \in [0,1]})$.
By Lemma \ref{lem:ExhaustingPi1},  there exists a projection
$P \in \Mul(\B)$ with $P - P_0 \in \B$ and  
$\{ D_t \}_{t \in [0,1]}$ homotopic to $\omega_{P_0, P}$.
Hence, since both $Sf$ and $sf$ satisfy the Homotopy Axiom (SF2) on $\Pg_{SA, \infty}$,
it suffices to prove that 
$Sf(\omega_{P_0, P}) = sf(\omega_{P_0, P})$.  
(Note that, by definition, 
$\omega_{P_0, P} \in \Pg_{SA, \infty}$.)  

By Definition \ref{df:ThirdExample},  there exists a (norm-) continuous
path $\{ E_t \}_{t \in [0,1]}$, of invertible elements of 
$\F_{SA, \infty}$, with $E_0 = 2P_0 - 1$ and $E_1 = 2P - 1$ 
 such that $\omega_{P_0, P}$ is the concatenation
$$\omega_{P_0, P} = \{ E_t \}_{t \in [0,1]} * 
\{ (1 - t) (2P - 1) + t (2 P_0 - 1) \}_{t \in [0,1]}.$$
Hence, since $sf$ satisfies the Trivilialty Principle on 
$\Pg_{SA, \infty}$ (see Proposition \ref{prop:RestrictedTrivialityPrinciple}),
and since $sf$ satisfies the Concatenation Axiom (SF3) and the Normalization Axiom
(SF4) on $\Pg_{SA, \infty}$,
\begin{eqnarray*}
sf(\omega_{P, P_0}) & = & sf(\{ E_t \}_{t \in [0,1]}) +
sf(\{ (1 - t) (2P - 1) + t (2 P_0 - 1) \}_{t \in [0,1]}) \\
& = &   sf(\{ (1 - t) (2P - 1) + t (2 P_0 - 1) \}_{t \in [0,1]}) \\
& = & [P: P_0].
\end{eqnarray*}
But by by Lemma \ref{lem:ThirdExampleSF},  
$$Sf(\omega_{P, P_0}) = [P: P_0] = sf(\omega_{P,P_0}).$$
Hence, from the above,
$$Sf(\{ A_t \}_{t \in [0,1]}) = sf(\{ A_t \}_{t \in [0,1]}).$$
Hence, since $\{ A_t \}_{t \in [0,1]}$ was arbitrary,
$Sf = sf$ on  $\Pg_{SA, \infty}$.
Since $\B$ was arbitrary, we are done.\\  
\end{proof}

Axiomatization of spectral flow $Sf$ may also be achieved with some 
variation on the axioms, and we now give an example of this.
Towards this, 
we first define a weaker version of the Homotopy Axiom:

\begin{df} \label{df:WeakHomotopyAxiom}                            
Suppose that for every separable stable C*-algebra $\B$, we have a map
$$sf : \Pg_{SA, \infty}(\Mul(\B))  \rightarrow K_0(\B).$$

Then $sf$ is said to 
 satisfy the \emph{Weak Homotopy Axiom} or \emph{Axiom (SF5)} 
if it satisfies the following restricted homotopy invariance:\\ 

Let $\{ B_t \}_{t \in [0,1]}$ and $\{ C_t \}_{t\in [0,1]}$ be two 
elements of $\Pg_{SA, \infty}(\Mul(\B))$ such that   
$B_0 = C_0$ and $B_1 = C_1$.

Suppose that there is a (norm-) continuous family 
$\{ B_{s,t} \}_{(s,t) \in [0,1]\times [0,1]}$ in $\Ff_{SA, \infty}$\
such that 
$$B_{0, t} = B_t \makebox{  and  } B_{1, t} = C_t \makebox{ for all  }
t \in [0,1]$$
and 
$$B_{s,0} = B_0 = C_0 \makebox{   and  } B_{s,1} = B_1 = C_1 
\makebox{  for all  } s \in [0,1].$$

Then $$sf(\{ B_t \}_{t \in [0,1]}) = sf(\{ C_t \}_{t \in [0,1]}).$$\\
\end{df}   

In our axiomatiation of spectral flow $Sf$, we can replace the 
Homotopy Axiom (SF2) with the Weak Homotopy Principle (SF5) and
the Restricted Triviality Principle.  The proof the next result is very similar to that of 
Theorem \ref{thm:LaQuinta}, and we leave it as an exercise for the reader.

\begin{thm}
Suppose that for every separable stable C*-algebra $\B$, we have a map
$$sf : \Pg_{SA, \infty} \rightarrow K_0(\B)$$
such that on $\Pg_{SA, \infty}$, $sf$ satisfies 
Axioms (SF1), (SF3), (SF4), (SF5) and the Restricted Triviality 
Principle (see Proposition
\ref{prop:RestrictedTrivialityPrinciple}).

Then 
$$Sf = sf.$$\\  
\end{thm}

\section{Appendix:  Some miscellaneous and useful computations, and
the DZWLP Spectral Flow}

In this last section, we present some useful items from operator theory
which are elementary, but which are not immediately trivial for a 
beginning analyst and are not so easily found in the standard literature.
We also present our definition of a (generalized) 
Fredholm operator and the (generalized) Fredholm index, for the convenience
of the beginning reader. We also present results from homotopy theory
used in this paper, which go beyond a first graduate course in North American
universities, and for which we cannot find a good reference.   Finally, we briefly summarize the approach
to spectral flow in \cite{Wu} and \cite{LeichtnamPiazzaJFA}.  \\

\subsection{Generalized Fredholm operators and Fredholm index}
\label{subsect:GeneralizedFredholm}

\begin{df} \label{df:FredholmOperator}
Let $\B$ be a nonunital C*-algebra.
An element $X \in \Mul(\B)$ is said to be \emph{Fredholm} 
if $\pi(X)$ is an invertible element of $\C(\B)$.
\end{df}

Let $\B$ be a nonunital C*-algebra. 
Recall that the exact sequence 
$$0 \rightarrow \B \rightarrow \Mul(\B) \rightarrow \C(\B) \rightarrow 0$$
induces a six term exact sequence in K theory (e.g.,
see \cite{WeggeOlsen} Theorem 9.3.2):
\begin{equation}  \label{equ:SixTerm} 
\begin{array}{ccccc}
K_0(\B) & \rightarrow & K_0(\Mul(\B)) & \rightarrow & K_0(\C(\B)) \\
\partial \uparrow &            &               &             & 
\downarrow \partial_1 \\
K_1(\C(\B)) & \leftarrow &  K_1(\Mul(\B)) & \leftarrow & K_1(\B) \\
\end{array}
\end{equation} 

The map $\partial : K_1(\C(\B)) \rightarrow K_0(\B)$ is a generalization
of the Fredholm index from the $\mathbb{B}(l_2)$ case (e.g.,
see \cite{WeggeOlsen} Remark 8.1.4) and is often called the \emph{index
map}.  The map $\partial_1 : K_0(\C(\B)) \rightarrow K_1(\B)$ is called
the \emph{exponential map} because it can be realized as a certain type 
of exponential (see \cite{WeggeOlsen} Exercise 9E).  $\partial_1$ is sometimes
considered an index map, because it is actually the composition of $\partial$
with the Bott map and a similar map (see \cite{WeggeOlsen} Definition
9.3.1).   Indeed, the vanishing index characterizations for the existence
of spectral sections, in  \cite{Wu}
and \cite{LeichtnamPiazzaJFA},  involve the map $\partial_1$.  (See also Definition \ref{df:K1Index} and Subsection 
\ref{subsect:LiftingProjectionsCondition} in general.)

Note additionally that if $\B$ is $\sigma$-unital and stable then
$K_0(\Mul(\B)) = K_1(\Mul(\B)) = 0$  (see \cite{WeggeOlsen} Theorem
16.8), and so by the exactness of (\ref{equ:SixTerm}),
the index map $\partial : K_1(\C(\B)) \rightarrow K_0(\B)$ and the
exponential map $\partial_1 : K_0(\C(\B)) 
\rightarrow K_1(\B)$ are both group isomorphisms.\\

\begin{df} \label{df:FredholmIndex}  
Let $\B$ be a nonunital C*-algebra, and let $\partial : K_1(\C(\B))
\rightarrow K_0(\B)$ be the index map from (\ref{equ:SixTerm}).

Let $X \in \Mul(\B)$ be a Fredholm operator (i.e., $\pi(X)$ is an 
invertible element of $\C(\B)$).
Then the \emph{Fredholm index} of $X$ is defined to be
$$\partial([\pi(X)]) \in K_0(\B).$$\\  
\end{df}

\subsection{$K_1$-injectivity and $K_1$-surjectivity}  

\label{subsect:CoronaK1InjAndSurj}

We will need some results concerning the $K_1$-injectivity and 
surjectivity of C*-algebras.  Some references for this material
are \cite{BlanchardRohdeRordam} and \cite{RohdeThesis}.

Recall that for a unital C*-algebra $\C$, $U(\C)$ is the unitary
group of $\C$ and $U(\C)_0$ is the path-connected component of $1$
in $U(\C)$.\\

\begin{df}
Let $\C$ be unital C*-algebra.
\begin{enumerate}
\item $\C$ is said to be \emph{$K_1$-injective} if the usual map
$$U(\C)/U(\C)_0 \rightarrow K_1(\C)$$
is injective.
\item  $\C$ is said to be \emph{$K_1$-surjective} if the usual map
$$U(\C)/U(\C)_0 \rightarrow K_1(\C)$$
is surjective.\\ 
\end{enumerate}
\end{df}

The next lemma concerns the $K_1$-injectivity (and surjectivity) 
of certain properly
infinite C*-algebras.  
This result should be known, but we have not found a precise reference and hence, provide a
proof for the convenience of the reader.
We further note that it is an interesting open
problem, of Blanchard, Rohde and Roerdam, whether every properly
infinite C*-algebra is 
$K_1$-injective (\cite{BlanchardRohdeRordam}).  
For example, $K_1$-injectivity
of Paschke dual algebras (which are properly infinite) would imply
some interesting KK uniqueness theorems.  (See, for example,
\cite{LoreauxNgSutradhar}.)

\begin{lem}
Let $\B$ be a separable stable C*-algebra and let $K$ be a compact 
Hausdorff topological space.

Then $C(K) \otimes \C(\B)$ is both $K_1$-injective and $K_1$-surjective. 
\label{lem:CoronaK1InjAndSurj}  
\end{lem}

\begin{proof}  It is not hard to see that a unital properly infinite
C*-algebra is always $K_1$-surjective (e.g., see \cite{CuntzKTh}).
Hence, $C(K) \otimes \C(\B)$ is $K_1$-surjective.

The proof that $C(K) \otimes \C(\B)$ is $K_1$-injective is a variation
on the argument of \cite{ChandNgSutradhar} Theorem 4.9 (7).  We here
provide the argument for the convenience of the reader.  Suppose
that $p, q \in C(K) \otimes \C(\B)$ are both full properly
infinite projections.
Hence, let $x \in C(K) \otimes \C(\B)$ for which 
$$x p x^* = 1_{C(K) \otimes \C(\B)}.$$  Let $X, A \in C(K) \otimes \Mul(\B)$ with $A \geq 0$
be such that
$\pi(X) = x$ and $\pi(A) = p$.  
Hence, $$X A X^* - 1_{C(K) \otimes \Mul(\B)} = c \in C(K) \otimes \B.$$
Since $\B$ is stable, we can find an isometry $V \in \Mul(\B)$ such that
if we define $Y =_{df} 1_{C(K)} \otimes V$ then
$$\| Y^* c Y \| < \frac{1}{10}.$$
Hence, $Y^*  X A X^* Y$ is within $\frac{1}{10}$ of $1_{C(K) \otimes
\Mul(\B)}$.  Hence, we can find $Y' \in C(K) \otimes \Mul(\B)$ for which
$$Y' Y^* X A X^* Y (Y')^* = 1_{C(K) \otimes \Mul(\B)}.$$
Hence, we can find a projection $R \in \overline{A (C(K) \otimes \Mul(\B))
A}$ such that $R \sim 1_{C(K) \otimes \Mul(\B)}$.  Moreover, since
$1_{C(K) \otimes \Mul(\B)}$ is properly infinite, we may assume that
$1_{C(K) \otimes \Mul(\B)} - R \sim 1_{C(K) \otimes \Mul(\B)}$
(e.g., see \cite{WeggeOlsen} Lemma 16.2).
Similarly, we can find $B \in C(K) \otimes 
\Mul(\B)$ with $B \geq 0$ and a projection
$S \in \overline{B(C(K) \otimes \Mul(\B))B}$ such that 
$\pi(B) = q$ and 
$$S \sim 1_{C(K) \otimes \Mul(\B)} \sim 1_{C(K) \otimes \Mul(\B)} - S.$$ 
Hence, we can find a unitary $U \in C(K) \otimes \Mul(\B)$ such that
$$U R U^* = S.$$
By \cite{WeggeOlsen} Corollary 16.7, the unitary group
$U(C(K) \otimes \Mul(\B))$ is path-connected.  Hence, 
$R$ is homotopic to $S$ in $Proj(C(K) \otimes \Mul(\B))$.
Hence, $\pi(R)$ and $\pi(S)$ are full properly infinite
subprojections of $p$ and $q$ respectively such that
$\pi(R)$ is homotopic to $\pi(S)$ in $Proj(C(K) \otimes \C(\B))$.
Since $p, q$ were arbitrary, by \cite{BlanchardRohdeRordam} Proposition
5.1, $C(K) \otimes \C(\B)$ is $K_1$-injective.\\ 
\end{proof}

\subsection{Two short, elementary computations in operator theory}

The next two computations are useful and elementary, but  
they are not as well-known to beginning students
of the subject as they should be. Hence, for the convenience of the reader,
 we present their short proofs.

\begin{lem}  \label{lem:Aug520233pm}  Let $\B$ be a separable, stable C*-algebra.
Say that $P \in \C(\B)$ is a projection such that $$P \sim 1_{\C(\B)} \sim 1_{\C(\B)} - P.$$

Then there exists a projection $P' \in \Mul(\B)$ such that $\pi(P') = P$ and
$$P' \sim 1_{\Mul(\B)} \sim 1_{\Mul(\B)} - P'.$$
\end{lem}

\begin{proof}
Since $\B$ is stable, we can find a projection $P_0 \in \Mul(\B)$ for which
$$P_0 \sim 1_{\Mul(\B)} \sim 1_{\Mul(\B)} - P_0.$$
Hence, by hypothesis,
$$\pi(P_0) \sim P \makebox{  and  } 1 - \pi(P_0) \sim 1 - P.$$
Hence, choose partial isometries $V, W \in \C(\B)$ for which
$V^* V = \pi(P_0)$, $VV^* = P$, $W^*W = 1 - \pi(P_0)$ and
$W W^* = 1 - P$.  Let $U' \in \C(\B)$ be the unitary given by 
$U' =_{df}  V + W$.  So $U' \pi(P_0) (U')^* = P$.
Since $\pi(P_0) \sim 1$, let $u \in \pi(P_0) \C(\B) \pi(P_0)$ be a unitary such that
$[u^*] = [U']$ in $K_0(\C(\B))$.  Let $U =_{df} U' (u + (1 - \pi(P_0)))$.  Then 
$U \in \C(\B)$ is a unitary such that $[U] = 0$ in $K_1(\C(\B))$.  
Since $\B$ is stable,
$\C(\B)$ is $K_1$-injective (see Lemma \ref{lem:CoronaK1InjAndSurj}).
Hence, $U$ is (norm-) path-connected to $1$ in $U(\C(\B))$.
Hence, $U$ is liftable, i.e., let $U_0 \in \Mul(\B)$ be a unitary such that 
$\pi(U_0) = U$  (see \cite{WeggeOlsen} Corollary 4.3.3).
Hence,
$$\pi(U_0 P_0 U_0^*) = 
U \pi(P_0) U^* = U'(u + (1 - \pi(P_0)) \pi(P_0) (u^* + (1 - \pi(P_0)) {U'}^* =   U' \pi(P_0) {U'}^* = P.$$
So let $$P' =_{df} U_0  P_0 U_0^*.$$\\
\end{proof}

\begin{lem}  \label{lem:AShortComputation}
Let $\B$ be a separable stable C*-algebra. If $P, Q \in \Mul(\B)$ are
projections such that $P \sim 1$ and $P - Q \in \B$, then
$Q \sim 1$.
\end{lem}

\begin{proof}
Since $P - Q \in \B$, let $b \in \B$ be such that $Q = P + b$.
Since $P \sim 1$ and since $\B$ is stable, let $\{ P_n \}_{n=1}^{\infty}$
be a sequence of pairwise orthogonal projections in $\Mul(\B)$, with
$P_n \sim 1$ for all $n$, such that $P = \sum_{n=1}^{\infty} P_n$, where
the sum converges strictly in $\Mul(\B)$.   
Choose $N \geq 1$ such that 
$$\left\|\sum_{n=N}^{\infty} P_n b \right\| < \frac{1}{10}.$$ 
By \cite{WeggeOlsen} Lemma 16.2, $\sum_{n=N}^{\infty} P_n \sim 1$.
Hence, let $V \in \Mul(\B)$ be a partial isometry such that 
$V^*V =  \sum_{n=N}^{\infty} P_n$ and $VV^* = 1$. 
So 
$$V Q V^* = V (P + b) V^* \approx_{\frac{1}{10}} V P V^* = 1.$$
Since $V Q V^*$ is norm witin $\frac{1}{10}$ of $1$, it is a positive
invertible. Hence, there exists $X \in \Mul(\B)$ such that
$X V Q V^* X^* = 1$.  Hence, $W =_{df} X V Q \in \Mul(\B)$ is such that
$W W^* = 1$ and $W^* W = Q V^* X^* X V Q \leq Q$.  (Note that
$\| W \|= 1$ since $WW^* = 1$.)    Hence, $1$ is Murray--von Neumann
subequivalent to $Q$.  Hence, by \cite{WeggeOlsen} Lemma 16.2,
$Q \sim 1$.
\end{proof}

\subsection{A weak homotopy equivalence} 
\label{subsect:KappaIsHomotopyEquivalence}
In this subsection, we provide some details from topology (especially homotopy
theory) that is used to prove that $\kappa$ (as  in Definition
\ref{df:PereraIsomInverseRealSecondComp}) is a homotopy equivalence.  These
details should be known to experts in homotopy, but we could not find a place
where it is explicitly and clearly presented in an elementary text, and in 
fact, we even found a mistake in the literature.  Thus, since this paper is
directed to analysts, we will provide explicit, detailed computations and 
full references.\\

The first result is essentially \cite{HatcherTopologyBook} Proposition
4.66.  However, there is a mistake in their statement (as $\pi_0$
need not be a group, and (even if $\pi_0$ is a group)  
their computation does not even  give a bijection
at the level of $\pi_0$).  For the convenience of the reader, we here
provide the corrected statement as well as (a slightly more detailed version 
of) the short proof.

\begin{lem}  \label{lem:HatcherPi_nIsom}
Let $E, B$ be topological spaces and $\alpha : E \rightarrow B$ a 
Hurewicz fibration.  Suppose that $b_0 \in B$,
and $e_0 \in F =_{df} \alpha^{-1}(b_0)$.

Suppose, in addition, that $E$ is contractible. 
Let $\kappa : F \rightarrow \Omega_{b_0} B$ be defined as in 
Definition \ref{df:PereraIsomInverseRealSecondComp}.

Then for all $n \geq 1$ and for all $f \in F$, the induced map
$$\kappa_* : \pi_n(F, f) \rightarrow \pi_n(\Omega_{b_0} B, \kappa(f))$$
is a group isomorphism.
\end{lem}

\begin{proof}
Let $P_{b_0} B$ be the set of 
 continous paths in $B$ with right endpoint $b_0$, i.e.,
$$P_{b_0} B =_{df} \{ \omega \in C([0,1], B) : \omega(1) = b_0 \}.$$ 

Since $E$ is contractible, fix a contraction of $E$ to $e_0$ (i.e., a homotopy
between the identity map $id_E : E \rightarrow E$ and the constant map
$E \rightarrow \{ e_0 \}$).  
Define a map
$$\overline{\kappa} : E \rightarrow P_{b_0}(B)$$
in the following manner:  For each $e \in E$, the contraction induces
a continuous path $\widetilde{\omega}$ from $e$ to $e_0$ (so 
$\widetilde{\omega}(0) = e$ and $\widetilde{\omega}(1) = e_0$); and
we defined 
$$\overline{\kappa}(e) =_{df} \alpha \circ \widetilde{\omega}.$$
Note that $$\overline{\kappa}|_F = \kappa.$$
(See Remark \ref{rmk:Sept520231AM}.  Note that the same argument also shows
that for any $e \in E$, $\overline{\kappa}(e)$ is also well-defined up to
homotopy in $\Omega_{b_0}B$.)   

We then have the following commuting diagram of continuous maps, which is a
morphism of Hurewicz fibrations  (see \cite{Switzer} Proposition 4.3):
\begin{equation}  \label{equ:Sept620231AM}
\begin{array}{ccccc}
F & \rightarrow & E & \stackrel{\alpha}{\rightarrow} & B \\
\kappa \downarrow & &  \overline{\kappa} \downarrow & & ||\\
\Omega_{b_0} B & \rightarrow & P_{b_0} B & \stackrel{ev_0}{\rightarrow} & B
\end{array}
\end{equation}
where $ev_0$ is the evaluation at $0$ map, i.e., for all $\omega \in 
P_{b_0}B$, $ev_0(\omega) =_{df} \omega(0)$.  

Let $f \in F$ be arbitrary. Then  (\ref{equ:Sept620231AM}) induces a map
from the long exact sequence of homotopy groups (corresponding to
$f \in F$) for the
Hurewicz fibration $F \rightarrow E \stackrel{\alpha}{\rightarrow} B$
to the long exact sequence of homotopy groups (corresponding to 
$\kappa(f) \in \Omega_{b_0} B$) for the Hurewicz   
fibration $\Omega_{b_0} B \rightarrow P_{b_0} B \rightarrow B$ (e.g., see 
\cite{Spanier} Chapter 7, Section 2, Theorem 10): 
\begin{equation} \label{equ:Sept620232AM}
\begin{array}{cccccccccc}
\rightarrow & \pi_n(E, f) & \rightarrow & \pi_n(B, b_0) & 
\rightarrow & \pi_{n-1}(F, f) & \rightarrow & \pi_{n-1}(E, f) & \rightarrow &
\cdot \cdot \cdot  \\ 
& \downarrow & & || & & \kappa_* \downarrow & & \downarrow & & \cdot \cdot  
\cdot  \\
\rightarrow & \pi_n(P_{b_0}, \overline{\kappa}(f)) & \rightarrow & 
\pi_n(B, b_0) &
\rightarrow & \pi_{n-1}(\Omega_{b_0} B, \kappa(f)) & \rightarrow & 
\pi_{n-1}(P_{b_0}B, \overline{\kappa}(f)) & \rightarrow &
\cdot \cdot \cdot  \\
\end{array}
\end{equation}
The above is a commuting diagram where the horozontal rows are
exact sequences.  
By hypothesis, $E$ is contractible, and by \cite{Switzer} Proposition 4.4,
$P_{b_0} B$ is also contractible.  Hence, for all $m \geq 0$,
$$\pi_m(E, f) = \pi_m(P_{b_0} B, \overline{\kappa}(f)) = 0.$$
From this and the diagram (\ref{equ:Sept620232AM}),  for all $m \geq 1$,
the map
$$\kappa_* : \pi_m(F, f) \rightarrow \pi_m(\Omega_{b_0} B, \kappa(f))$$
is a group isomorphism.
Since $f \in F$ was arbitrary, we are done.
(Note that the case of $\pi_0$ is problematic because $\pi_0$ need not be 
a group, and (even if $\pi_0$ is a group) the sequence at the level of
$\pi_0$ may only be an exact sequence of pointed sets (e.g., 
see \cite{Switzer} 2.29).)\\
\end{proof}

Note that by Whitehead's Theorem (\cite{Spanier} Chapter 7, Section 6, 
Corollary 24;  see also \cite{HatcherTopologyBook} Theorem 4.5), and by 
Lemma \ref{lem:HatcherPi_nIsom}, to show that $\kappa$ is a 
weak homotopy equivalence, 
it suffices to show that $\kappa_* : \pi_0(F) \rightarrow \pi_0(\Omega_{b_0}B)$
is a bijection.  
The rest of this subsection will be a proof of this (that $\kappa_*$
induces a bijection at the level of $\pi_0$). Towards this, we will take a 
detour through path-lifting functions and construct a map which will have
the same $\pi_0$ as $\gamma$ (a homotopy inverse of $\kappa$, which
we do not know yet exists at this point in the paper).\\

\begin{thm} \label{thm:PathLiftingFunctionExists}
Let $E$ and $B$ be topological spaces,  let
 $\alpha : E \rightarrow B$ be a continuous map, and
let 
\begin{equation}  \label{equ:Sept620237AM}
Z =_{df} \{ (e, \omega) \in E \times C([0,1], B) : \omega(0) = \alpha(e) \}.
\end{equation} 
     
Then $\alpha$ is a Hurewicz fibration if and only if there exists a continuous map
$\lambda : Z \rightarrow C([0,1], E)$ such that
\begin{equation} \label{equ:Sept120231AM} 
\lambda(e, \omega)(0) = e \makebox{  and  } \alpha \circ \lambda(e, \omega) = 
\omega \makebox{ for all  } (e, \omega) \in Z. \end{equation} 

In the above, $\lambda$ is called a \emph{path lifting function} for $\alpha$.
\end{thm}

\begin{proof}
This is \cite{Spanier} Chapter 2 Theorem 8.  See also 
\cite{Switzer} Exercise 4.22.\\
\end{proof}

\begin{df} \label{df:PereraIsom2ndPart}
Let $E$ and $B$ be topological spaces, and 
suppose that $\alpha : E \rightarrow B$ is a Hurewicz fibration.
Let $b_0 \in B$ and $e_0 \in F =_{df} \alpha^{-1}(b_0)$.   

Let $\lambda : Z \rightarrow C([0,1], E)$ be path lifting function for $\alpha$
as in Theorem \ref{thm:PathLiftingFunctionExists}.

We define a map 
$\gamma_{\lambda} : \Omega_{b_0} B \rightarrow F$ by 
$$\gamma_{\lambda}(\omega) =_{df} (\lambda(e_0, \omega^{-1}))(1) 
\makebox{  for all  } \omega \in \Omega B.$$\\
\end{df}

We will show that, in our setting, $\gamma_{\lambda}$, as in 
Definition \ref{df:PereraIsom2ndPart}, gives a ``$\pi_0$ inverse" 
of $\kappa$ (as in Definition 
\ref{df:PereraIsomInverseRealSecondComp}).  
This will then show that $\kappa$ induces a bijection (actually
a group isomorphism)  in $\pi_0$.   
Combined with Lemma \ref{lem:HatcherPi_nIsom}, we will have that, in our
setting,  $\kappa :
F \rightarrow \Omega_{b_0} B$ is a weak homotopy equivalence.\\

\begin{prop}  \label{prop:UniqueLifting}
Let $\B$ be a separable stable C*-algebra, let $P_0 \in \Mul(\B)$ be a 
projection such that $P_0 \sim 1_{\Mul(\B)} \sim 1_{\Mul(\B)} - P_0$, and
let $\alpha_{P_0} : U(\Mul(\B)) \rightarrow \ProI (\C(\B))$ be the Hurewicz
fibration given by Definition \ref{df:AFibration}  (see also Theorem 
\ref{thm:AFibration}).

Suppose that $\omega : [0,1] \rightarrow \ProI(\C(\B))$
is a (norm-) continuous path such that 
$\omega(0) = \omega(1) = \pi(P_0)$.   

Suppose that 
 $\widetilde{\omega}_0, \widetilde{\omega}_1 : 
[0,1]   \rightarrow U(\Mul(\B))$ are  (norm-) continuous
paths such that
$$\widetilde{\omega}_0(0) = \widetilde{\omega}_1(0) = 1$$
and
$$\alpha_{P_0} \circ \widetilde{\omega}_0 = \omega = 
\alpha_{P_0} \circ \widetilde{\omega}_1.$$

Then $\widetilde{\omega}_0(1)$ and $\widetilde{\omega}_1(1)$
are both connected by a (norm-) continuous path in 
$U_{P_0}(\Mul(\B)) = \alpha_{P_0}^{-1}(\pi(P_0))$. 
(See Definition \ref{df:AFibration}.)
\end{prop}

\begin{proof}
Let $G =_{df} U(\Mul(\B))$ and $F =_{df} U_{P_0}(\Mul(\B)) 
= \alpha_{P_0}^{-1}(\pi(P_0))$, with both given the restriction of
the norm topology from $\Mul(\B)$ (and thus both $G$ and $F$ are
topological groups). In fact, $F$ is a closed subgroup of $G$.   
Recall, from Theorem \ref{thm:AFibration}, that 
$\alpha_{P_0} : G \rightarrow \ProI(\C(\B))$ is a principal
$F$-bundle. 
By \cite{Steenrod} pages 30-31, sections 7.3 and 7.4 (e.g., see Theorem 7.3),
we can replace the principal F-bundle  
$\alpha_{P_0} : G \rightarrow \ProI(\C(\B))$ with the 
quotient map $\alpha : G \rightarrow G/F$,    
which is also a principal F-bundle (i.e., the obvious map
$G/F \rightarrow \ProI$ is a homeomorphism, and
$\alpha_{P_0}$ and $\alpha$ 
are topologically equivalent).  Here,  $G/F$  
is the left coset space of $F = U_{P_0}(\Mul(\B))$.  
We then replace $\alpha_{P_0}(1) = \pi(P_0)$ with
the coset $\alpha(1) = 1F = F$ in $G/F$. 
(Here, $1$ is the unit of $G$ and hence, of $F$.)

So let $\omega : [0,1] \rightarrow G/F$ be a (norm-) continuous map
with $\omega(0) = \omega(1) = 1F = F$.
Suppose that $\widetilde{\omega}_0, \widetilde{\omega}_1 :
[0,1]   \rightarrow G$ are  (norm-) continuous
paths such that
$$\widetilde{\omega}_0(0) = \widetilde{\omega}_1(0) = 1$$
and
\begin{equation} \label{equ:May820231AM}   
\alpha \circ \widetilde{\omega}_0 = \omega =
\alpha \circ \widetilde{\omega}_1.
\end{equation}   
Note that $\widetilde{\omega}_0(1), \widetilde{\omega}_1(1) \in 
\alpha^{-1}(1F) = F$. 

By (\ref{equ:May820231AM}),
for all $t \in [0,1]$,
$$\widetilde{\omega}_1(t)^{-1} \widetilde{\omega}_0(t) \in F.$$ 
Hence, let 
$f : [0,1] \rightarrow F$ be the (norm-) continuous path such that
$$f(t) =_{df} \widetilde{\omega}_1(1) \widetilde{\omega}_1(t)^{-1} \widetilde{\omega}_0(t) 
\in F \makebox{  for all  } t \in [0,1].$$
Note that $f(0) = \widetilde{\omega}_1(1)$ and 
$f(1) =  \widetilde{\omega}_1(1) \widetilde{\omega}_1(1)^{-1} \widetilde{\omega}_0(1) = 
\widetilde{\omega}_0(1)$.
Hence, 
$$f: [0,1] \rightarrow F$$       
is a (norm-) continuous path in $F$ that connects
$\widetilde{\omega}_1(1)$ to $\widetilde{\omega}_0(1)$.\\ 
\end{proof}

\begin{lem} \label{lem:kappaPi0Bijection}
Let $\B$ be a separable stable C*-algebra, let $P_0 \in \Mul(\B)$ be a 
projection such that $P_0 \sim 1_{\Mul(\B)} \sim 1_{\Mul(\B)} - P_0$, and
let $\alpha_{P_0} : U(\Mul(\B)) \rightarrow \ProI(\C(\B))$ be the Hurewicz
fibration defined in Definition \ref{df:AFibration}  (see also Theorem
\ref{thm:AFibration}).

Let $E =_{df} U(\Mul(\B))$, $B =_{df} \ProI(\C(\B))$, $b_0 =_{df} \pi(P_0)
\in B$, $\alpha =_{df} \alpha_{P_0}$, $e_0 =_{df} 1_{\Mul(\B)}
\in F =_{df} \alpha^{-1}(\pi(P_0)) = U_{P_0}(\Mul(\B))$.

Since $U(\Mul(\B))$ is (norm-) contractible 
(e.g., see \cite{WeggeOlsen} Theorem 16.8), let 
$$\kappa : F \rightarrow \Omega_{b_0} B$$
be defined as in Definition \ref{df:PereraIsomInverseRealSecondComp}.   

Since $\alpha =_{df} \alpha_{P_0}$ is a Hurewicz fibration, let
$\gamma_{\lambda} : \Omega_{b_0} B \rightarrow F$  (for a given
path-lifting function $\lambda$) be defined as in Definition
\ref{df:PereraIsom2ndPart}.

Now consider the induced maps $\kappa_* : \pi_0(F) \rightarrow
\pi_0(\Omega_{b_0} B)$ and $(\gamma_{\lambda})_* : \pi(\Omega_{b_0} B)
\rightarrow \pi(F)$.  

We then have the following: 
$$(\gamma_{\lambda})_* \circ \kappa_* = id_{\pi_0(F)} 
\makebox{  and  }  
\kappa_* \circ (\gamma_{\lambda})_* = id_{\Omega_{b_0}B}.$$

As a consequence, 
$$\kappa_* : \pi_0(F) \rightarrow \pi_0(\Omega_{b_0} B)$$
is a bijection.   
\end{lem}

\begin{proof}

Let 
$Z$ and the path-lifting function 
$\lambda : Z \rightarrow C([0,1], E)$ be as
in  Theorem \ref{thm:PathLiftingFunctionExists} such that $\lambda$ induces the 
map $\gamma_{\lambda}$ (see Definition \ref{df:PereraIsom2ndPart}).
    
Also, since $E = U(\Mul(\B))$ is contractible (see 
\cite{WeggeOlsen} Theorem 16.8), fix a contraction from $E$ to 
$e_0 = 1_E = 1_F  = 1_{\Mul(\B)}$
(i.e., a homotopy from $id_E : E \rightarrow E$ to the constant map
$E \rightarrow \{ e_0 \}$) such that this contraction induces $\kappa$
(see Definition \ref{df:PereraIsomInverseRealSecondComp}). 

Let us first prove that $(\gamma_{\lambda})_* \circ \kappa_*  = id_{\pi_0(F)}$.
Let $f \in F$ be arbitrary.  The contraction induces a continuous path
$\omega_f$ from $f$ to $e_0 =  1_{\Mul(\B)}$ (i.e., $\omega_f(0) = f$ and 
$\omega_f(1) = e_0$).   By Definition \ref{df:PereraIsomInverseRealSecondComp},
\begin{equation}  \label{equ:Sept720231AM}
\kappa(f) =_{df} \alpha \circ \omega_f \in \Omega_{b_0} B.
\end{equation}                    
Hence, by Definition \ref{df:PereraIsom2ndPart},
\begin{equation} \label{equ:Sept720232AM}
\gamma_{\lambda}(\kappa(f)) =  
\lambda(1_{\Mul(\B)}, (\alpha \circ \omega_f)^{-1})(1) \in F.
\end{equation} 
But by the definition of $\lambda$  (see Theorem  
\ref{thm:PathLiftingFunctionExists}), 
\begin{align}  \label{equ:Sept720233AM}
\alpha \circ \lambda(1_{\Mul(\B)}, (\alpha \circ \omega_f)^{-1})
= (\alpha \circ \omega_f)^{-1} = \alpha \circ (\omega_f^{-1}) \\
\makebox{  and  } \lambda(1_{\Mul(\B)}, (\alpha \circ \omega_f)^{-1})(0)
= 1_{\Mul(\B)}. 
\end{align} 
And also, by the definition of $\omega_f$,
\begin{equation} \label{equ:Sept720234AM}
\omega^{-1}_f(0) = 1_{\Mul(\B)} \makebox{ and  } \omega^{-1}_f(1) = f.
\end{equation}
By Proposition \ref{prop:UniqueLifting}, (\ref{equ:Sept720232AM}),
(\ref{equ:Sept720233AM}) and (\ref{equ:Sept720234AM}),  we must have that
$\gamma_{\lambda}(\kappa(f))$ is in the same (path-connected) component
as $\omega^{-1}_f(1) = f$ in $F$.   
Since $f \in F$ was arbitrary, we have that for all $f \in F$,  
$$(\gamma_{\lambda})_* \circ \kappa_*([f]) = [f]$$ 
where $[f] \in \pi_0(F)$ is the (path-) connected component of $f$ in $F$.\\  

Next, let us prove that $\kappa_* \circ (\gamma_{\lambda})_* 
= id_{\Omega_{b_0} B}$. 
Let $\omega \in \Omega_{b_0}B$ be arbitrary. 
By the definition of $\gamma_{\lambda}$ (Definition \ref{df:PereraIsom2ndPart}),
\begin{equation} \label{equ:Sept720235AM}
\gamma_{\lambda}(\omega) =_{df} 
\lambda(1_{\Mul(\B)}, \omega^{-1})(1) \in F,
\end{equation}
and (by Theorem \ref{thm:PathLiftingFunctionExists}), $\lambda(1_{\Mul(\B)},
\omega^{-1})$ is a continuous path in $E = U(\Mul(\B))$ with
\begin{equation} \label{equ:Sept720236AM}
\alpha \circ \lambda(1_{\Mul(\B)}, \omega^{-1}) = \omega^{-1} 
\makebox{  and  }
\lambda(1_{\Mul(\B)}, \omega^{-1})(0) = 1_{\Mul(\B)}.
\end{equation}

Now the contraction of $E$ (to $e_0 = 1_{\Mul(\B)}$), defined in the second
paragraph of this proof, induces a continuous path $\widetilde{\omega}$ in 
$E$ that connects   $\lambda(1_{\Mul(\B)}, \omega^{-1})(1)$ to $e_0 = 
1_{\Mul(\B)}$, i.e., 
\begin{equation} \label{equ:Sept720237AM}
\widetilde{\omega}(0) = \lambda(1_{\Mul(\B)}, \omega^{-1})(1) 
\makebox{  and  }
\widetilde{\omega}(1) = 1_{\Mul(\B)}.
\end{equation} 
Moreover, by Definition \ref{df:PereraIsomInverseRealSecondComp} and Definition
\ref{df:PereraIsom2ndPart},
\begin{equation} \label{equ:Sept720238AM}
\kappa(\gamma_{\lambda}(\omega)) =_{df} 
\alpha \circ \widetilde{\omega} \in \Omega_{b_0} B.
\end{equation}

From (\ref{equ:Sept720236AM}) and (\ref{equ:Sept720237AM}), we see
that $\widetilde{\omega}$ and 
$\lambda(1_{\Mul(\B)}, \omega^{-1})^{-1}$ 
are two continuous paths in $E$ with the same 
endpoints.  But since $E = U(\Mul(\B))$ is (norm-) contractible, $E$ is
simply connected.  Hence, there is a homotopy between $\widetilde{\omega}$
and $\lambda(1_{\Mul(\B)}, \omega^{-1})^{-1}$, in $E$, which fixes the
endpoints.    
Hence, by (\ref{equ:Sept720236AM}) and (\ref{equ:Sept720238AM}),
there is a homotopy between 
$\kappa(\gamma_{\lambda}(\omega)) = \alpha \circ \widetilde{\omega}$ and 
$$\alpha \circ \lambda(1_{\Mul(\B)}, \omega^{-1})^{-1} = \omega$$
in $\Omega_{b_0} B$.
Since $\omega$ was arbitrary, we have that          
for all $\omega \in \Omega_{b_0} B$,
$$\kappa_* \circ (\gamma_{\lambda})_* ([\omega])  = [\omega]$$
where $[\omega]$ is the (path-connected) component of 
$\omega$ in $\pi_0(\Omega_{b_0} B)$.\\  

\end{proof}

From Lemmas \ref{lem:HatcherPi_nIsom} and \ref{lem:kappaPi0Bijection},
we have that (in our context),
the map $\kappa : U_{P_0}(\Mul(\B)) \rightarrow \Omega_{b_0} \ProI (\C(\B))$
is a weak homotopy equivalence  (see \cite{Switzer} Definition 3.17).\\

\subsection{The DZWLP approach to spectral flow}
\label{subsection:DZWLPDefinition}
Here, we very briefly summarize parts of  the approach to spectral flow 
in \cite{Wu} and \cite{LeichtnamPiazzaJFA} (see also
\cite{DaiZhang} and \cite{WahlSpectralFlow}).  
``DZWLP" abbreviates ``Dai--Zhang--Wu--Leichtnam--Piazza". 
We will focus on the bounded case and
 only present quickly the major definitions without going into nontrivial 
discussions of existence (and without giving examples).  We also recall that in the original theory in
\cite{Wu} for the unbounded case, there was a need to assume that the unbounded
self-adjoint operators had compact resolvent, but this was not needed (and
not true) for the bounded case.

\begin{df} \label{df:SpectralCutSection}
Let $\B$ be a separable stable C*-algebra, and let $X \in \Mul(\B)$
be a self-adjoint Fredholm operator (i.e., $\pi(X)$ is invertible in 
$\C(\B)$). 

\begin{enumerate}
\item A \emph{spectral cut} of $X$ is a continuous function $\chi : \mathbb{R}
\rightarrow [0,1]$ such that 
there exists $r, s \in \mathbb{R}$ with
$$\max(sp(\pi(X)_-) < r < s < \min(sp(\pi(X)_+)$$
for which 
$$\chi(t) = \begin{cases} 1  & \makebox{ if  } t > s \\
0 & \makebox{ if } t < r.
\end{cases}$$  
\item A projection $P \in \Mul(\B)$ is a \emph{spectral section} of $X$ 
if there exist spectral cuts $\chi_1$ and $\chi_2$ for $X$, with
$$\chi_2 \chi_1 = \chi_1,$$
 such that
 $$\chi_1(X) \leq P \leq \chi_2(X)$$  
or, equivalently, $$P \chi_1(X) = \chi_1(X) \makebox{  and  } \chi_2(X)P = P.$$
\end{enumerate}
\end{df}

As mentioned in this paper (e.g., see the Introduction and Subsection \ref{subsect:LiftingProjectionsCondition}), the existence
of spectral sections is a very nontrivial question, which we will not discuss
in this appendix. Note though how similar the definition of spectral
sections, as in Definition \ref{df:SpectralCutSection} part (2), 
is to interpolation of projection results associated to real rank zero
C*-algebras, especially real rank zero multiplier and corona algebras
(e.g., see \cite{BrownPedersen} Theorem 2.6 and \cite{BrownInterpolation};  see again
Subsection \ref{subsect:LiftingProjectionsCondition}).
Note also that it follows immediately from Definition 
\ref{df:SpectralCutSection} that if $P, Q \in \Mul(\B)$ are both 
spectral sections of $X$, then $P - Q \in \B$; and so if, in addition,
$Q \leq P$, then $P - Q$ is a projection in $\B$.

\begin{df} \label{df:DifferenceClass}
Let $\B$ be a separable stable C*-algebra, and let $X \in \Mul(\B)$ be a 
self-adjoint Fredholm operator.  Suppose that $P, Q \in \Mul(\B)$ are spectral
sections of $X$.  Then the \emph{difference class} or
\emph{difference element}  $[P-Q]$ (if it exists) is 
defined to be 
$$[P-Q] =_{df} [P-R] - [Q-R] \in K_0(\B)$$
where $R$ is a spectral section of $X$ for which
$$R \leq P \makebox{  and  } R \leq Q$$
(if it exists).  

Note that in the above, $P-R$ and $Q-R$ are necessarily projections in $\B$. 
\end{df}

We leave to the reader the exercise of showing that the difference
class $[P- Q]$ in Definition \ref{df:DifferenceClass} (if it exists)
is well-defined and independent of the choice of $R$, and that $[P-Q]$
is actually equal to the essential codimension $[P:Q]$. A key thing is that
$[P-Q]$ need not exist but the essential codimension $[P:Q]$ always exists. 
See Definition \ref{df:EssentialCodimension}, the beginning of Subsection
\ref{subsection:DZWLPWDefinition}, and the Introduction.
The definition of $[P-Q]$ implies that $\B$ has a nonzero projection, which,
as we have discussed (e.g., see the Introduction), is a strong assumption, and we
can easily find examples of a stable, stably projectionless C*-algebra $\B$ and projections
$P, Q \in \Mul(\B)$ with $P - Q \in \B$ but $[P:Q] \neq 0$ in $K_0(\B)$.
(Moreover, Definition \ref{df:DifferenceClass} (or something similar) will not work,
since $\B$ has no nonzero projections.)

Finally, again, for the DZWLP definition of spectral flow, we do not
here discuss the question of existence, and we do not give examples, since this is meant to be a very quick
summary.

\begin{df} \label{df:DZWLPSpectralFlow}
Let $\B$ be a separable stable C*-algebra, and let $\{ X_t \}_{t \in [0,1]}$ 
be a norm-continous path of self-adjoint Fredholm operators in $\Mul(\B)$.
Suppose that $X_0$ and $X_1$ have spectral sections $P_0$ and $P_1$ 
respectively.
Suppose also that there exists a 
norm-continuous path $\{ Q_t \}_{t \in [0,1]}$ of projections in $\Mul(\B)$
such that for all $t \in [0,1]$, $Q_t$ is a spectral section of $X_t$.

Then we define the DZWLP-spectral flow, if it exists,
 in terms of difference classes:  
$$sf_{DZWLP}(\{ X_t \}_{t \in [0,1]}; P_0, P_1) =_{df} 
[P_0 - Q_0] - [P_1 - Q_1] \in K_0(\B).$$ 

The DZWLP-spectral flow
$sf_{DZWLP}(\{ X_t \}_{t \in [0,1]}; P_0, P_1)$, if it exists, is well-defined
and independent of the choice of $\{ Q_t \}_{t \in [0,1]}$, but we do not
prove this here. 
\end{df}

\end{document}